\documentclass[a4paper,12pt,reqno]{amsart}
\pdfoutput=1
\usepackage[margin=1in]{geometry}
\usepackage{hyperref}
\usepackage{amsmath,mathtools}
\usepackage{enumerate}
\usepackage{tikz}
\usepackage[all]{xy}
\usepackage{verbatim}
\usepackage{setspace}
\usepackage{enumitem}
\usepackage{mathdots}
\usepackage{longtable}
\usepackage{amssymb}
\usetikzlibrary{patterns}
\usetikzlibrary{calc}

\xyoption{rotate}

\DeclareMathOperator{\Hom}{Hom}

\DeclareMathOperator{\End}{End}

\DeclareMathOperator{\Ext}{Ext}

\DeclareMathOperator{\im}{Im}
\DeclareMathOperator{\Ker}{Ker}

\DeclareMathOperator{\sgn}{sgn}
\DeclareMathOperator{\tp}{top}
\DeclareMathOperator{\rad}{rad}

\DeclareMathOperator{\op}{op}
\DeclareMathOperator{\mod*}{mod}

\DeclareMathOperator{\Coker}{Coker}

\DeclareMathOperator{\dimvect}{\mathbf{dim}}
\DeclareMathOperator{\id}{id}

\DeclareMathOperator{\proj}{proj}
\DeclareMathOperator{\add}{add}

\DeclareMathOperator{\Ob}{Ob}

\newcommand{\dynA}{{A}}

\newcommand{\dynD}{{D}}
\newcommand{\dynE}{{E}}
\newcommand{\dynF}{{F}}
\newcommand{\dynG}{{G}}

\newcommand{\coxA}{\dynA}

\newcommand{\coxD}{\dynD}
\newcommand{\coxE}{\dynE}
\newcommand{\coxF}{\dynF}

\newcommand{\coxI}{{I}}
\newcommand{\coxH}{{H}}

\newcommand{\gend}{\Delta}	
\newcommand{\ccv}{\mathbf{c}}
\newcommand{\cv}{c}
\newcommand{\ggv}{\mathbf{g}}
\newcommand{\gv}{g}
\newcommand{\rootsys}{\Phi}
\newcommand{\posroots}{\rootsys_+}

\newcommand{\der}{\mathcal{D}}
\newcommand{\bder}{\der^b}

\newcommand{\catH}{\mathcal{H}}
\newcommand{\clus}{\mathcal{C}}
\newcommand{\gen}{\mathcal{G}}

\newcommand{\syz}{\Omega}
\newcommand{\calM}{\mathcal{M}}
\newcommand{\sus}{\Sigma}
\newcommand{\dgart}{\tau_{\der_\gend}}

\newcommand{\dfart}{\tau_{\der_F}}
\newcommand{\cfart}{\tau_{\clus_F}}
\newcommand{\vw}{\kappa} 
\newcommand{\aw}{\xi} 

\newcommand{\ivx}{\phi}
\newcommand{\dimproj}{\delta}
\newcommand{\derdim}{\wh\dimproj}
\newcommand{\tree}{\mathbb{T}}
\newcommand{\gratio}{\varphi}

\newcommand{\croot}{\psi}

\newcommand{\chebsr}{\chi_+}
\newcommand{\chebr}{\chi}
\newcommand{\imunit}{\mathrm{i}}
\newcommand{\e}{\mathrm{e}}
\newcommand{\field}{K}
\newcommand{\real}{\mathbb{R}}
\newcommand{\integer}{\mathbb{Z}}

\newcommand{\nnint}{\mathbb{Z}_{\geq 0}}
\newcommand{\goldint}{\mathbb{Z}[\gratio]}
\newcommand{\goldsr}{\nnint[\gratio]}
\newcommand{\srhom}{\sigma_+}
\newcommand{\inv}{^{-1}}
\newcommand{\wh}{\widehat}

\newcommand{\ps}[1]{^{(#1)}}

\def\wt{\widetilde}
\def\h{\widehat}
\newcommand{\tikznode}[2]{
	\ifmmode
		\tikz[remember picture,baseline=(#1.base),inner sep=0pt] \node (#1) {$#2$};
	\else
		\tikz[remember picture,baseline=(#1.base),inner sep=0pt] \node (#1) {#2};
	\fi
}

\theoremstyle{plain}
\newtheorem{thm}{Theorem}[section]
\newtheorem*{thm*}{Theorem}
\newtheorem{lem}[thm]{Lemma}
\newtheorem{cor}[thm]{Corollary}
\newtheorem{prop}[thm]{Proposition}
\newtheorem*{prop*}{Proposition}

\theoremstyle{definition}
\newtheorem{defn}[thm]{Definition}

\newtheorem{exam}[thm]{Example}

\newtheorem*{exam*}{Example}

\theoremstyle{remark}
\newtheorem{rem}[thm]{Remark}

\numberwithin{equation}{section}
\numberwithin{figure}{section}
\allowdisplaybreaks

\begin{document}
	\title[Categorifications of Non-Integer Quivers: Types $H_4$, $H_3$ and $I_2(2n+1)$]{Categorifications of Non-Integer Quivers:\\ Types $H_4$, $H_3$ and $I_2(2n+1)$}
	\author{Drew Damien Duffield and Pavel Tumarkin}
       \address{Department of Mathematical Sciences, Durham University, Mathematical Sciences \& Computer Science Building, Upper Mountjoy Campus, Stockton Road, Durham, DH1 3LE, UK}
        \email{drew.d.duffield@durham.ac.uk, pavel.tumarkin@durham.ac.uk}
        \thanks{Research was supported by the Leverhulme Trust research grant RPG-2019-153.}
	\maketitle
	\begin{abstract}
		We define the notion of a weighted unfolding of quivers with real weights, and use this to provide a categorification of mutations of quivers of finite types $H_4$, $H_3$ and $I_2(2n+1)$. In particular, the (un)folding induces a semiring action on the categories associated to the unfolded quivers of types $E_8$, $D_6$ and $A_{2n}$ respectively. We then define the tropical seed pattern on the folded quivers, which includes $\cv$- and $\gv$-vectors, and show its compatibility with the unfolding.    
        \end{abstract}
        
              
	\setcounter{tocdepth}{1}
	\tableofcontents
	
	\section{Introduction and main results}
Cluster algebras were introduced by Fomin and Zelevinsky~\cite{FominZelevinskyI} in an effort to describe dual canonical bases in the universal enveloping algebras of Borel subalgebras of simple complex Lie algebras. 
A cluster algebra of geometric type is completely defined by an integer skew-symmetrizable $n\times n$ {\em exchange matrix} (or, equivalently, by the corresponding {\em cluster quiver}), which undergoes involutive transformations called {\em mutations}. A cluster algebra possesses a distinguished set of generators, \emph{cluster variables}, organized in groups of the same cardinality called \emph{clusters}. Pairs composed of a cluster and the corresponding exchange matrix are called {\em seeds} which form a combinatorial structure called a {\em seed pattern}. 

A distinguished class of cluster algebras consists of algebras of {\em finite type}, i.e. cluster algebras having only finitely many distinct cluster variables. These algebras were classified by Fomin and Zelevinsky in~\cite{FominZelevinskyII} by establishing a connection with Cartan-Killing classification of simple Lie algebras. More precisely, it was shown in~\cite{FominZelevinskyII} that cluster algebras of finite type are in one-to-one correspondence with finite root systems, and there is a bijection between the set of cluster variables and the set of roots consisting of all positive roots and negative simple roots.

A categorification of cluster algebras associated to quivers of finite type was constructed in a series of papers by Buan, Marsh, Reineke, Reiten, and Todorov~\cite{BMRRT,BMRTilting,BMRMutation} and by Caldero, Chapoton and Schiffler~\cite{CCS2,CCS}. The categorification involves the construction of the \emph{cluster category} and the development of cluster-tilting theory.
Cluster-tilting theory has since inspired exciting new directions in the representation theory of finite-dimensional algebras, prompting further developments in topics concerning exceptional sequences~\cite{BMExceptional,SpeyerThomas}, $\tau$-tilting theory~\cite{TauTilting}, and the wall and chamber structure of finite-dimensional algebras~\cite{BST} (to name but a few). The theory is particularly well-understood in the finite type setting as a result of the formal connection made between cluster algebras of finite type and cluster categories via the Caldero-Chapoton map~\cite{CCMap}, which induces a bijection between the cluster variables of a cluster algebra of finite type and the indecomposable objects of the corresponding cluster category (cf.~\cite{CalderoKeller}).

Cluster algebras of finite type correspond to finite root systems and thus to finite crystallographic Coxeter groups. In this paper, we consider cluster structures defined from finite non-crystallographic Coxeter groups and their associated root systems (in the sense of Deodhar~\cite{Deodhar}).
The main goal is to extend the cluster combinatorics and categorification to the exchange matrices and quivers of types $\coxI_2(n)$, $\coxH_3$ and $\coxH_4$. 

Mutations of matrices with real entries were considered e.g. in~\cite{Lampe,FeliksonTumarkin19}. In contrast to the usual integer framework, the seed pattern defined as in the integer setting lacks some basic good properties. We show that the tropical degeneration of the seed pattern can nevertheless be consistently defined, and mutations of quivers of types  $\coxI_2(2n+1)$, $\coxH_3$ and $\coxH_4$ can be categorified in line with the integer counterparts. There are additional technical considerations and slight differences in theory for the $\coxI_2(2n)$ cases which arise from the existence of short and long roots. As such, categorification of quivers of types $\coxI_2(2n)$ are addressed in a separate paper (\cite{DT2}).

One of our main tools is a {\em weighted (un)folding} of a quiver (see Section~\ref{section-unf} for the details), the application of which follows the projection of root systems of types $\coxE_8$, $\coxD_6$ and $\coxA_4$ developed in~\cite{Lusztig,WaveFrontsReflGroups,MoodyPatera}. In particular, the (un)foldings give rise to projections of dimension vectors of objects in module categories of integer quivers to the roots associated to folded quivers (Section~\ref{section-proj}). We obtain the following theorem.

\begin{thm}[Theorem~\ref{thm:Folding},Proposition~\ref{prop:DerivedProj}]
Let $F\colon Q^{\gend} \rightarrow Q^{\gend'}$ be a weighted folding of quivers, where $Q^{\gend}$ is a quiver of type $\gend \in \{\coxA_{2n},\coxD_6,\coxE_8\}$ and $Q^{\gend'}$ is a quiver of type $\gend'\in\{\coxI_2(2n+1),\coxH_3,\coxH_4\}$. Then $F$ determines a weighting on the rows of the Auslander-Reiten quiver of $\mod* KQ^{\gend}$ such that the following hold:
	\begin{enumerate}[label=(\alph*)]
		\item The dimension vectors of the modules in the rows with weight 1 project onto the positive roots of $\gend'$.
		\item Let $\mathcal{R}_1$ and $\mathcal{R}_w$ be rows (with respective weights 1 and $w$) that correspond to the same vertex of the folded quiver. Then the projected dimension vector $v_w$ of a module $M_w \in \mathcal{R}_w$ is such that $v_w = w v_1$, where $v_1$ is the projected dimension vector of the module $M_1 \in \mathcal{R}_1$ in the same column as $M_w$.
	\end{enumerate}
	Moreover, these results naturally extend to the bounded derived category $\bder(\mod*KQ)$. 
  \end{thm}

  This, in turn, gives rise to a semiring action on the module, derived and cluster categories of the integer quiver, where the semiring (which we denote by $R_+$) is defined from the weights of arrows of the folded quiver (Section~\ref{section-action}). The main result here can be formulated as follows.

  \begin{thm}[Theorems~\ref{thm:Action},~\ref{thm:Generators}]
     Let $F,\gend$ and $\gend'$ be as above. Then $\mod*KQ^\gend$ has an action of the semiring $R_+$. The collection of indecomposables of $\mod*KQ^\gend$ whose projected dimension vector is a root of $\gend'$ forms a minimal set of $R_+$-generators for $\mod*KQ^\gend$. These results naturally extend to the bounded derived category $\bder(\mod*KQ)$.
      \end{thm}

   The semiring action also applies to the cluster category of $Q^\gend$, and this allows us to extend the results of~\cite{BMRRT} to categorify mutations of folded quivers by considering distinguished objects with respect to the action, which we call $R_+$-{\em tilting objects}  (Section~\ref{section-mutations}). We prove the following.

   \begin{thm}[Theorems~\ref{thm:TiltingEquiv},~\ref{thm:Length},~\ref{thm:Complements},~\ref{thm:TiltedRPAction},~\ref{thm:TiltedFolding}]
     Let $F\colon Q^\gend \rightarrow Q^{\gend'}$ be a folding of quivers with $\gend' \in \{\coxH_3,\coxH_4, \coxI_2(2n+1)\}$. Then the following holds for the category cluster category $\clus_F$ of $Q^\gend$:
     \begin{enumerate}[label=(\alph*)]
       \item
       Every basic $R_+$-tilting object $T \in \clus_F$ (injectively) corresponds to a basic tilting object $\wh{T} \in \clus_F$.
     \item
       If $T$ is a basic $R_+$-tilting object, then $T$ has $|Q_0^{\gend'}|$ indecomposable direct summands.
     \item
       If $T$ is an almost complete $R_+$-tilting object, then it has exactly two complements.
     \item
       If $T$ is a basic $R_+$-tilting object, and $A_T$ is the cluster-tilted algebra corresponding to $\wh{T}$, then there exists an $R_+$-action on $\mod*A_T$.
     \item
       Let $T$ be an almost complete $R_+$-tilting object with complements $X_1$ and $X_2$, let $T_1=T\oplus X_1$ and $T_2=T\oplus X_2$, and let $A_1$, $A_2$ be cluster-tilted algebras corresponding to $\wh T_1$ and $\wh T_2$. Then the (non-integer) folded quivers associated to $A_1$ and $A_2$ differ by a single mutation.
\end{enumerate}
     \end{thm}
     We also note that for the two complements in (c) one obtains triangles between indecomposables that are $R_+$-generated by the complements (see Corollary~\ref{thm:RPApproximations}).

     Finally, we define the tropical seed patterns of folded quivers. We define $\cv$-vectors and $C$-matrices as in the integer case, and we consider two definitions of $\gv$-vectors (or $G$-matrices). The first definition goes along the results of~\cite{NZ} by defining $G$-matrix as the inverse of the transposed $C$-matrix. The second way to define $\gv$-vectors is to apply the projection with respect to the folding to $\gv$-vectors (corresponding to vertices with weight $1$) of the unfolded quiver. We then prove the following (see Section~\ref{section-cg} for details).

     \begin{thm}[Theorem~\ref{thm:Cube}, Lemma~\ref{lem:Roots}, Corollary~\ref{thm:SignCoherent}]
\label{thm}
       Let $F\colon Q^\gend \rightarrow Q^{\gend'}$ be a folding of quivers with $\gend' \in \{\coxH_3,\coxH_4, \coxI_2(2n+1)\}$. Then $\cv$-vectors of  $Q^{\gend'}$ are roots of $\gend'$ (and thus are sign-coherent), and the two definitions of $\gv$-vectors of $Q^{\gend'}$ coincide. Further, $C$-matrices of  $Q^{\gend'}$ can also be obtained by projection of $C$-matrices of  $Q^{\gend}$.
      \end{thm}

      We note that Theorem~\ref{thm} provides a categorical interpretation of $\gv$-vectors and $G$-matrices of $Q^{\gend'}$, which we describe in Section~\ref{section-cat-g}. Theorem~\ref{thm} can also be used to provide an explicit construction of non-crystallographic generalized associahedra, see~\cite{FTY}.  

\medskip

      The paper is organised as follows. In Section~\ref{section-basics} we recall the basics about mutations of quivers and exchange matrices. In Section~\ref{section-unf}, we describe the projections of root systems we will use throughout the paper, remind the classical definition of unfolding, and define a generalisation called weighted unfolding. Section~\ref{cheb} is devoted to the basic properties of Chebyshev polynomials of second kind. In Section~\ref{section-proj}, we construct the projections of module and bounded derived categories induced by the (un)foldings of quivers. Section~\ref{section-action} is devoted to the description of the semiring action on the module and bounded derived categories. In Section~\ref{section-mutations}, we extend the semiring action to the cluster categories of the unfolded quivers, thus providing a categorification of mutations of the folded quivers. Finally, Section~\ref{section-cg} is devoted to the construction of the tropical seed pattern, and to the compatibility of the projections and mutations. 

      \subsection*{Acknowledgements}
 We would like to thank Anna Felikson and Pierre-Guy Plamondon for helpful discussions.  A substantial part of the paper was written at the Isaac Newton Institute for Mathematical Sciences, Cambridge; we are grateful to the organisers of the programme “Cluster algebras and representation theory”, and to the Institute for support and hospitality during the programme; this work was supported by EPSRC grant no EP/R014604/1.    

	\section{Exchange matrices, quivers and their mutations}
\label{section-basics}
In the theory of cluster algebras and mutations of quivers, one classically has an integer skew-symmetric (or more generally, skew-symmetrizable) matrix, called the \emph{exchange matrix}. The exchange matrix can equivalently be viewed as an integer-weighted (valued) quiver. One can then perform \emph{mutations} on the exchange matrix (or quiver) to obtain a new exchange matrix (resp. quiver). It is easy to see that exactly the same mutation rule can be applied to any skew-symmetrizable matrix with entries from any totally ordered ring. Throughout this section, we will assume that $R$ is a totally ordered ring. We also do not consider \emph{`frozen vertices'} in our quiver, and thus all exchange matrices in this paper are assumed to be square ones. 

Throughout the paper, $Q$ is a quiver with vertex set $Q_0$, arrow set $Q_1$.


\begin{defn}
 By an \emph{exchange matrix over $R$} we mean a  skew-symmetric matrix with entries in $R$. We will abuse notation by using the notion {\em exchange matrix} instead if there is no ambiguity what the ring $R$ is. 
\end{defn}

\begin{defn}
	Let $Q_0$ be a set of indices and let $B=(b_{ij})_{i,j\in Q_0}$ be an exchange matrix over $R$. A \emph{mutation of $B$ at an index $k \in Q_0$} is an exchange matrix $\mu_k(B)=(b'_{ij})_{i,j \in Q_0}$ whose entries are given by the \emph{mutation formula}
	\begin{equation*}
		b'_{ij}=
		\begin{cases}
			-b_{ij}													&	\text{if } i=k \text{ or } j=k \\
			b_{ij}+\frac{|b_{ik}|b_{kj} + b_{ik}|b_{kj}|}{2}	& \text{otherwise}.
		\end{cases}
	\end{equation*}
\end{defn}

\begin{rem}
	The mutation formula above is precisely the same formula used in the classical setting of mutations of integer exchange matrices~\cite{FominZelevinskyI}.
\end{rem}

\begin{defn}
	An \emph{$R$-quiver} is a quiver $Q$ that has at most one arrow between distinct vertices, and each such arrow has a strictly positive $R$-weight. Specifically, $Q$ is a tuple $(Q_0,Q_1,\aw)$ such that the following holds:
	\begin{enumerate}[label=(R\arabic*)]
		\item $(Q_0,Q_1)$ is a quiver without loops and 2-cycles.
		\item For any $a\colon i \rightarrow j$ and $a'\colon i' \rightarrow j'$ in $Q_1$ such that $a \neq a'$, we have $i \neq i'$ or $j \neq j'$.
		\item $\aw\colon Q_1 \rightarrow R_{>0}$ is a function mapping each arrow to its strictly positive weight.
	\end{enumerate}
\end{defn}

Exchange matrices over $R$ are in bijective correspondence with $R$-quivers. Given an exchange matrix $B=(b_{ij})_{i,j\in Q_0}$, one obtains an $R$-quiver $Q^B$ in the following way. The set of vertices $Q_0^B$ of $Q^B$ is precisely the index set $Q_0$ of $B$. There exists an arrow $a\colon i \rightarrow j$ in $Q_1^B$ if and only if $b_{ij} > 0_R$. Moreover, each arrow $a\colon i \rightarrow j$ in $Q_1^B$ has weight $\aw(a)=b_{ij}$. 

\begin{defn}
	Let $Q^B$ be an $R$-quiver and let $B$ be its corresponding exchange matrix. A mutation of $Q^B$ at the vertex $k \in Q_0^B$ is the $R$-quiver $Q^{\mu_k(B)}$ corresponding to $\mu_k(B)$.
\end{defn}

Since mutation is an involution, one can define {\em mutation classes} as sets of quivers that can be obtained from each other by iterated mutations. A quiver is called {\em mutation-finite} if its mutation class is finite.    

\section{Folding and unfolding}
\label{section-unf}
In this section, we define weighted foldings and weighted unfoldings of skew-symmetrizable matrices, and then apply it to quivers constructed from root systems of finite non-crystallographic Coxeter groups. 

\subsection{Unfolding a skew-symmetrizable matrix} \label{unf}
Here we define a notion of a \emph{weighted unfolding} of a skew-symmetrizable matrix, and show how it relates to the notion of unfolding due to A. Zelevinsky (see \cite{FST2} and \cite{FST3} for details).

\subsubsection{Rescaling}
\label{rescaling}
We will first need the definition of a \emph{rescaling} of a skew-symmetrizable matrix introduced by N.~Reading in~\cite{Reading}.

\begin{defn}
  \label{rescaling-def}
  Let $B$ and $\wt B$ be skew-symmetrizable matrices. Then $\wt B$ is a \emph{rescaling} of $B$ if there exists a diagonal matrix $P$ with positive entries $p_i$ such that $\wt B=P^{-1}BP$. In particular, the $ij-$entry of $\wt B$ is $p_j/p_i$ times the $ij$-entry of $B$.
\end{defn}

\begin{exam} \label{rescaling-ex}
	 The matrix $\wt B$ below is a rescaling of the matrix $B$. Here the rescaling matrix $P$ has diagonal entries $(1,1,1/\sqrt{2},1/\sqrt{2})$.
	\begin{equation*}
		\wt B=P^{-1}BP=\begin{pmatrix}
			0&-1&0&0\\
			1&0&-1&0\\
			0&2&0&-1\\
			0&0&1&0
		\end{pmatrix}
		\qquad\qquad
		B=\begin{pmatrix}
			0&-1&0&0\\
			1&0&-\sqrt{2}&0\\
			0&\sqrt{2}&0&-1\\
			0&0&1&0
		\end{pmatrix}
	\end{equation*}
\end{exam}

\begin{rem}
Every skew-symmetrizable matrix has a skew-symmetric rescaling: if $B$ is skew-symmetrizable and $D$ is a diagonal matrix such that $BD$ is skew-symmetric, then it is easy to check that the matrix $\sqrt{D^{-1}}B\sqrt{D}$ is skew-symmetric.
  \end{rem}

It was shown in~\cite{Reading} that rescaling commutes with mutations. Therefore, we can associate any skew-symmetrizable exchange matrix with its skew-symmetric rescaling, and thus with the corresponding quiver (with real weights).

\subsubsection{Weighted unfoldings}
We will now give the definition of a weighted unfolding. Throughout this subsection, we let $S$ be an $m \times m$ skew-symmetric matrix and $B$ be an indecomposable $n \times n$ skew-symmetrizable matrix with real entries.

\begin{defn} \label{defn:Unfolding}
	Suppose that there exist disjoint index sets $E_1,\ldots,E_n$ of the rows/columns of $S$ such that $m=\sum_{i=1}^n |E_i|$. We say that the pair of matrices $(B,S)$ is an \emph{origami pair} if
	\begin{enumerate}[label=(\arabic*)]
		\item the sum of entries in each column of each $E_i \times E_j$ block of $S$ equals $b_{ij}$;
		\item if $b_{ij} \geq 0$ then the $E_i \times E_j$ block of $S$ has all entries non-negative.
	\end{enumerate}
	
	Given an origami pair $(B,S)$, we define a \emph{composite mutation} $\h\mu_i = \prod_{\hat\imath \in E_i} \mu_{\hat\imath}$ on $S$. This mutation is well-defined, since all the mutations $\mu_{\hat\imath}$, $\hat\imath\in E_i$, for given $i$ commute (as it follows from (1) and (2) that every $E_i\times E_i$ block is a zero submatrix).
	
	Suppose that there exists some rescaling $P\inv B P$ of $B$, and a diagonal $m \times m$ matrix $W = (w_i)$ with positive entries (called \emph{weights}) such that the pair $(P\inv B P, WSW\inv)$ is origami. We say that $S$ is a \emph{weighted unfolding} of $B$ \emph{with weights $w_i$} if for any sequence of iterated mutations $\mu_{k_1}\ldots\mu_{k_l}$ of $B$, the pair
	\begin{equation*}
		(P^{-1}\mu_{k_1}\dots\mu_{k_l}(B)P, W\h\mu_{k_1}\dots\h\mu_{k_l}(S)W^{-1})
	\end{equation*}
	is also origami.
\end{defn}

\begin{rem}
	Note that the elements in any given index set $E_i$ need not be indices of consecutive rows/columns. See Examples~\ref{I5} and \ref{G2}.
\end{rem}

From Definition~\ref{defn:Unfolding}, one can recover the classical notion of unfolding an indecomposable $n \times n$ integer skew-symmetrizable matrix $B$ (due to A. Zelevinsky) in the following way. First, let $BD$ be a skew-symmetric matrix, where $D=(d_{i})$ is a diagonal integer matrix with positive diagonal entries. 
Then an $m \times m$ skew-symmetric integer matrix $S$ is an (integer) unfolding of $B$ if $S$ is a weighted unfolding of $B$ with all weights equal to 1, and if the index sets are such that $|E_i|=d_i$.

\begin{exam} \label{unf-ex}
	The matrix $S$ below is an integer unfolding of the matrix $B$. Here $d_1=d_2=1$, $d_3=d_4=2$, $E_1=\{1\}$, $E_2=\{2\}$, $E_3=\{3,4\}$, $E_2=\{5,6\}$.
	\begin{equation*}
		B=\begin{pmatrix}
			0&-1&0&0\\
			1&0&-1&0\\
			0&2&0&-1\\
			0&0&1&0
		\end{pmatrix}
		\qquad\qquad
		S=\begin{pmatrix}
			0&-1&0&0&0&0\\
			1&0&-1&-1&0&0\\
			0&1&0&0&-1&0\\
			0&1&0&0&0&-1\\
			0&0&1&0&0&0\\
			0&0&0&1&0&0
		\end{pmatrix}
	\end{equation*}
	The matrix $B$ defines a cluster algebra of type $\coxF_4$, and $S$ is of type $\coxE_6$. The matrix $S$ is also an unfolding of the rescaled matrix $B$ from Example~\ref{rescaling-ex}.
\end{exam}

\begin{exam} \label{I5}
	Denote $\gratio=2\cos(\pi/5)$, and consider the following two matrices
	\begin{equation*}
		B=\begin{pmatrix}
			0&-\gratio\\
			\gratio&0
		\end{pmatrix}
		\qquad\qquad
		S=\begin{pmatrix}
			0&-1&0&0\\
			1&0&1&0\\
			0&-1&0&-1\\
			0&0&1&0
		\end{pmatrix}
	\end{equation*}
	It is easy to check that $S$ is a weighted unfolding of $B$ with weights $(1,\gratio,\gratio,1)$, where $E_1=\{1,3\}$, $E_2=\{2,4\}$ and $P=\mathrm{Id}$. We can observe that the matrix $S$ is of type $\coxA_4$, and it is natural to say that the matrix $B$ is of type $H_2=\coxI_2(5)$.
\end{exam}

\begin{exam} \label{G2}
	Consider the following two matrices
	\begin{equation*}
		B=\begin{pmatrix}
			0&-\sqrt{3}\\
			\sqrt{3}&0
		\end{pmatrix}
		\qquad\qquad
		S=\begin{pmatrix}
			0&-1&0&0&0\\
			1&0&1&0&0\\
			0&-1&0&-1&0\\
			0&0&1&0&1\\
			0&0&0&-1&0
		\end{pmatrix}
	\end{equation*}
	Again, it is easy to see that $S$ is a weighted unfolding of $B$ with weights $(1,\sqrt{3},2,\sqrt{3},1)$, where $E_1=\{2,4\}$, $E_2=\{1,3,5\}$ and $P=\mathrm{Id}$. Here, $S$ is of type $\coxA_5$, and $B$ is of type $\coxI_2(6)$ (which is a rescaling of the Dynkin type $\dynG_2$).
\end{exam}

Henceforth, we will abuse notation by omitting the word ``weighted'' throughout the paper and specify the weights instead.

Examples~\ref{I5} and~\ref{G2} fit into a series of unfoldings: for every dihedral group $\coxI_2(k+1)$ there is an unfolding of the corresponding matrix to $\coxA_{k}$ with weights
\begin{equation*}
	w_i=U_{i-1}\left(\cos\frac{\pi}{k+1}\right)
\end{equation*}
and $P=\mathrm{Id}$, where $U_i$ is the $i$-th Chebyshev polynomial of the second kind (see Section~\ref{cheb}).

Example~\ref{I5} also gives rise to unfoldings of matrices corresponding to groups $\coxH_3$ and $\coxH_4$ to $\coxD_6$ and $\coxE_8$ respectively, see Fig.~\ref{H} for the weights (the unfolding blocks are composed of two vertices each, given by vertices in the same column).

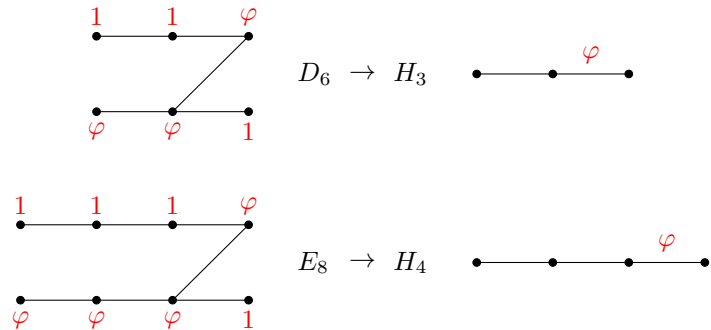
\begin{figure}[t]
	\begin{tikzpicture}
	\draw [fill=black](3,2.5) ellipse (0.05 and 0.05);
	\draw [fill=black](4,2.5) ellipse (0.05 and 0.05);
	\draw [fill=black](5,2.5) ellipse (0.05 and 0.05);
	\draw [fill=black](3,1.5) ellipse (0.05 and 0.05);
	\draw [fill=black](4,1.5) ellipse (0.05 and 0.05);
	\draw [fill=black](5,1.5) ellipse (0.05 and 0.05);
	\draw (3,2.5) -- (5,2.5) -- (4,1.5) -- (5,1.5);
	\draw (3,1.5) -- (4,1.5);
	\draw [red,anchor=south] (3,2.5) node {\footnotesize$1$};
	\draw [red,anchor=south] (4,2.5) node {\footnotesize$1$};
	\draw [red,anchor=south] (5,2.5) node {\footnotesize$\varphi$};
	\draw [red,anchor=north] (3,1.5) node {\footnotesize$\varphi$};
	\draw [red,anchor=north] (4,1.5) node {\footnotesize$\varphi$};
	\draw [red,anchor=north] (5,1.5) node {\footnotesize$1$};
	
	\draw [anchor=west] (5.5,2) node {\footnotesize$D_6$};
	\draw (6.5,2) node {\footnotesize$\rightarrow$};
	\draw [anchor=east] (7.5,2) node {\footnotesize$H_3$};
	
	\draw [fill=black](8,2) ellipse (0.05 and 0.05);
	\draw [fill=black](9,2) ellipse (0.05 and 0.05);
	\draw [fill=black](10,2) ellipse (0.05 and 0.05);
	\draw (8,2) -- (10,2);
	\draw[red,anchor=south] (9.5,2) node {\footnotesize$\gratio$};
	
	\draw [fill=black](2,0) ellipse (0.05 and 0.05);
	\draw [fill=black](3,0) ellipse (0.05 and 0.05);
	\draw [fill=black](4,0) ellipse (0.05 and 0.05);
	\draw [fill=black](5,0) ellipse (0.05 and 0.05);
	\draw [fill=black](2,-1) ellipse (0.05 and 0.05);
	\draw [fill=black](3,-1) ellipse (0.05 and 0.05);
	\draw [fill=black](4,-1) ellipse (0.05 and 0.05);
	\draw [fill=black](5,-1) ellipse (0.05 and 0.05);
	\draw (2,0) -- (5,0) -- (4,-1) -- (5,-1);
	\draw (2,-1) -- (4,-1);
	\draw [red,anchor=south] (2,0) node {\footnotesize$1$};
	\draw [red,anchor=south] (3,0) node {\footnotesize$1$};
	\draw [red,anchor=south] (4,0) node {\footnotesize$1$};
	\draw [red,anchor=south] (5,0) node {\footnotesize$\varphi$};
	\draw [red,anchor=north] (2,-1) node {\footnotesize$\varphi$};
	\draw [red,anchor=north] (3,-1) node {\footnotesize$\varphi$};
	\draw [red,anchor=north] (4,-1) node {\footnotesize$\varphi$};
	\draw [red,anchor=north] (5,-1) node {\footnotesize$1$};
	
	\draw [anchor=west] (5.5,-0.5) node {\footnotesize$E_8$};
	\draw (6.5,-0.5) node {\footnotesize$\rightarrow$};
	\draw [anchor=east] (7.5,-0.5) node {\footnotesize$H_4$};
	
	\draw [fill=black](8,-0.5) ellipse (0.05 and 0.05);
	\draw [fill=black](9,-0.5) ellipse (0.05 and 0.05);
	\draw [fill=black](10,-0.5) ellipse (0.05 and 0.05);
	\draw [fill=black](11,-0.5) ellipse (0.05 and 0.05);
	\draw (8,-0.5) -- (11,-0.5);
	\draw[red, anchor = south] (10.5,-0.5) node {\footnotesize$\gratio$};
\end{tikzpicture}
	\caption{Foldings of quivers $Q^{\coxD_6} \rightarrow Q^{\coxH_3}$ and $Q^{\coxE_8} \rightarrow Q^{\coxH_4}$. Weights are labelled next to the appropriate vertices or arrows. Each collection of edges that appear in the same column in the domain of the folding are either all pointing left or all pointing right, and these arrows are mapped to an arrow in the codomain with the same orientation.} \label{H}
\end{figure}

\subsection{Weighted foldings of quivers}
The unfoldings described above are consistent with foldings/projections of root systems of Coxeter groups. The construction goes back to Lusztig~\cite{Lusztig} who noticed the embedding of the root system of the Coxeter group of type $\coxH_4$ into the root system of the Coxeter group of type $\coxE_8$ in the context of admissible maps of Coxeter groups. An embedding of the root system of the dihedral group $I_2(n)$ into the root systems of the Coxeter group of type $\coxA_{n-1}$ was observed in~\cite{WaveFrontsReflGroups,Muhlherr}. These constructions have later been further developed and extensively used in various contexts (see~\cite{WaveFrontsReflGroups,MoodyPatera,Dyer,Lanini}). 

In this section, we formalise the notion of folding in the context of weighted quivers, of which the projection above features as an example. To formally define foldings onto $R$-quivers, we need the notion of a vertex-weighted quiver.

\begin{defn} \label{defn:RVWQ}
  An \emph{$R$-vertex-weighted quiver} $Q$ is a tuple $(Q_0,Q_1,\vw)$ such that $(Q_0,Q_1)$ is a quiver and $\vw\colon Q_0 \rightarrow R$ is a function mapping each vertex to its weight.
\end{defn}

The above definition allows us to define the dual notion of weighted unfolding of exchange matrices --- the folding of weighted quivers.

\begin{defn} \label{defn:RQfolding}
	Let $Q^B$ be an $R$-quiver and $Q^S$ be an $R$-vertex-weighted quiver. We call a morphism of quivers $F\colon Q^S \rightarrow Q^B$ a \emph{weighted folding of quivers} if the following holds.
	\begin{enumerate}[label=(\roman*)]
		\item $Q^B$ is the $R$-quiver of an $n \times n$ exchange matrix $B$ over $R$,
		\item $Q^S$ is the quiver of an $m \times m$ integer exchange matrix $S$,
		\item $S$ is a weighted unfolding of $B$ with weight matrix $W=(w_i)$ such that each $w_i \in R$, and the blocks of $S$ are given by index sets $E_1,\ldots, E_n$,
		\item $\vw(i)=w_i$ for each $i \in Q_0^S$,
		\item $F$ is a surjective morphism such that for any $j \in Q_0^B$, we have $F(i) = j$ for all $i \in E_j$.
	\end{enumerate}
	We call $Q^B$ a \emph{(weighted) folding} of $Q^S$ if there exists a folding of weighted quivers $F\colon Q^S \rightarrow Q^B$.
\end{defn}

Henceforth, for each $i \in Q^S_0$, we will write $[i] \in Q^B_0$ as the vertex such that $F(i) = [i]$, and thus, $[i]=[j]$ for any $i,j \in Q^S_0$ such that $F(i)=F(j)$. The exchange $\integer$-matrix $S$ then has the structure of a block matrix $(S_{[i][j]})_{[i],[j] \in Q^B_0}$. Given a simply-laced Dynkin diagram of type $\gend$, we mean by $Q^\gend$ a quiver whose underlying graph is of shape $\gend$. Similarly, given a Coxeter diagram $\gend'$, we mean by $Q^{\gend'}$ an $R$-quiver whose underlying graph is of shape $\gend'$ and whose arrows are weighted by $\aw(a)=2\cos\theta_{\alpha}$ for each $a \in Q^{\gend'}_1$, where $\theta_{\alpha}$ is the dihedral angle of the edge $\alpha \in \gend'$ that corresponds to the arrow $a \in Q^{\gend'}_1$.

\begin{exam}
	Figure~\ref{H} shows the foldings onto $\coxH$-type quivers. In terms of projections of root systems, the simple roots of the root system $\coxE_8$ with prescribed weights $1$ are mapped to simple roots of the root system $\coxH_4$, and the simple roots from the same unfolding blocks with weights $\gratio$ are mapped to $\gratio$-multiples of the corresponding  simple roots of $\coxH_4$. Extending this map by linearity results in a bijection between the roots of $\coxE_8$ and roots of $\coxH_4$ together with their $\gratio$-multiples. The projection of $\coxD_6$ onto $\coxH_3$ can be obtained by considering the natural embedding of the root system of type $\coxD_6$ into $\coxE_8$.
	
	The construction can also be extended to the projections of $\coxA_k$ onto $\coxI_{2}(k+1)$ by mapping each simple root from one of the unfolding blocks to $w_i\e^{\imunit k\pi/(k+1)}$, and each simple root from the other blocks to $w_i$. If $k$ is even, then extending this by linearity results in a bijective map to the union $\bigcup_{i=0}^{(k-2)/2}w _i \rootsys$, where $\rootsys$ is the root system of $\coxI_{2}(k+1)$ (cf. Lemma~\ref{lem:IRootsOfUnity}).

        If $k$ is odd, then the projection is a bit more subtle. Extending the map above by linearity, we obtain $\bigcup\limits_{i+j\,\text{even}}w _{i}\e^{\imunit \pi j/(k+1)}$. If we consider the roots in the root system of type $I_2(k+1)$ to be of two distinct lengths $1$ (short roots) and $2\cos(\pi/(k+1))$ (long roots), then the result of the projection can be rewritten as the union of $\bigcup\limits_{0\le i\le(k-1)/4}w_{2i} \rootsys_s$ and $\bigcup\limits_{0\le i\le(k-3)/4}(w_{2i+1}/w_1) \rootsys_l$, where  $\rootsys_s$ and $\rootsys_l$ denote the set of short and long roots respectively.
      \end{exam}

      \begin{rem} \label{rem:IFolding}
	For (un)foldings of type $\coxI_2(k+1)$, we will focus mainly on the case where $k$ is even, and we will investigate the case where $k$ is odd in a forthcoming paper. Thus for this paper, we will work with foldings $F\colon Q^{\coxA_{2n}} \rightarrow Q^{\coxI_{2}(2n+1)}$, where $n\geq 2$ and $Q^{\coxA_{2n}}$ is the bipartite quiver
	\begin{equation*}
		\xymatrix@1{0 \ar[r] & 1 & \ar[l] 2 \ar[r] & 3 & \ar[l] \cdots \ar[r] & 2n-1}
	\end{equation*}
	(or its opposite), with $\vw(i) = U_i(\cos\frac{\pi}{2n+1})$ for $i\in Q^{\coxA_{2n}}_0$. On the other hand, $Q^{\coxI_2(2n+1)}$ is the $\integer\left[2\cos\frac{\pi}{2n+1}\right]$-quiver 
	\begin{equation*}
		\xymatrix@1{[0] \ar[rr]^-{2 \cos \frac{\pi}{2n+1}} && [1]}
	\end{equation*}
	(or its opposite). In particular, $F(i)=F(j)$ if and only if $i$ and $j$ are either both even or both odd.
\end{rem}

\section{Chebyshev polynomials of the second kind}
\label{cheb}
As highlighted from the last section, Chebyshev polynomials of the second kind are central to the paper. Thus, we will briefly review them here and define some technical tools based on their properties.

\subsection{Definitions and basic properties}

The $k$-th Chebyshev polynomial of the second kind is the polynomial $U_k(x)$ that satisfies the relation
\begin{equation*}
	U_{k}(\cos \alpha) \sin \alpha = \sin (k+1) \alpha.
\end{equation*}
It is a well-known fact that the polynomial $U_{k}(x)$ is a degree $k$ polynomial with integer coefficients. In the case where $x = \cos \frac{\pi}{n+1}$, the Chebyshev polynomials of the second kind satisfy nice symmetry properties, which are also well-known.

\begin{lem} \label{lem:ChebProps}
	Define a sequence $(\theta_k)_{k \in \nnint}$ by $\theta_k=U_k(\cos \frac{\pi}{n+1})$. Then
	\begin{enumerate}[label=(\alph*)]
		\item $\theta_1 = 2 \cos \frac{\pi}{n+1}$,
		\item $\theta_n=0$,
		\item $\theta_k=\theta_{n-1-k}$,
		\item $\theta_{n+k}=-\theta_{n-k}$ for $k\leq n$,
		\item $\theta_k \theta_l = \sum_{j=0}^l \theta_{k-l+2j}$ for $k\geq l$.
		\item $\theta_k > 1$ for $0 < k < n-1$.
	\end{enumerate}
\end{lem}
\begin{proof}
	All statements except for (e) immediately follow from trigonometric identities.
%
%
	The statement (e) follows from evaluating the imaginary part of the geometric series 
	\begin{align*}
		\sum_{j=0}^{l}\frac{\e^{\frac{\imunit(k-l+2j + 1)\pi}{n+1}}}{\sin\frac{\pi}{n+1}} &= \frac{\left(1-\e^{\frac{\imunit(2 l+1)\pi}{n+1}}\right) \e^{\frac{\imunit(k-l+1)\pi}{n+1}}}{\left(1-\e^{\frac{2\imunit \pi}{n+1}}\right) \sin\frac{\pi}{n+1}}
	=\frac{\left(\e^{\frac{\imunit(l+1)\pi}{n+1}} - \e^{\frac{-\imunit(l+1)\pi}{n+1}}\right) \e^{\frac{\imunit(k+1)\pi}{n+1}}}{\left(\e^{\frac{\imunit \pi}{n+1}} - \e^{\frac{-\imunit \pi}{n+1}}\right) \sin\frac{\pi}{n+1}} \\
	&= \frac{\sin\frac{(l+1)\pi}{n+1}}{\sin\frac{\pi}{n+1}} \frac{\e^{\frac{\imunit(k+1)\pi}{n+1}}}{\sin\frac{\pi}{n+1}}.
	\end{align*}
	
\end{proof}

Motivated by the properties above, we will define a class of rings and semirings that will be used extensively throughout the paper.

\begin{defn} \label{def:ChebSemiring}
	Let $n \geq 2$. Then we define the following family of commutative rings and semirings by
	\begin{align*}
          \chebr\ps{2n+1} &= \integer[\croot_1,\ldots,\croot_{n-1}] \\
          \chebsr\ps{2n+1} &= \nnint[\croot_1,\ldots,\croot_{n-1}] \\
	\wh\chebr\ps{2n+1} &= \integer\left[2 \cos \frac{\pi}{2n+1}\right]
	\end{align*}
	such that
	\begin{equation*}
		\croot_k \croot_l = \croot_l \croot_k = \sum_{j=0}^l \croot_{k-l+2j} \quad (k\geq l),
	\end{equation*}
	where the symbols $\croot_k$ with $k \not\in\{1,\ldots, n-1\}$ resulting from the above product are interpreted by the following axioms
	\begin{enumerate}[label=(\alph*)]
		\item $\croot_0=1$,
		\item $\croot_{k} = \croot_{2n-1-k}$ for $k < 2n$.
	\end{enumerate}
\end{defn}

One should notice that property (c) of Lemma~\ref{lem:ChebProps} corresponds to axiom (b) of Definition~\ref{def:ChebSemiring}, and that the multiplication rule in Definition~\ref{def:ChebSemiring} is just property (e) of Lemma~\ref{lem:ChebProps}.

\begin{rem} \label{rem:ChebBasis}
	As a result of the multiplication rule in Definition~\ref{def:ChebSemiring} and Lemma~\ref{lem:ChebProps}(e), we have a ring epimorphism
	\begin{align*}
		f\colon \integer[x] / (U_{n-1}(\tfrac{x}{2}) - U_{n}(\tfrac{x}{2}))  &\rightarrow \chebr\ps{2n+1}, \\
		\sum_{j=0}^{2n-1}a_j U_j(\tfrac{x}{2}) \mapsto \sum_{j=0}^{2n-1}a_j\croot_j.
	\end{align*}
	Note that we have $f(x)=\croot_1$ and one obtains axiom (b) of Definition~\ref{def:ChebSemiring} under $f$ via the relations
	\begin{align*}
		U_{n-1}(\tfrac{x}{2}) - U_{n}(\tfrac{x}{2}) &= 0 \\
		U_{1}(\tfrac{x}{2})(U_{n-1}(\tfrac{x}{2}) - U_{n}(\tfrac{x}{2})) &= U_{n-2}(\tfrac{x}{2}) + U_{n}(\tfrac{x}{2}) - U_{n-1}(\tfrac{x}{2}) - U_{n+1}(\tfrac{x}{2})\\ &= U_{n-2}(\tfrac{x}{2}) - U_{n+1}(\tfrac{x}{2}) = 0 \\
		&\vdots \\
		U_{n-1}(\tfrac{x}{2})(U_{n-1}(\tfrac{x}{2}) - U_{n}(\tfrac{x}{2})) &= U_{0}(\tfrac{x}{2}) - U_{2n-1}(\tfrac{x}{2}) =0
	\end{align*}
	in $\integer[x] / (U_{n-1}(\tfrac{x}{2}) - U_{n}(\tfrac{x}{2}))$. In particular, $\Ker f=0$ and thus we actually have an isomorphism
	\begin{equation*}
		\chebr\ps{2n+1} \cong \integer[x] / (U_{n-1}(\tfrac{x}{2}) - U_{n}(\tfrac{x}{2}))
	\end{equation*}
	A consequence of this isomorphism is that $\{\croot_0,\ldots,\croot_{n-1}\}$ is an integral basis for $\chebr\ps{2n+1}$, and that $\chebsr\ps{2n+1}$ is isomorphic to the subsemiring of the above quotient ring whose elements are non-negative sums of the polynomials $U_{i}(\frac{x}{2})$ with $0 \leq i \leq n-1$.
\end{rem}

\begin{rem} The reader may be wondering why we have decided not to work with, for example, the semiring
\begin{equation} \label{eq:AltChebsrDef} \tag{$\ast$}
	\nnint\left[U_1\left(\cos\frac{\pi}{2n+1}\right),\ldots,U_{n-1}\left(\cos\frac{\pi}{2n+1}\right)\right].
\end{equation}
The reason is due to technicalities later in the paper. We will later define an action of $\chebsr\ps{2n+1}$ on a category $\mod*KQ^{\coxA_{2n}}$, and for this action to be well-defined, we want to ensure that the generators $\croot_k$ are linearly independent for $k \leq 2n-1$. The problem occurs where for some values of $n$ we have
\begin{equation*}
	U_k\left(\cos\frac{\pi}{2n+1}\right) = \sum_{j < k} a_j U_j\left(\cos\frac{\pi}{2n+1}\right) \qquad a_j \in \nnint
\end{equation*}
This happens whenever the degree of the minimal polynomial of $\cos(\frac{\pi}{2n+1})$ is strictly less than $n$. For example, this happens in the case of $n=4$ where we want to have $\{\croot_0, \ldots, \croot_3\}$ as a basis for $\chebsr\ps{9}$, but we also have $U_3(\cos\frac{\pi}{9})=U_0(\cos\frac{\pi}{9})+U_1(\cos\frac{\pi}{9})$. On the other hand, if the minimal polynomial of $\cos(\frac{\pi}{2n+1})$ is of degree $n$, then we have $\chebr\ps{2n+1} \cong \wh{\chebr}_{2n+1}$ and we can write $\chebsr\ps{2n+1}$ as in (\ref{eq:AltChebsrDef}). For example, one can safely define $\chebsr\ps{5}=\nnint[\gratio]$, where $\gratio$ is the golden ratio.
\end{rem}

\begin{rem} \label{rem:ChebsrOrder}
	There exist homomorphisms of rings
	\begin{align*}
		\sigma\ps{2n+1}\colon \chebr\ps{2n+1} &\rightarrow \wh\chebr\ps{2n+1} \\ 
		\wh\sigma\ps{2n+1}\colon \wh\chebr\ps{2n+1} &\rightarrow \real
	\end{align*}	
	defined by $\sigma\ps{2n+1}(\croot_k) = U_k\left(\cos\frac{\pi}{2n+1}\right)$ and where $\wh\sigma\ps{2n+1}$ is the canonical embedding. In particular, $\sigma\ps{2n+1}$ is an epimorphism. We also have a restriction of $\sigma\ps{2n+1}$ to a homomorphism of semirings
	\begin{equation*}
		\srhom\ps{2n+1}\colon \chebsr\ps{2n+1} \rightarrow \wh\chebr\ps{2n+1}.
	\end{equation*}
	The ring $\real$ is totally ordered, which endows $\chebsr\ps{2n+1}$ with a partial ordering given by
	\begin{equation*}
		r \leq s \Leftrightarrow r = s \text{ or } \wh\sigma\ps{2n+1}\srhom\ps{2n+1}(r) <  \wh\sigma\ps{2n+1}\srhom\ps{2n+1}(s).
	\end{equation*}
\end{rem}

\subsection{Representations of $\chebr\ps{2n+1}$ and weighted (un)foldings}
Suppose $S$ is an exchange matrix over $\integer$ that is obtained by unfolding an exchange matrix $B$ over $\wh\chebr\ps{2n+1}$ of type $\coxH_4$, $\coxH_3$ or $\coxI_2(2n+1)$. In this case, the regular representation of the ring $\chebr\ps{2n+1}$ plays an important role with respect to the structure of $S$. We will see that the blocks of $S$ are the matrices that arise from the $\integer$-linear action of $\croot_1$ on the ring $\chebr\ps{2n+1}$. First let us investigate the regular representation of $\chebr\ps{2n+1}$.

\begin{defn} \label{def:RegRep}
	Consider the basis of $\chebr\ps{2n+1}$ given in Remark~\ref{rem:ChebBasis}. We denote by
	\begin{equation*}
		\rho\colon \chebr\ps{2n+1} \rightarrow \integer^{n \times n}
	\end{equation*}
	the regular representation of $\chebr\ps{2n+1}$ with respect to this basis. That is, $\rho$ is the ring homomorphism such that $\rho(\croot_i)v(r) = v(\croot_i r)$ for any $r \in \chebr\ps{2n+1}$, where $v(r) = (a_0,\ldots,a_{n-1}) \in \integer^n$ is the vector representing $r= \sum_{i=0}^{n-1} a_i \croot_i$. 
\end{defn}

\begin{lem} \label{lem:ChebRegRep}
	For any $i \in \{0,\ldots,n-1\}$, we have the following.
	\begin{enumerate}[label=(\alph*)]
		\item The entries of $\rho(\croot_i)$ are either $0$ or $1$. 
		\item $\rho(\croot_i)$ is symmetric.
	\end{enumerate}
\end{lem}
\begin{proof}
	From the product formula, it is easy to see that the matrix of $\rho(\croot_i)$ is the following. The matrix $\rho(\croot_0)$ is the identity. The matrix $\rho(\croot_{n-1})$ is such that $(\rho(\croot_{n-1}))_{jk}= 1$ if and only if $j \geq n-k-1$ (note that we use a zero-based index here, so the bottom right triangle of the matrix consists of ones). If $0 < i < n-1$, then $\rho(\croot_i)$ is of the form
	\begin{equation*}
		\begin{pmatrix}
			\tikznode{tlt1}{0} &  & \tikznode{tlt2}{0} & \tikznode{h1t}{1} & \tikznode{trt1}{0} &  &  & \tikznode{trt2}{0} \\
			 &  &  & \tikznode{h2t}{0} &  &  & &  \\
			\tikznode{tlt3}{0} &  & & \tikznode{h3t}{1}  &  & & &  \\
			\tikznode{h1l}{1} & \tikznode{h2l}{0} &\tikznode{h3l}{1}& & & & & \tikznode{trt3}{0}\\
			\tikznode{blt1}{0} & & &  &  & \tikznode{h3r}{1}& \tikznode{h2r}{0} & \tikznode{h1r}{1} \\
			 &  &  &  & \tikznode{h3b}{1} &  &  & \tikznode{brt1}{1} \\
			 &  & &  & \tikznode{h2b}{0} &  & &  \\
			\tikznode{blt2}{0} & & & \tikznode{blt3}{0} & \tikznode{h1b}{1} & \tikznode{brt2}{1} & & \tikznode{brt3}{1}
	\end{pmatrix}
	\begin{tikzpicture}[remember picture, overlay,shorten >=1pt,shorten <=1pt]
 		 \draw (tlt1) -- (tlt2) -- (tlt3) -- (tlt1);
		  \draw (trt1) -- (trt2) -- (trt3) -- (trt1);
		  \draw (blt1) -- (blt2) -- (blt3) -- (blt1);
		  \draw (brt1) -- (brt2) -- (brt3) -- (brt1);
		  \draw (h1t) -- (h1l) -- (h1b) -- (h1r) -- (h1t);
		  \draw (h2t) -- (h2l) -- (h2b) -- (h2r) -- (h2t);
		  \draw (h3t) -- (h3l) -- (h3b) -- (h3r) -- (h3t);
		  \draw ($0.5*(h3r)+0.5*(h3l)$) node {$\cdots$};
	\end{tikzpicture}
	\end{equation*}
	where the top triangles and bottom left triangle consists of zeros, and the bottom right triangle consists of ones. The rectangular hatches in the middle consist either entirely of ones or entirely of zeros, where the corners of the first hatch are at the the $(i,0)$ $(0,i)$, $(n-1-i, n-1)$ and $(n-1, n-1-i)$-th entries. In particular, the $(i,j)$-th entry is $1$ if $j$ is even and $0$ if $j$ is odd for any $j$ such that $(i,j)$ is within the area of the first hatch. All of these matrices are clearly symmetric with entries that are precisely either $0$ or $1$.
\end{proof}

Now let us consider a weighted folding $F\colon Q^S \rightarrow Q^B$ that corresponds to a weighted unfolding $S$ of an exchange matrix $B$ over $\wh\chebr\ps{2n+1}$. In particular, we consider $F$ to be either one of the $\coxH$-type foldings of Figure~\ref{H} or $\coxI$-type foldings of Remark~\ref{rem:IFolding}. Recall that the exchange $\integer$-matrix $S$ has the structure of a block matrix $(S_{[i][j]})_{[i],[j] \in Q^B_0}$, where for any $i \in Q^S_0$, we have $F(i)=[i]$. Under this notation, we have the following result.

\begin{lem} \label{lem:RegRepExchange}
	Let $B=(b_{[i][j]})_{[i],[j] \in Q^B_0}$ be an exchange matrix over $\wh\chebr\ps{2n+1}$ of type $\coxH$ or $\coxI_2(2n+1)$. Suppose $S=(s_{ij})_{i,j \in Q^S_0}$ is an integer block matrix $(S_{[i][j]})_{[i],[j] \in Q^B_0}$ such that $S_{[i][j]}=\rho(b'_{[i][j]})$, where $b'_{[i][j]} \in \chebr\ps{2n+1}$ is such that $\sigma\ps{2n+1}(b'_{[i][j]}) = b_{[i][j]}$ and $\rho$ is the regular representation of $\chebr\ps{2n+1}$.
	\begin{enumerate}[label=(\alph*)]
		\item $S$ is block skew-symmetric: $S_{[i][j]}=-S_{[j][i]}$.
		\item $S$ is skew-symmetric: $s_{ij}=-s_{ji}$.
		\item $S$ satisfies (1) and (2) of Definition~\ref{defn:Unfolding}.
	\end{enumerate}
\end{lem}
\begin{proof}
	First note that by definition we have $b_{[i][j]} \in \{0,\pm 1,\pm 2\cos\frac{\pi}{2n+1}\}$. The corresponding values of $\chebr\ps{2n+1}$ we need to consider are therefore the values $b'_{[i][j]} \in \{0,\pm 1,\pm \croot_1\}$. Since $B$ is skew-symmetric and $\rho$ is a representation of $\chebr\ps{2n+1}$, $S$ is a skew-symmetric block matrix (that is, $S_{[i][j]}=-S_{[j][i]}$). This proves (a). In particular, $S$ is skew-symmetric since $\rho(b'_{[i][j]})$ is symmetric, which proves (b). Moreover, (2) of Definition~\ref{defn:Unfolding} is satisfied by Lemma~\ref{lem:ChebRegRep}(a).
	
	It remains to show that (1) holds. Let $\wh{W}=(\wh{W}_i)_{0\leq i \leq n-1}$ be the diagonal matrix such that $\wh{W}_i = \sigma\ps{2n+1}(\croot_i) = U_i(2\cos\frac{\pi}{2n+1})$. Consider the weight matrix $W$ with block diagonal structure $(W_{[i]})_{i \in Q^B_0}$ such that $W_{[i]}=W_{[j]}=\wh{W}$ for any $[i],[j] \in Q^B_0$. The matrix $WSW\inv$ then has a block structure indexed by $Q^B_0$, with $(WSW\inv)_{[i][j]} = \wh{W}S_{[i][j]}\wh{W}\inv$. Since $\rho$ is the regular representation of $\chebr\ps{2n+1}$, the sum of entries in the $k$-th column of the matrix $\wh{W}S_{[i][j]}$ is the value $\sigma\ps{2n+1}(b'_{[i][j]} \croot_k)$. Multiplying on the right by the matrix $\wh{W}\inv$ has the effect of multiplying the $k$-th column of $\wh{W}S_{[i][j]}$ by $(\sigma\ps{2n+1}(\croot_k))\inv$. Thus, the sum of each column of $\wh{W}S_{[i][j]}\wh{W}\inv$ is $\sigma\ps{2n+1}(b'_{[i][j]})=b_{[i][j]}$. Hence, (1) holds.
\end{proof}

\begin{prop} \label{prop:RegRepExchange}
	If $F\colon Q^S \rightarrow Q^B$ is a folding as in Figure~\ref{H} or Remark~\ref{rem:IFolding}, then there exists an ordering $\leq$ of $Q^S_0$ such that $S_{[i][j]}=\rho(b'_{[i][j]})$, where $b'_{[i][j]} \in \chebr\ps{2n+1}$ is such that $\sigma\ps{2n+1}(b'_{[i][j]}) = b_{[i][j]}$.
\end{prop}
\begin{proof}
	Let $\preceq$ be any total ordering of $Q^B_0$. Define $i \leq j$ if and only if $\kappa(i) \leq \kappa(j)$ and $[i] \preceq [j]$. For foldings as in Figure~\ref{H}, we have $b_{[i][j]} = 0$, $b_{[i][j]} = \pm 1$ or $b_{[i][j]} =\pm \gratio$ for each entry of $B$. It is easy to see that the corresponding block of $S$ under this ordering is
	\begin{equation*}
		S_{[i][j]} = \pm\begin{pmatrix}0 & 0 \\ 0 & 0 \end{pmatrix}, \qquad
		S_{[i][j]} = \pm\begin{pmatrix}1 & 0 \\ 0 & 1 \end{pmatrix}
		\qquad \text{or} \qquad
		S_{[i][j]} = \pm\begin{pmatrix}0 & 1 \\ 1 & 1 \end{pmatrix}
	\end{equation*}
	respectively. We also have $\chebr\ps{5} \cong \wh\chebr\ps{5}$ in this setting, and it is easy to see that we precisely have $S_{[i][j]} =\rho(b_{[i][j]})$. For foldings as in Remark~\ref{rem:IFolding}, one notes that for any vertex $i \in Q^S_0$, we have arrows (of some orientation) between pairs of vertices $(i,i-1)$ and $(i,i+1)$ in $Q^S$, unless $i\in \{0,2n-1\}$, in which case we have only one arrow. In particular, this mimics the relation $\croot_1 \croot_i = \croot_{i-1}+\croot_{i+1}$ for $i \not\in\{0,2n-1\}$. Further noting that $\sigma\ps{2n+1}(\croot_1) = 2\cos\frac{\pi}{2n+1}$ and that $b_{[i][j]} \in \{0, \pm 2\cos\frac{\pi}{2n+1}\}$, the result follows.
\end{proof}

\section{Projections of module and derived categories under foldings}
\label{section-proj}
One of the earliest and most celebrated results in the representation theory of path algebras is Gabriel's Theorem~\cite{Gabriel}. In short, Gabriel's Theorem states that a finite connected quiver $Q$ has finitely many iso-classes of indecomposable representations (over an algebraically closed field) if and only if the underlying graph of $Q$ is a simply-laced Dynkin diagram. Furthermore, the dimension vectors of the iso-classes of indecomposable representations bijectively correspond to the positive roots of the root system of the Dynkin diagram. This result was extended by Dlab and Ringel in \cite{DlabRingelFinite} to show that a finite connected valued quiver has finitely many iso-classes of indecomposable representations if and only if the underlying graph is a Dynkin diagram (including the multiply-laced diagrams). Moreover, the dimension vectors of these representations are also in bijective correspondence with positive roots.

The purpose of this section is to review and investigate the connection between a folding of quivers $Q^\gend \rightarrow Q^{\gend'}$, where $\gend$ is a simply-laced Dynkin diagram, and categories associated to a quiver of type $\gend$. For cases where $\gend'$ corresponds to a crystallographic system, some work has already been done in this respect (for example, \cite{DengDuI,DengDuII,FoldingGabriel}). Our main focus will be on the non-crystallographic cases where $\gend' \in \{\coxH_{4},\coxH_{3},\coxI_2(2n+1)\}$. Recall that in this case $Q^\gend$ is a $\wh\chebr\ps{2n+1}$-vertex-weighted quiver with weight function $\kappa$ (see Definitions~\ref{defn:RVWQ} and \ref{defn:RQfolding}).

From now on, we assume $\field$ to be an algebraically closed field. We denote by $\field Q$ the path algebra of $Q$ and $\mod*\field Q$ the category of finitely-generated right $\field Q$ modules. In particular, paths in $\field Q$ are read from left to right. In a category with Auslander-Reiten sequences (such as with $\mod*\field Q$), we denote the Auslander-Reiten translate by $\tau$.

We will show a result analogous to Gabriel's Theorem holds for these non-crystallo\-graphic foldings. Specifically, we will show that one can project the dimension vectors of indecomposable $KQ^\gend$-modules onto the set of vectors $\bigcup_{i=1}^k w_i \rootsys^+_{\gend'}$, where $\rootsys^+_{\gend'}$ is the set of positive roots of $\gend'$. This in itself is not surprising for type $\coxH$, as it follows directly from both Moody and Patera's projection of root systems and Gabriel's Theorem. We show this also holds for type $\coxI_2(2n+1)$. However what is surprising is that this projection has a remarkably well-organised structure, in the sense that each row in the Auslander-Reiten quiver of $\mod* KQ^\gend$ is such that either every module in a row projects onto positive $\gend'$-roots, or every module in a row projects onto a specific multiple of positive $\gend'$-roots. The end result of this projection is a module category that has a well-defined semiring action.

\subsection{Definitions and results}

Henceforth, for any given quiver $Q$, we respectively denote by $S(i)$, $I(i)$ and $P(i)$ the simple module, indecomposable injective module, and indecomposable projective module in $\mod*KQ$ corresponding to the vertex $i \in Q_0$. 

\begin{defn} \label{def:ProjMap}
	Let $F\colon Q^\gend \rightarrow Q^{\gend'}$ be a folding of weighted quivers, where $\gend' \in \{\coxH_3,\coxH_4,\coxI_2(2n+1)\}$. Then there exists a unique arrow $a \in Q^{\gend'}_1$ such that $\aw(a) = 2 \cos\frac{\pi}{2n+1}$, where $n=2$ if $\gend' \in \{\coxH_3,\coxH_4\}$. Define a function
	\begin{equation*}
		d_F\colon \integer^{|Q^\gend_0|} \rightarrow (\wh\chebr\ps{2n+1})^{|Q^{\gend'}_0|}
	\end{equation*}
	that maps a vector $v = (v_i)_{i\in Q^\gend_0}$ to $d_F(v) = (v'_{i'})_{i'\in Q^{\gend'}_0}$ by the weighted sum
	\begin{equation*}
		v'_{i'}=\sum_{i:F(i)=i'} \vw(i) v_i.
	\end{equation*}
	In addition, we define a map $\dimproj_F\colon \Ob(KQ^\gend) \rightarrow (\wh\chebr\ps{2n+1})^{|Q^{\gend'}_0|}$, where $\Ob(KQ^\gend)$ is the class of objects of $\mod*KQ$. We define $\dimproj_F(M)=d_F \dimvect M$ and call $\dimproj_F(M)$ the \emph{$F$-projected dimension vector} of $M$.
\end{defn}

The aim of this section is to prove the following result for module categories, which we will later extend to the bounded derived categories.

\begin{thm}\label{thm:Folding}
Let $F\colon Q^{\gend} \rightarrow Q^{\gend'}$ be a weighted folding of quivers, where $\gend \in \{\coxA_{2n},\coxD_6,\coxE_8\}$ and $\gend'\in\{\coxI_2(2n+1),\coxH_3,\coxH_4\}$. For $\gend' \in \{\coxH_3,\coxH_4\}$, define $n=2$.
	\begin{enumerate}[label=(\alph*)]
		\item For each $0 \leq j \leq n-1$, define a set
		\begin{equation*}
			\mathcal{I}\ps{j} = \left\{\tau^m I(i) \in \mod* KQ^\gend : i\in Q^\gend_0, \vw(i)=U_j\left(\cos \frac{\pi}{2n+1}\right)\!, m\in\nnint \right\}.
		\end{equation*}
		Then $\dimproj_F(M)$ is a positive root of $\gend'$ for any $M \in \mathcal{I}\ps{0}$.
		\item For any $i,j\in Q^\gend_0$ such that $F(i)=F(j)$ and any $m \in \nnint$, we have
		\begin{equation*}
			\vw(j) \dimproj_F( \tau^m I(i)) = \vw(i) \dimproj_F (\tau^m I(j)).
		\end{equation*}
		\item For each $0 \leq j \leq n-1$, we have $||\dimproj_F(M)|| = U_j\left(\cos\frac{\pi}{2n+1}\right)$ for all $M \in \mathcal{I}\ps{j}$.
	\end{enumerate}
\end{thm}

\begin{rem}
	Given the folding in the above theorem, let $i\in Q^\gend_0$. Then let $j\in Q^\gend_0$ be such $P(j) \cong \tau^m I(i)$ for some $m \geq 0$. The folding forces the quiver $Q^\gend$ to be of a particular weighting such that $\vw(i)=\vw(j)$. This is easy to see from the Auslander-Reiten theory of $\coxA\coxD\coxE$-quivers.
\end{rem}

\begin{rem} \label{rem:ThmDual}
	Since $KQ^\gend$ is representation-finite, one could dually work with the set
	\begin{equation*}
		\mathcal{P}\ps{j} = \left\{\tau^{-m} P(i) \in \mod* KQ^\gend : i\in Q^\gend_0, \vw(i)=U_j\left(\cos \frac{\pi}{2n+1}\right), m\in\nnint \right\}
	\end{equation*}
	for Theorem~\ref{thm:Folding}(a). Similarly, one could write (b) as
	\begin{equation*}
		\vw(j)\dimproj_F( \tau^{-m} P(i)) = \vw(i) \dimproj_F( \tau^{-m} P(j)).
	\end{equation*}
\end{rem}

\subsection{Folding onto $\coxH$-type quivers}
Foldings onto a diagram of type $\coxH$ are given in Figure~\ref{H}. Recall that $\gratio=2 \cos \frac{\pi}{5}$, the golden ratio. The $\goldint$-quiver of type $\coxH_2$ is unique up to a relabelling of the vertices, and thus serves as an easy starting point. This will provide a motivating example for the theory we will develop.

\begin{exam}
	The folding $F\colon Q^{\coxA_4} \rightarrow Q^{\coxH_2}$ must be of the following form (or its opposite).
	\begin{center}
		\begin{tikzpicture}
\draw (0.2,0.4) node {$1$};
\draw (1.2,0.4) node {$\phi_2$};
\draw (0.2,-0.4) node {$\phi_1$};
\draw (1.2,-0.4) node {$2$};

\draw (1.8,0) node {$\rightarrow$};

\draw (2.4,0) node {$[1]$};
\draw (3.4,0) node {$[2]$};
\draw (2.9,0.2) node {\footnotesize$\gratio$};
\draw [->](0.4,0.4) -- (1,0.4);
\draw [->](0.4,-0.4) -- (1,-0.4);
\draw [->](0.4,-0.2) -- (1,0.2);
\draw [->](2.6,0) -- (3.2,0);
\end{tikzpicture}
	\end{center}
	The vertices labelled $i$ have weight 1, the vertices labelled $\phi_i$ have weight $\varphi$, and the vertices labelled $i$ and $\phi_i$ map onto $[i]$. Noting the relation $\gratio^2=\gratio+1$, Theorem~\ref{thm:Folding} becomes obvious from the Auslander-Reiten quiver of $\mod* KQ^{\coxA_4}$, which is illustrated in Figure~\ref{fig:A4ARQuiver}.
\end{exam}

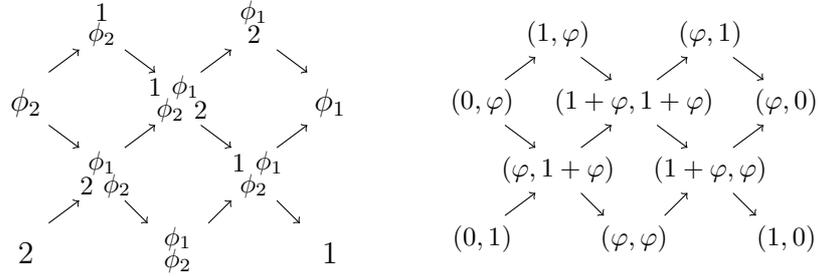
\begin{figure}[h]
	\begin{tikzpicture}
\draw (-0.5,0.5) node {$1$};

\draw (-2.5,0.7) node {\footnotesize$\phi_1$};
\draw (-2.5,0.4) node {\footnotesize$\phi_2$};

\draw (-4.5,0.5) node {$2$};

\draw (-1.7,1.7) node {\footnotesize$1$};
\draw (-1.3,1.7) node {\footnotesize$\phi_1$};
\draw (-1.5,1.4) node {\footnotesize$\phi_2$};

\draw (-3.7,1.4) node {\footnotesize$2$};
\draw (-3.3,1.4) node {\footnotesize$\phi_2$};
\draw (-3.5,1.7) node {\footnotesize$\phi_1$};

\draw (-0.5,2.5) node {$\phi_1$};

\draw (-2.4,2.7) node {\footnotesize$\phi_1$};
\draw (-2.6,2.4) node {\footnotesize$\phi_2$};
\draw (-2.8,2.7) node {\footnotesize$1$};
\draw (-2.2,2.4) node {\footnotesize$2$};

\draw (-4.5,2.5) node {$\phi_2$};

\draw (-1.5,3.7) node {\footnotesize$\phi_1$};
\draw (-1.5,3.4) node {\footnotesize$2$};
\draw (-3.5,3.7) node {\footnotesize$1$};
\draw (-3.5,3.4) node {\footnotesize$\phi_2$};

\draw [->](-4.2,2.9) -- (-3.8,3.2);
\draw [->](-3.2,3.2) -- (-2.8,2.9);
\draw [->](-2.2,2.9) -- (-1.8,3.2);
\draw [->](-1.2,3.2) -- (-0.8,2.9);
\draw [->](-4.2,2.2) -- (-3.8,1.9);
\draw [->](-3.2,1.9) -- (-2.8,2.2);
\draw [->](-2.2,2.2) -- (-1.8,1.9);
\draw [->](-1.2,1.9) -- (-0.8,2.2);
\draw [->](-4.2,0.9) -- (-3.8,1.2);
\draw [->](-3.2,1.2) -- (-2.9,0.9);
\draw [->](-2.1,0.9) -- (-1.8,1.2);
\draw [->](-1.2,1.2) -- (-0.9,0.9);

\draw (5.5,0.7) node {\footnotesize$(1,0)$};

\draw (3.5,0.7) node {\footnotesize$(\gratio,\gratio)$};

\draw (1.5,0.7) node {\footnotesize$(0,1)$};

\draw (4.5,1.6) node {\footnotesize$(1+\gratio,\gratio)$};

\draw (2.5,1.6) node {\footnotesize$(\gratio,1+\gratio)$};

\draw (5.5,2.5) node {\footnotesize$(\gratio,0)$};

\draw (3.5,2.5) node {\footnotesize$(1+\gratio,1+\gratio)$};

\draw (1.5,2.5) node {\footnotesize$(0,\gratio)$};

\draw (4.5,3.4) node {\footnotesize$(\gratio,1)$};
\draw (2.5,3.4) node {\footnotesize$(1,\gratio)$};

\draw [->](1.8,2.8) -- (2.2,3.1);
\draw [->](2.8,3.1) -- (3.2,2.8);
\draw [->](3.8,2.8) -- (4.2,3.1);
\draw [->](4.8,3.1) -- (5.2,2.8);
\draw [->](1.8,2.2) -- (2.2,1.9);
\draw [->](2.8,1.9) -- (3.2,2.2);
\draw [->](3.8,2.2) -- (4.2,1.9);
\draw [->](4.8,1.9) -- (5.2,2.2);
\draw [->](1.8,1) -- (2.2,1.3);
\draw [->](2.8,1.3) -- (3.1,1);
\draw [->](3.9,1) -- (4.2,1.3);
\draw [->](4.8,1.3) -- (5.1,1);
\end{tikzpicture}
	\caption{Left: Auslander-Reiten quiver of $KQ^{\coxA_4}$ in the folding $F\colon Q^{\coxA_4} \rightarrow Q^{\coxH_2}$. Right: The corresponding $F$-projected dimension vectors of the indecomposable $KQ^{\coxA_4}$-modules.}
	\label{fig:A4ARQuiver}
\end{figure}

For the quivers of type $\coxH_3$ and $\coxH_4$, the following technical lemma is incredibly useful.

\begin{lem} \label{lem:FoldingRadSeries}
	Let $F\colon Q^\gend \rightarrow Q^{\coxH_n}$ be a folding of quivers and write
	\begin{equation*}
		Q^\gend_0 = \{i : 1 \leq  i \leq n, \vw(i)=1\} \cup \{\phi_i : 1 \leq i \leq n, \vw(\phi_i)=\gratio\}
	\end{equation*}
	such that $F(i)=F(\phi_i)$ for each $i$. For each $1 \leq i \leq n$, consider the radical series
	\begin{equation*}
		\tau^j I(i) \supset \rad \tau^j I(i) \supset \ldots \supset \rad^m \tau^j I(i) = 0
	\end{equation*}
	with factors $X_{i,j,k} = \rad^k \tau^j I(i) / \rad^{k+1} \tau^j I(i)$, and the corresponding radical series
	\begin{equation*}
		\tau^j I(\phi_i) \supset \rad \tau^j I(\phi_i) \supset \ldots \supset \rad^{m'} \tau^j I(\phi_i) = 0
	\end{equation*}
	with factors $X_{\phi_i,j,k} = \rad^k \tau^j I(\phi_i) / \rad^{k+1} \tau^j I(\phi_i)$. Then $m=m'$ and
	\begin{equation*}
		\dimproj_F( X_{\phi_i,j,k}) = \gratio \dimproj_F( X_{i,j,k})
	\end{equation*}
	for each $1 \leq i \leq n$ and $j,k \geq 0$.
\end{lem}
\begin{proof}
	It is sufficient to consider the case where $\gend= \coxE_8$ and $\gend'=\coxH_4$, as the quivers of type $\coxA_4$ and $\coxD_6$ may be embedded into $\coxE_8$ in the natural way. In this case, $F$ is of the form
	\begin{center}
		\begin{tikzpicture}
\draw (-2,0.4) node {$1$};
\draw (-1,0.4) node {$2$};
\draw (0,0.4) node {$3$};
\draw (1,0.4) node {$\phi_4$};
\draw (-2,-0.4) node {$\phi_1$};
\draw (-1,-0.4) node {$\phi_2$};
\draw (0,-0.4) node {$\phi_3$};
\draw (1,-0.4) node {$4$};
\draw (-1.8,0.4) -- (-1.2,0.4);
\draw (-0.8,0.4) -- (-0.2,0.4);
\draw (0.2,0.4) -- (0.8,0.4);
\draw (-1.8,-0.4) -- (-1.2,-0.4);
\draw (-0.8,-0.4) -- (-0.2,-0.4);
\draw (0.2,-0.4) -- (0.8,-0.4);
\draw (0.2,-0.2) -- (0.8,0.2);

\draw (1.6,0) node {$\rightarrow$};

\draw (2.2,0) node {$[1]$};
\draw (3.2,0) node {$[2]$};
\draw (4.2,0) node {$[3]$};
\draw (5.2,0) node {$[4]$};
\draw (4.7,0.2) node {\footnotesize$\gratio$};
\draw (2.4,0) -- (3,0);
\draw (3.4,0) -- (4,0);
\draw (4.4,0) -- (5,0);
\end{tikzpicture}
	\end{center}
	where all arrows in the same column point in the same direction, each vertex labelled $j$ has weight $1$ and each vertex labelled $\phi_j$ has weight $\gratio$. We will first prove the result for projectives. That is, we will show that
	\begin{equation*} \tag{$\ast$} \label{eq:HRadProj}
		\dimproj_F( \rad^k P(\phi_i) / \rad^{k+1} P(\phi_i)) = \varphi \dimproj_F( \rad^k  P(i) / \rad^{k+1} P(i))
	\end{equation*}
	for each $i$ and $k$.
	
	First note that $S(v) \subseteq \rad^k P(u) / \rad^{k+1} P(u)$ with $u,v \in Q^{\coxE_8}_0$ if and only if there exists a path in $Q^{\coxE_8}$ of length $k$ from $u$ to $v$. Thus, to prove the result for projectives, it is sufficient to compare the arrows of source $l$ for each $1 \leq l \leq n$ with the arrows of source $\phi_l$ in $Q^{\coxE_8}$. It follows from the folding $F$ that we have
	\begin{center}
		\begin{tikzpicture}
			\draw [anchor=east](-9,0) node {$Q^{\coxE_8}_1\ni$};
			\draw (-8.9,0.4) node {\footnotesize$l$};
			\draw (-8.9,0) node {$\downarrow$};
			\draw (-8.9,-0.4) node {\footnotesize$l'$};
			
			\draw (-8.5,0) node {$\Rightarrow$};
			
			\draw (-8.1,0.4) node {\footnotesize$\phi_l$};
			\draw (-8.1,0) node {$\downarrow$};
			\draw (-8.1,-0.4) node {\footnotesize$\phi_{l'}$};
			\draw [anchor=west](-8,0) node {$\in Q^{\coxE_8}_1,$};
			
			\draw [anchor=east](-5,0) node {$Q^{\coxE_8}_1\ni$};
			\draw (-4.9,0.4) node {\footnotesize$3$};
			\draw (-4.9,0) node {$\downarrow$};
			\draw (-4.9,-0.4) node {\footnotesize$\phi_4$};
			
			\draw (-4.5,0) node {$\Rightarrow$};
			
			\draw (-3.7,0.4) node {\footnotesize$\phi_3$};
			\draw (-4,0) node {$\swarrow$};
			\draw (-3.4,0) node {$\searrow$};
			\draw (-4.3,-0.4) node {\footnotesize$4$};
			\draw (-3.1,-0.4) node {\footnotesize$\phi_4$};
			\draw [anchor=west](-3.1,0) node {$\in Q^{\coxE_8}_1,$};
			
			\draw [anchor=east](0,0) node {$Q^{\coxE_8}_1\ni$};
			\draw (0.1,0.4) node {\footnotesize$4$};
			\draw (0.1,0) node {$\downarrow$};
			\draw (0.1,-0.4) node {\footnotesize$\phi_3$};
			
			\draw (0.5,0) node {$\Rightarrow$};
			
			\draw (1.3,0.4) node {\footnotesize$\phi_4$};
			\draw (1,0) node {$\swarrow$};
			\draw (1.6,0) node {$\searrow$};
			\draw (0.7,-0.4) node {\footnotesize$3$};
			\draw (1.9,-0.4) node {\footnotesize$\phi_3$};
			\draw [anchor=west](1.9,0) node {$\in Q^{\coxE_8}_1$.};
		\end{tikzpicture}
	\end{center}
	Now note that a maximal path in $Q^{\coxE_8}$ from a vertex $i$ has the same length as a maximal path from $\phi_i$, and note that
	\begin{align*}
		\dimproj_F( S(\phi_l)) = d_F(e_{\phi_l}) &= \gratio e_{[l]} = \gratio d_F(e_l) = \gratio \dimproj_F( S(l)) \\
		\dimproj_F( S(l) \oplus S(\phi_l)) = d_F(e_{l} + e_{\phi_l}) &= (1 + \gratio) e_{[l]} = \gratio^2 e_{[l]} = \gratio d_F(e_{\phi_l}) = \gratio \dimproj_F( S(\phi_l)),
	\end{align*}
	where $e_v$ is the $\integer^{8}$-vector with a 1 at entry $v$ and 0 otherwise, and $e_{[v]}$ is the  $(\integer[\gratio])^{4}$-vector with a 1 at entry $[v]$ and 0 otherwise. This proves the result for projectives. The proof for injectives is similar --- we merely compare the arrows of target $l$ and $\phi_l$ instead. That is, it is easy to see that we also have
	\begin{equation*} \tag{$\ast\ast$}\label{eq:HRadInj}
		\dimproj_F( \rad^k I(\phi_i) / \rad^{k+1} I(\phi_i)) = \varphi \dimproj_F( \rad^k  I(i) / \rad^{k+1} I(i))
	\end{equation*}
	for each $i$ and $k$.
	
	The proof for the more general result in the lemma statement works by induction. So let suppose there exists $1\leq i \leq n$ and $j \geq 0$ such that
	\begin{equation*}
		\dimproj_F( X_{\phi_i,j,k}) = \gratio \dimproj_F( X_{i,j,k})
	\end{equation*}
	for all $k\geq 0$. We will show that this implies
	\begin{equation*}
		\dimproj_F( X_{\phi_i,j+1,k}) = \gratio \dimproj_F( X_{i,j+1,k})
	\end{equation*}
	for all $k\geq 0$.
	
	First consider the projective presentations
	\begin{equation*}
		P' \rightarrow P \rightarrow M \rightarrow 0,
	\end{equation*}
	\begin{equation*}
		P'_\gratio \rightarrow P_\gratio \rightarrow M_\gratio \rightarrow 0,
	\end{equation*}
	where $M = \tau^j(I(i))$ and $M_\gratio = \tau^j(I(\phi_i))$. Since $\dimproj_F( X_{\phi_i,j,0}) = \gratio \dimproj_F( X_{i,j,0})$, we have $\dimproj_F( \tp P_\gratio) = \gratio \dimproj_F( \tp P)$. By (\ref{eq:HRadProj}), this implies that
	\begin{equation*}
		\dimproj_F( \rad^{k} P_\gratio / \rad^{k+1} P_\gratio) = \gratio \dimproj_F( \rad^{k} P_\gratio / \rad^{k+1} P)
	\end{equation*}
	for each $k \geq 0$, and thus
	\begin{align*}
	\dimproj_F( \rad^{k} \syz(M_\gratio) / \rad^{k+1} \syz(M_\gratio)) &= \gratio \dimproj_F( \rad^{k} \syz(M) / \rad^{k+1} \syz(M)), \\
		\dimproj_F( \rad^{k} P'_\gratio / \rad^{k+1} P'_\gratio) &= \gratio \dimproj_F( \rad^{k} P' / \rad^{k+1} P').
	\end{align*}
	The projective presentations above give rise to exact sequences
	\begin{equation*}
		0 \rightarrow \tau M \rightarrow I' \rightarrow I,
	\end{equation*}
	\begin{equation*}
		0 \rightarrow \tau M_\gratio \rightarrow I'_\gratio \rightarrow I_\gratio,
	\end{equation*}
	where $I$, $I'$, $I_\gratio$ and $I'_\gratio$ are the injective envelopes of $\tp P$, $\tp P'$, $\tp P_\gratio$ and $\tp P'_\gratio$ respectively. By similar reasoning to the projective presentations and by (\ref{eq:HRadInj}), we get
	\begin{equation*}
		\dimproj_F( \tau M_\gratio)=\dimproj_F( X_{\phi_i,j+1,k}) = \gratio \dimproj_F( X_{i,j+1,k}) = \gratio \dimproj_F( \tau M),
	\end{equation*}
	as required. Since we know that $\dimproj_F( X_{\phi_i,0,k}) = \gratio \dimproj_F( X_{i,0,k})$ for all $1 \leq i \leq n$ and $k\geq 0$, this proves the result.
\end{proof}

\begin{exam} \label{ex:D6H3Folding}
	Consider the folding $F\colon Q^{\coxD_6} \rightarrow Q^{\coxH_3}$ given by
	\begin{center}
		\begin{tikzpicture}
\draw (-0.8,0.4) node {$1$};
\draw (0.2,0.4) node {$2$};
\draw (1.2,0.4) node {$\phi_3$};
\draw (-0.8,-0.4) node {$\phi_1$};
\draw (0.2,-0.4) node {$\phi_2$};
\draw (1.2,-0.4) node {$3$};

\draw (1.8,0) node {$\rightarrow$};

\draw (2.4,0) node {$[1]$};
\draw (3.4,0) node {$[2]$};
\draw (4.4,0) node {$[3]$};
\draw (3.9,0.2) node {\footnotesize$\gratio$};
\draw [->](-0.6,0.4) -- (0,0.4);
\draw [->](0.4,0.4) -- (1,0.4);
\draw [->](-0.6,-0.4) -- (0,-0.4);
\draw [->](0.4,-0.4) -- (1,-0.4);
\draw [->](0.4,-0.2) -- (1,0.2);
\draw [->](2.6,0) -- (3.2,0);
\draw [->](3.6,0) -- (4.2,0);
\end{tikzpicture}
	\end{center}
	where the vertices labelled $i$ have weight 1, the vertices labelled $\phi_i$ have weight $\gratio$, and the vertices labelled $i$ and $\phi_i$ project onto $[i]$. One can see Theorem~\ref{thm:Folding} (and Lemma~\ref{lem:FoldingRadSeries}) holds in Figure~\ref{fig:D6ARQuiver}.
\end{exam}

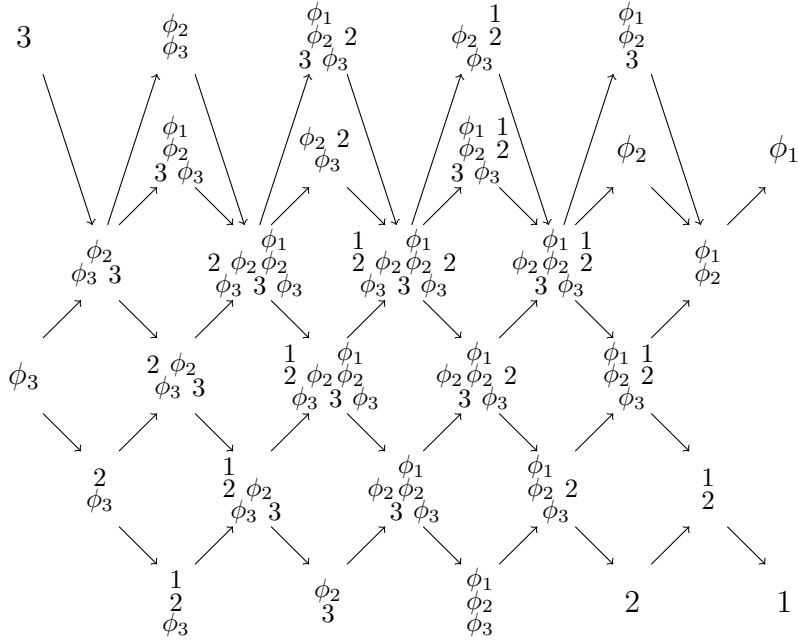
\begin{figure}[h]
	\begin{tikzpicture}
\draw (1,-6) node {$1$};

\draw (-1,-6) node {$2$};

\draw (-3,-5.7) node {\footnotesize$\phi_1$};
\draw (-3,-6) node {\footnotesize$\phi_2$};
\draw (-3,-6.3) node {\footnotesize$\phi_3$};

\draw (-5,-5.85) node {\footnotesize$\phi_2$};
\draw (-5,-6.15) node {\footnotesize$3$};

\draw (-7,-5.7) node {\footnotesize$1$};
\draw (-7,-6) node {\footnotesize$2$};
\draw (-7,-6.3) node {\footnotesize$\phi_3$};

\draw (0,-4.35) node {\footnotesize$1$};
\draw (0,-4.65) node {\footnotesize$2$};

\draw (-2.2,-4.2) node {\footnotesize$\phi_1$};
\draw (-2.2,-4.5) node {\footnotesize$\phi_2$};
\draw (-1.8,-4.5) node {\footnotesize$2$};
\draw (-2,-4.8) node {\footnotesize$\phi_3$};

\draw (-3.9,-4.2) node {\footnotesize$\phi_1$};
\draw (-3.9,-4.5) node {\footnotesize$\phi_2$};
\draw (-4.3,-4.5) node {\footnotesize$\phi_2$};
\draw (-3.7,-4.8) node {\footnotesize$\phi_3$};
\draw (-4.1,-4.8) node {\footnotesize$3$};

\draw (-6.3,-4.2) node {\footnotesize$1$};
\draw (-5.9,-4.5) node {\footnotesize$\phi_2$};
\draw (-6.3,-4.5) node {\footnotesize$2$};
\draw (-5.7,-4.8) node {\footnotesize$3$};
\draw (-6.1,-4.8) node {\footnotesize$\phi_3$};

\draw (-8,-4.35) node {\footnotesize$2$};
\draw (-8,-4.65) node {\footnotesize$\phi_3$};

\draw (-1.2,-2.7) node {\footnotesize$\phi_1$};
\draw (-0.8,-2.7) node {\footnotesize$1$};
\draw (-1.2,-3) node {\footnotesize$\phi_2$};
\draw (-0.8,-3) node {\footnotesize$2$};
\draw (-1,-3.3) node {\footnotesize$\phi_3$};

\draw (-3,-2.7) node {\footnotesize$\phi_1$};
\draw (-3,-3) node {\footnotesize$\phi_2$};
\draw (-3.4,-3) node {\footnotesize$\phi_2$};
\draw (-2.6,-3) node {\footnotesize$2$};
\draw (-2.8,-3.3) node {\footnotesize$\phi_3$};
\draw (-3.2,-3.3) node {\footnotesize$3$};

\draw (-4.7,-2.7) node {\footnotesize$\phi_1$};
\draw (-5.5,-2.7) node {\footnotesize$1$};
\draw (-4.7,-3) node {\footnotesize$\phi_2$};
\draw (-5.1,-3) node {\footnotesize$\phi_2$};
\draw (-5.5,-3) node {\footnotesize$2$};
\draw (-4.5,-3.3) node {\footnotesize$\phi_3$};
\draw (-4.9,-3.3) node {\footnotesize$3$};
\draw (-5.3,-3.3) node {\footnotesize$\phi_3$};

\draw (-6.9,-2.85) node {\footnotesize$\phi_2$};
\draw (-7.3,-2.85) node {\footnotesize$2$};
\draw (-6.7,-3.15) node {\footnotesize$3$};
\draw (-7.1,-3.15) node {\footnotesize$\phi_3$};

\draw (-9,-3) node {$\phi_3$};

\draw (0,-1.35) node {\footnotesize$\phi_1$};
\draw (0,-1.65) node {\footnotesize$\phi_2$};

\draw (-2,-1.2) node {\footnotesize$\phi_1$};
\draw (-1.6,-1.2) node {\footnotesize$1$};
\draw (-2,-1.5) node {\footnotesize$\phi_2$};
\draw (-2.4,-1.5) node {\footnotesize$\phi_2$};
\draw (-1.6,-1.5) node {\footnotesize$2$};
\draw (-1.8,-1.8) node {\footnotesize$\phi_3$};
\draw (-2.2,-1.8) node {\footnotesize$3$};

\draw (-3.8,-1.2) node {\footnotesize$\phi_1$};
\draw (-4.6,-1.2) node {\footnotesize$1$};
\draw (-3.8,-1.5) node {\footnotesize$\phi_2$};
\draw (-4.2,-1.5) node {\footnotesize$\phi_2$};
\draw (-4.6,-1.5) node {\footnotesize$2$};
\draw (-3.4,-1.5) node {\footnotesize$2$};
\draw (-3.6,-1.8) node {\footnotesize$\phi_3$};
\draw (-4,-1.8) node {\footnotesize$3$};
\draw (-4.4,-1.8) node {\footnotesize$\phi_3$};

\draw (-5.7,-1.2) node {\footnotesize$\phi_1$};
\draw (-5.7,-1.5) node {\footnotesize$\phi_2$};
\draw (-6.1,-1.5) node {\footnotesize$\phi_2$};
\draw (-6.5,-1.5) node {\footnotesize$2$};
\draw (-5.5,-1.8) node {\footnotesize$\phi_3$};
\draw (-5.9,-1.8) node {\footnotesize$3$};
\draw (-6.3,-1.8) node {\footnotesize$\phi_3$};

\draw (-8,-1.35) node {\footnotesize$\phi_2$};
\draw (-8.2,-1.65) node {\footnotesize$\phi_3$};
\draw (-7.8,-1.65) node {\footnotesize$3$};

\draw (1,0) node {$\phi_1$};

\draw (-1,0) node {$\phi_2$};

\draw (-2.7,0.3) node {\footnotesize$1$};
\draw (-3.1,0.3) node {\footnotesize$\phi_1$};
\draw (-3.1,0) node {\footnotesize$\phi_2$};
\draw (-2.7,0) node {\footnotesize$2$};
\draw (-2.9,-0.3) node {\footnotesize$\phi_3$};
\draw (-3.3,-0.3) node {\footnotesize$3$};

\draw (-5.2,0.15) node {\footnotesize$\phi_2$};
\draw (-4.8,0.15) node {\footnotesize$2$};
\draw (-5,-0.15) node {\footnotesize$\phi_3$};

\draw (-7,0.3) node {\footnotesize$\phi_1$};
\draw (-7,0) node {\footnotesize$\phi_2$};
\draw (-7.2,-0.3) node {\footnotesize$3$};
\draw (-6.8,-0.3) node {\footnotesize$\phi_3$};

\draw (-1,1.8) node {\footnotesize$\phi_1$};
\draw (-1,1.5) node {\footnotesize$\phi_2$};
\draw (-1,1.2) node {\footnotesize$3$};

\draw (-2.8,1.8) node {\footnotesize$1$};
\draw (-3.2,1.5) node {\footnotesize$\phi_2$};
\draw (-2.8,1.5) node {\footnotesize$2$};
\draw (-3,1.2) node {\footnotesize$\phi_3$};

\draw (-5.1,1.8) node {\footnotesize$\phi_1$};
\draw (-5.1,1.5) node {\footnotesize$\phi_2$};
\draw (-4.7,1.5) node {\footnotesize$2$};
\draw (-4.9,1.2) node {\footnotesize$\phi_3$};
\draw (-5.3,1.2) node {\footnotesize$3$};

\draw (-7,1.65) node {\footnotesize$\phi_2$};
\draw (-7,1.35) node {\footnotesize$\phi_3$};

\draw (-9,1.5) node {$3$};

\draw [->](0.25,-5) -- (0.75,-5.5);
\draw [->](-0.75,-5.5) -- (-0.25,-5);
\draw [->](-1.75,-5) -- (-1.25,-5.5);
\draw [->](-2.75,-5.5) -- (-2.25,-5);
\draw [->](-3.75,-5) -- (-3.25,-5.5);
\draw [->](-4.75,-5.5) -- (-4.25,-5);
\draw [->](-5.75,-5) -- (-5.25,-5.5);
\draw [->](-6.75,-5.5) -- (-6.25,-5);
\draw [->](-7.75,-5) -- (-7.25,-5.5);

\draw [->](-0.75,-3.5) -- (-0.25,-4);
\draw [->](-1.75,-4) -- (-1.25,-3.5);
\draw [->](-2.75,-3.5) -- (-2.25,-4);
\draw [->](-3.75,-4) -- (-3.25,-3.5);
\draw [->](-4.75,-3.5) -- (-4.25,-4);
\draw [->](-5.75,-4) -- (-5.25,-3.5);
\draw [->](-6.75,-3.5) -- (-6.25,-4);
\draw [->](-7.75,-4) -- (-7.25,-3.5);
\draw [->](-8.75,-3.5) -- (-8.25,-4);

\draw [->](-0.75,-2.5) -- (-0.25,-2);
\draw [->](-1.75,-2) -- (-1.25,-2.5);
\draw [->](-2.75,-2.5) -- (-2.25,-2);
\draw [->](-3.75,-2) -- (-3.25,-2.5);
\draw [->](-4.75,-2.5) -- (-4.25,-2);
\draw [->](-5.75,-2) -- (-5.25,-2.5);
\draw [->](-6.75,-2.5) -- (-6.25,-2);
\draw [->](-7.75,-2) -- (-7.25,-2.5);
\draw [->](-8.75,-2.5) -- (-8.25,-2);

\draw [->](0.25,-1) -- (0.75,-0.5);
\draw [->](-0.75,-0.5) -- (-0.25,-1);
\draw [->](-1.75,-1) -- (-1.25,-0.5);
\draw [->](-2.75,-0.5) -- (-2.25,-1);
\draw [->](-3.75,-1) -- (-3.25,-0.5);
\draw [->](-4.75,-0.5) -- (-4.25,-1);
\draw [->](-5.75,-1) -- (-5.25,-0.5);
\draw [->](-6.75,-0.5) -- (-6.25,-1);
\draw [->](-7.75,-1) -- (-7.25,-0.5);

\draw [->](-0.75,1) -- (-0.1,-1);
\draw [->](-1.9,-1) -- (-1.25,1);
\draw [->](-2.75,1) -- (-2.1,-1);
\draw [->](-3.9,-1) -- (-3.25,1);
\draw [->](-4.75,1) -- (-4.1,-1);
\draw [->](-5.9,-1) -- (-5.25,1);
\draw [->](-6.75,1) -- (-6.1,-1);
\draw [->](-7.9,-1) -- (-7.25,1);
\draw [->](-8.75,1) -- (-8.1,-1);
\end{tikzpicture}
	\caption{The Auslander-Reiten quiver of $\mod* KQ^{\coxD_6}$ associated to the folding $Q^{\coxD_6} \rightarrow Q^{\coxH_3}$, where $Q^{\coxH_3}$ is linearly oriented.}
	\label{fig:D6ARQuiver}
\end{figure}

The lemma above allows us to prove Theorem~\ref{thm:Folding} for the cases where $\gend' \in \{\coxH_2=\coxI_2(5), \coxH_3, \coxH_4\}$.
\begin{proof}
	By Gabriel's Theorem, $\dimvect(M)$ is a positive root of $\gend$ for any indecomposable $KQ^\gend$-module $M$. It follows from Figures~\ref{fig:A4ARQuiver} and \ref{fig:D6ARQuiver} and \cite{Lusztig,MoodyPatera} that the set
	\begin{equation*}
		V_F = \{\dimproj_F( M) : M \text{ indecomposable } KQ^\gend\text{-module}\}
	\end{equation*}
	may be partitioned into two subsets $V_F^1$ and $V_F^\gratio$ such that $V_F^1$ is the set of positive roots of $\gend'$ and
	\begin{equation*}
		V_F^\gratio = \{\gratio v : v \in V_F^1\}.
	\end{equation*}

	We know that for each $i \in Q^\gend_0$, we have that $\dimproj_F(\tau^m I(i))$ is either in $V_F^1$ or $V_F^\gratio$. There is precisely one other vertex $j \in Q^\gend_0$ such that $F(i)=F(j)$, and by Lemma~\ref{lem:FoldingRadSeries}, the vector $\dimproj_F(\tau^m I(j))$ belongs to a different one of the sets $V_F^1$ or $V_F^\gratio$ compared to $\dimproj_F(\tau^m I(i))$. The result for (b) then follows.
	
	For (a), that we specifically have $\dimproj_F(M) \in V_F^1$ for any $M \in \mathcal{I}$ follows from (b).
\end{proof}

\subsection{Folding onto $\coxI$-type quivers} \label{sec:IFoldings}
Throughout this section, we will work with the folding $F\colon Q^{\coxA_{2n}} \rightarrow Q^{\coxI_2(2n+1)}$ defined in Remark~\ref{rem:IFolding}. Recall that the orientation of both $Q^{\coxA_{2n}}$ and $Q^{\coxI_2(2n+1)}$ is such that all arrows go from an even vertex to an odd vertex.

\begin{rem}
	A consequence of Lemma~\ref{lem:ChebProps}(c) is that the vertex weights of $Q^{\coxA_{2n}}$ are such that $\vw(k) = \vw(n-1-k)$.
\end{rem}

\begin{lem} \label{lem:InjProjIMultiples}
	Let $F\colon Q^{\coxA_{2n}} \rightarrow Q^{\coxI_2(2n+1)}$ be as above. Then for any $0\leq k,l \leq 2n-1$, such that $k$ and $l$ are either both even or both odd, we have
	\begin{enumerate}[label=(\alph*)]
		\item $\vw(l) \dimproj_F(I(k)) = \vw(k) \dimproj_F(I(l))$; and
		\item $\vw(l) \dimproj_F(P(k)) = \vw(k) \dimproj_F(P(l))$.
	\end{enumerate}
\end{lem}
\begin{proof}
	(a) Assume that $Q^{\coxA_{2n}}$ is oriented as above --- the proof for the opposite quiver is dual. For the purpose of readability, we will write $\theta_m = U_m(\cos \frac{\pi}{2n+1})$. It follows from the definitions, that we have
	\begin{equation*}
		\dimproj_F(I(k))=(\theta_k,0)
	\end{equation*}
	for $k$ even. In this case, it is obvious that (a) holds. We also have
	\begin{equation*}
		\dimproj_F(I(k)) =
		\begin{cases} 
			(\theta_{k-1}+\theta_{k+1},\theta_k)	& \text{if } k < 2n-1, \\
			(\theta_{2n-2},\theta_{2n-1}) & \text{if } k = 2n-1,
		\end{cases}
	\end{equation*}
	for $k$ odd, where by Lemma~\ref{lem:ChebProps}, we may write
	\begin{align*}
		(\theta_{k-1}+\theta_{k+1},\theta_k) &= (\theta_1\theta_k, \theta_k), \\
		(\theta_{2n-2},\theta_{2n-1}) &= (\theta_1,1).
	\end{align*}
	Since each $\vw(k)=\theta_k$, it is clear that (a) holds.

	(b) The proof is dual to (a).
\end{proof}

\begin{lem} \label{lem:IFoldingB}
	Let $F\colon Q^{\coxA_{2n}} \rightarrow Q^{\coxI_2(2n+1)}$ be as above. Then for any $i,j \in Q^{\coxA_{2n}}_0$ such that $F(i)=F(j)$, and for any $m \in \nnint$,
	\begin{align*}
		\vw(j) \dimproj_F(\tau^m I(i)) &= \vw(i) \dimproj_F(\tau^m I(j)).
	\end{align*}
\end{lem}
\begin{proof}
	The proof follows by induction. Assume the statement holds for a given $m \in \nnint$. The path algebra $KQ^{\coxA_{2n}}$ is radical square zero. Moreover for any indecomposable non-projective $M \in \mod*KQ^{\coxA_{2n}}$, we have
	\begin{equation*}
		\dimproj_F(\tp M) = (r,0) \qquad \text{and} \qquad \dimproj_F(\rad M) = (0,r')
	\end{equation*}
	for some $r,r' \in \chebr\ps{2n+1}$. It is therefore easy to see from Lemma~\ref{lem:InjProjIMultiples} that the projective presentations
	\begin{align*}
		P_1\ps{i} \rightarrow P_0\ps{i} \rightarrow \tau^m I(i) &\rightarrow 0 \\
		P_1\ps{j} \rightarrow P_0\ps{j} \rightarrow \tau^m I(j) &\rightarrow 0
	\end{align*}
	satisfy the property 
	\begin{equation*}
		\vw(j) \dimproj_F(P_k\ps{i}) = \vw(i) \dimproj_F(P_k\ps{j})
	\end{equation*}
	for each $k$. It is also straightforward to deduce from Lemma~\ref{lem:InjProjIMultiples} that with the exact sequences
	\begin{align*}
		0 \rightarrow \tau^{m+1} I(i) &\rightarrow \nu P_1\ps{i} \rightarrow \nu P_0\ps{i} \rightarrow \nu \tau^m I(i) \rightarrow 0 \\
		0 \rightarrow \tau^{m+1} I(j) &\rightarrow \nu P_1\ps{j} \rightarrow \nu P_0\ps{j} \rightarrow \nu \tau^m I(j) \rightarrow 0
	\end{align*}
	where $\nu$ is the Nakayama functor, we have
	\begin{equation*}
		\vw(j) \dimproj_F(\nu P_k\ps{i}) = \vw(i) \dimproj_F(\nu P_k\ps{j})
	\end{equation*}
	and thus,
	\begin{equation*}
		\vw(j) \dimproj_F(\tau^{m+1} I(i)) = \vw(i)\dimproj_F(\tau^{m+1} I(j)).
	\end{equation*}
	Lemma~\ref{lem:InjProjIMultiples} then completes the proof.
\end{proof}

\begin{lem} \label{lem:IRootsOfUnity}
	Let $F\colon Q^{\coxA_{2n}} \rightarrow Q^{\coxI_2(2n+1)}$ be as above and denote $\theta = \cos \frac{\pi}{2n+1}$. Let $e_F$ be the linear change-of-basis transformation from the basis of positive simple roots of $\coxI_2(2n+1)$ to the standard basis over $\real$. That is, 
	\begin{align*}
		e_F(1,0) &= (1,0)	&	e_F(0,1) &= (\cos 2n\theta, \sin 2n\theta).
	\end{align*}
	Then
	\begin{align*}
		e_F \dimproj_F(\tau^m I(0)) &= (\cos 2m\theta,\sin 2m\theta) & &(0 \leq m \leq n), \\
		e_F \dimproj_F(\tau^m I(2n-1)) &= (\cos (2m+1)\theta,\sin (2m+1)\theta)  & &(0 \leq m < n).
		\end{align*}
\end{lem}
\begin{proof}
	The results follow from trigonometric identities and known results on the Auslander-Reiten theory of quivers of type $\coxA_{2n}$ (in particular, see \cite{ButlerRingel} for results on the Auslander-Reiten translate). We will assume the orientation of $Q^{\coxA_{2n}}$ given at the start of this subsection --- the proof for the opposite quiver is dual. For the purpose of readability, we will also write
	\begin{equation*}
		\theta_k = U_k(\cos \theta)=\frac{\sin (k+1)\theta}{\sin \theta}.
	\end{equation*}
	
	(a) The row of the Auslander-Reiten quiver containing the injective module $I(0)$ is such that
	\begin{equation*}
		\tp \tau^m I(0) = S(2m) \qquad \text{and} \qquad \rad \tau^m I(0) = S(2m-1)
	\end{equation*}
	for any $0 < m < n$, and $I(0) \cong S(0)$ and $\tau^{n}I(0) \cong S(2n-1)$. Thus, we have 
	\begin{equation*}
		\dimproj_F(\tau^m I(0)) = 
		\begin{cases}
			(1,0)								&	m=0, \\
			(\theta_{2m},\theta_{2m-1})	&	0 < m < n, \\
			(0,1)								&	m=n,
		\end{cases}
	\end{equation*}
	and hence,
	\begin{equation*}
		e_F \dimproj_F(\tau^m I(0)) =
		\begin{cases}
			(1,0)								
			&	m=0, \\
			(\theta_{2m} + \theta_{2m-1} \cos 2n\theta,\theta_{2m-1}\sin 2n\theta)	
			&	0 < m < n, \\
			(\cos 2n\theta,\sin 2n\theta)
			&	m=n.
		\end{cases}
	\end{equation*}
	The result clearly holds for $m\in\{0,n\}$, so we need only consider $0 < m < n$. By the angle summation identity
	\begin{equation*}
		\sin \theta \cos 2m\theta = \sin(2m+1)\theta - \sin 2m\theta \cos \theta = \sin(2m+1)\theta + \sin 2m\theta \cos 2n\theta,
	\end{equation*}
	we obtain
	\begin{equation*}
		\cos 2m\theta =\theta_{2m}+ \theta_{2m-1} \cos 2n\theta.
	\end{equation*}
	Trivially, we also have
	\begin{equation*}
		\sin 2m \theta=\frac{\sin 2m\theta}{\sin\theta}\sin 2n\theta
	\end{equation*}
	So $e_F \dimproj_F(\tau^m I(0)) = (\cos 2m \theta,\sin 2m \theta)$, as required.
	
	On the other hand, the row of the Auslander-Reiten quiver containing the injective module $I(2n-1)$ is such that
	\begin{align*}
		\dimproj_F(\tau^m I(2n-1)) &= (\theta_{2n-2m-2},\theta_{2n-2m-1}) \\
		&=(\theta_{2m+1},\theta_{2m})
	\end{align*}
	for any $0 \leq m < n$. By similar arguments to the above, the result follows from the identities
	\begin{align*}
		\sin \theta \cos (2m+1)\theta &= \sin(2m+2)\theta + \sin (2m+1)\theta \cos 2n\theta, \\
		\sin (2m+1) \theta&=\frac{\sin (2m+1)\theta}{\sin\theta}\sin 2n\theta.
	\end{align*}
\end{proof}

We can now prove Theorem~\ref{thm:Folding} for foldings onto $\coxI$-type quivers.
\begin{proof}
	(a) That $\dimproj_F(M)$ is a positive root of $\coxI_2(2n+1)$ for any $M \in \mathcal{I}\ps{0}$ follows from Lemma~\ref{lem:IRootsOfUnity}.
	
	(b) This is a direct consequence of Lemma~\ref{lem:IFoldingB}.
	
	(c) This follows directly from (a) and (b).
\end{proof}

	\begin{exam} \label{ex:I7}
	Consider the folding $F\colon Q^{\coxA_6} \rightarrow Q^{\coxI_2(7)}$ of the form
	\begin{center}
		\begin{tikzpicture}
\draw (-0.7,0) node {$F\colon$};

\draw (0,0.8) node {$0$};
\draw (1.2,0.8) node {$1$};
\draw (0,0) node {$2$};
\draw (1.2,0) node {$3$};
\draw (0,-0.8) node {$4$};
\draw (1.2,-0.8) node {$5$};

\draw (1.8,0) node {$\rightarrow$};

\draw (2.4,0) node {$[0]$};
\draw (4,0) node {$[1]$};
\draw (3.2,0.2) node {\footnotesize$2 \cos \frac{\pi}{7}$};
\draw [->](0.3,0.8) -- (0.9,0.8);
\draw [->](0.3,0) -- (0.9,0);
\draw [->](0.3,-0.8) -- (0.9,-0.8);
\draw [->](0.3,0.2) -- (0.9,0.6);
\draw [->](0.3,-0.6) -- (0.9,-0.2);

\draw [->](2.7,0) -- (3.7,0);
\end{tikzpicture}
	\end{center}
	where $F(0)=F(2)=F(4)=[0]$, $F(1)=F(3)=F(5)=[1]$, and $\vw(i) = U_i(\cos \frac{\pi}{7})$ for each $i$. Explicitly, we have
	\begin{align*}
		\vw(0)&=1,			&	\vw(1)&= 2x, \\
		\vw(2)&=4x^2-1,		&	\vw(3)&= 4x^2-1, \\
		\vw(4)&=2x,			&	\vw(5)&= 1,
	\end{align*}
	where $x = \cos \frac{\pi}{7}$. The module category $KQ^{\coxA_6}$ then maps to the $\coxI_2(7)$ root lattice via the function $\dimproj_F$ as shown in Figure~\ref{fig:I7AR}
\end{exam}

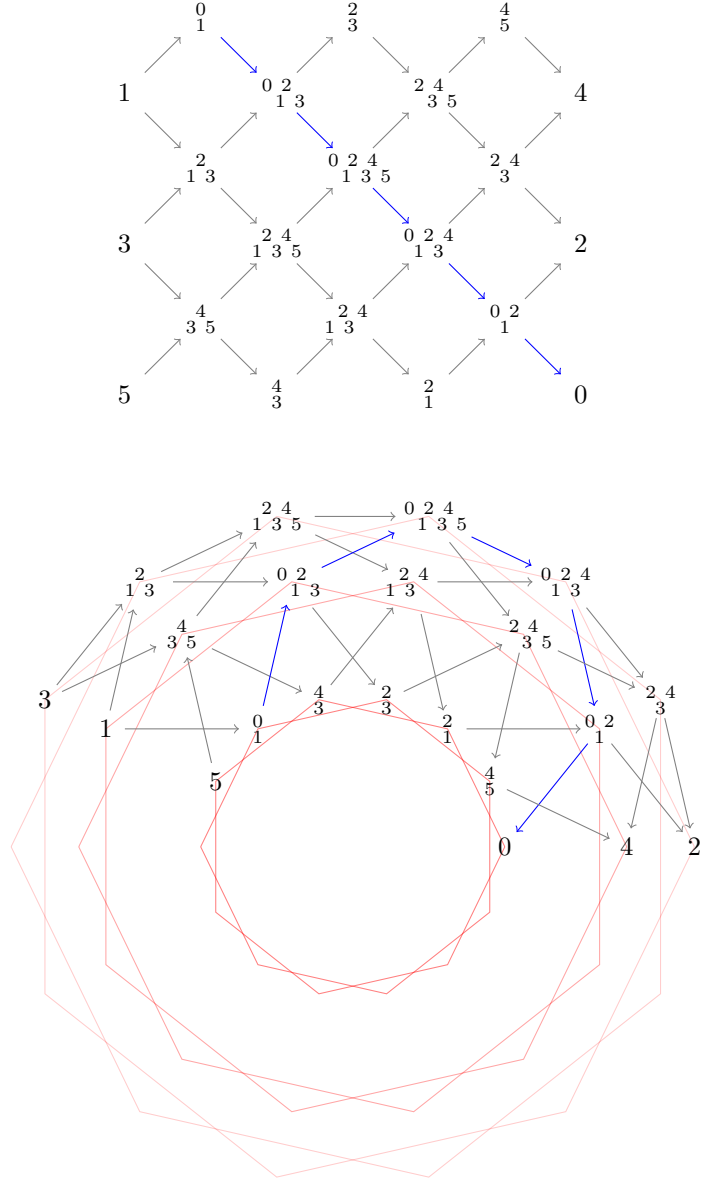
\begin{figure}[h] 
	\usetikzlibrary{calc}
\def\y{0.25cm}
\begin{tikzpicture}[scale=2]
	\foreach \x in {0,...,3} {
		\foreach \z in {0,2,4} {
			\coordinate (t\x_I\z) at ($(1.5-\x,{\z/2 + 3})$);
		}
	}
	\foreach \x in {0,...,2} {
		\foreach \z in {1,3,5} {
			\coordinate (t\x_I\z) at ($(1-\x,{((\z - 1)/2 + 3.5)})$);
		}
	}
	
	\draw[<-, shorten <= \y*1.5,shorten >= \y*1.5,blue] (t0_I0) -- (t0_I1);
	\draw[<-, shorten <= \y*1.5,shorten >= \y*1.5,gray] (t0_I2) -- (t0_I1);
	\draw[<-, shorten <= \y*1.5,shorten >= \y*1.5,gray] (t0_I2) -- (t0_I3);
	\draw[<-, shorten <= \y*1.5,shorten >= \y*1.5,gray] (t0_I4) -- (t0_I3);
	\draw[<-, shorten <= \y*1.5,shorten >= \y*1.5,gray] (t0_I4) -- (t0_I5);
	
	\draw[->, shorten <= \y*1.5,shorten >= \y*1.5,gray] (t1_I0) -- (t0_I1);
	\draw[->, shorten <= \y*1.5,shorten >= \y*1.5,blue] (t1_I2) -- (t0_I1);
	\draw[->, shorten <= \y*1.5,shorten >= \y*1.5,gray] (t1_I2) -- (t0_I3);
	\draw[->, shorten <= \y*1.5,shorten >= \y*1.5,gray] (t1_I4) -- (t0_I3);
	\draw[->, shorten <= \y*1.5,shorten >= \y*1.5,gray] (t1_I4) -- (t0_I5);
	
	\draw[<-, shorten <= \y*1.5,shorten >= \y*1.5,gray] (t1_I0) -- (t1_I1);
	\draw[<-, shorten <= \y*1.5,shorten >= \y*1.5,gray] (t1_I2) -- (t1_I1);
	\draw[<-, shorten <= \y*1.5,shorten >= \y*1.5,blue] (t1_I2) -- (t1_I3);
	\draw[<-, shorten <= \y*1.5,shorten >= \y*1.5,gray] (t1_I4) -- (t1_I3);
	\draw[<-, shorten <= \y*1.5,shorten >= \y*1.5,gray] (t1_I4) -- (t1_I5);
	
	\draw[->, shorten <= \y*1.5,shorten >= \y*1.5,gray] (t2_I0) -- (t1_I1);
	\draw[->, shorten <= \y*1.5,shorten >= \y*1.5,gray] (t2_I2) -- (t1_I1);
	\draw[->, shorten <= \y*1.5,shorten >= \y*1.5,gray] (t2_I2) -- (t1_I3);
	\draw[->, shorten <= \y*1.5,shorten >= \y*1.5,blue] (t2_I4) -- (t1_I3);
	\draw[->, shorten <= \y*1.5,shorten >= \y*1.5,gray] (t2_I4) -- (t1_I5);
	
	\draw[<-, shorten <= \y*1.5,shorten >= \y*1.5,gray] (t2_I0) -- (t2_I1);
	\draw[<-, shorten <= \y*1.5,shorten >= \y*1.5,gray] (t2_I2) -- (t2_I1);
	\draw[<-, shorten <= \y*1.5,shorten >= \y*1.5,gray] (t2_I2) -- (t2_I3);
	\draw[<-, shorten <= \y*1.5,shorten >= \y*1.5,gray] (t2_I4) -- (t2_I3);
	\draw[<-, shorten <= \y*1.5,shorten >= \y*1.5,blue] (t2_I4) -- (t2_I5);
	
	\draw[->, shorten <= \y*1.5,shorten >= \y*1.5,gray] (t3_I0) -- (t2_I1);
	\draw[->, shorten <= \y*1.5,shorten >= \y*1.5,gray] (t3_I2) -- (t2_I1);
	\draw[->, shorten <= \y*1.5,shorten >= \y*1.5,gray] (t3_I2) -- (t2_I3);
	\draw[->, shorten <= \y*1.5,shorten >= \y*1.5,gray] (t3_I4) -- (t2_I3);
	\draw[->, shorten <= \y*1.5,shorten >= \y*1.5,gray] (t3_I4) -- (t2_I5);
	
	\draw (t0_I0) node {\footnotesize$0$};
	\draw (t1_I0) node {\footnotesize$\begin{smallmatrix} 2 \\ 1  \end{smallmatrix}$};
	\draw (t2_I0) node {\footnotesize$\begin{smallmatrix} 4 \\ 3  \end{smallmatrix}$};
	\draw (t3_I0) node {\footnotesize$5$};
	
	\draw (t0_I1) node {\footnotesize$\begin{smallmatrix} 0 \ 2 \\ \quad 1 \quad \end{smallmatrix}$};
	\draw (t1_I1) node {\footnotesize$\begin{smallmatrix} 2 \ 4 \\ 1 \ 3 \quad \end{smallmatrix}$};
	\draw (t2_I1) node {\footnotesize$\begin{smallmatrix} \quad 4 \quad \\ 3 \  5 \end{smallmatrix}$};
	
	\draw (t0_I2) node {\footnotesize$2$};
	\draw (t1_I2) node {\footnotesize$\begin{smallmatrix} 0 \  2 \  4 \\ 1 \  3 \end{smallmatrix}$};
	\draw (t2_I2) node {\footnotesize$\begin{smallmatrix} 2 \  4 \\ 1 \  3 \  5 \end{smallmatrix}$};
	\draw (t3_I2) node {\footnotesize$3$};
	
	\draw (t0_I3) node {\footnotesize$\begin{smallmatrix} 2 \ 4 \\ 3 \end{smallmatrix}$};
	\draw (t1_I3) node {\footnotesize$\begin{smallmatrix} 0 \ 2 \ 4 \\ \quad 1 \ 3 \ 5 \end{smallmatrix}$};
	\draw (t2_I3) node {\footnotesize$\begin{smallmatrix} \quad 2 \quad \\ 1 \ 3 \end{smallmatrix}$};
	
	\draw (t0_I4) node {\footnotesize$4$};
	\draw (t1_I4) node {\footnotesize$\begin{smallmatrix} 2 \  4 \\ \quad 3 \  5 \end{smallmatrix}$};
	\draw (t2_I4) node {\footnotesize$\begin{smallmatrix} 0 \  2 \\ \quad 1 \  3 \end{smallmatrix}$};
	\draw (t3_I4) node {\footnotesize$1$};
	
	\draw (t0_I5) node {\footnotesize$\begin{smallmatrix} 4 \\  5 \end{smallmatrix}$};
	\draw (t1_I5) node {\footnotesize$\begin{smallmatrix} 2 \\  3 \end{smallmatrix}$};
	\draw (t2_I5) node {\footnotesize$\begin{smallmatrix} 0 \\  1 \end{smallmatrix}$};

	\foreach \x in {1,...,7} {
		\coordinate (p\x) at ($({cos((180/7)*\x-(180/7))}, {sin((180/7)*\x-(180/7))})$);
		\coordinate (u1p\x) at ($({2*cos((180/7))*cos((180/7)*\x-(180/7))}, {2*cos((180/7))*sin((180/7)*\x-(180/7))})$);
		\coordinate (u2p\x) at ($({(4*(cos((180/7)))^2-1)*cos((180/7)*\x-(180/7))}, {(4*(cos((180/7)))^2-1)*sin((180/7)*\x-(180/7))})$);
	}
	\foreach \x in {1,...,7} {
		\coordinate (n\x) at ($({cos((180/7)*\x+(6*180/7))}, {sin((180/7)*\x+(6*180/7))})$);
		\coordinate (u1n\x) at ($({2*cos((180/7))*cos((180/7)*\x+(6*180/7))}, {2*cos((180/7))*sin((180/7)*\x+(6*180/7))})$);
		\coordinate (u2n\x) at ($({(4*(cos((180/7)))^2-1)*cos((180/7)*\x+(6*180/7))}, {(4*(cos((180/7)))^2-1)*sin((180/7)*\x+(6*180/7))})$);
	}
	\draw[red,opacity=0.5] (p1) -- (p3)--(p5) --(p7)-- (n2) -- (n4) -- (n6) -- (p1);
	\draw[red,opacity=0.5] (n1) -- (n3)--(n5) -- (n7) -- (p2) -- (p4) --(p6)-- (n1);
	\draw[red,opacity=0.35] (u1p1) -- (u1p3)--(u1p5)--(u1p7) -- (u1n2) -- (u1n4) -- (u1n6) -- (u1p1);
	\draw[red,opacity=0.35] (u1n1) -- (u1n3)--(u1n5)--(u1n7) -- (u1p2) -- (u1p4) -- (u1p6) -- (u1n1);
	\draw[red,opacity=0.2] (u2p1) -- (u2p3)--(u2p5)--(u2p7) -- (u2n2) -- (u2n4) -- (u2n6) -- (u2p1);
	\draw[red,opacity=0.2] (u2n1) -- (u2n3)--(u2n5)--(u2n7) -- (u2p2) -- (u2p4) -- (u2p6) -- (u2n1);
	
	\draw (p1) node {\footnotesize$0$};
	\draw (p3) node {\footnotesize$\begin{smallmatrix} 2 \\ 1  \end{smallmatrix}$};
	\draw (p5) node {\footnotesize$\begin{smallmatrix} 4 \\ 3  \end{smallmatrix}$};
	\draw (p7) node {\footnotesize$5$};
	
	\draw (u1p2) node {\footnotesize$\begin{smallmatrix} 0 \ 2 \\ \quad 1 \quad \end{smallmatrix}$};
	\draw (u1p4) node {\footnotesize$\begin{smallmatrix} 2 \ 4 \\ 1 \ 3 \quad \end{smallmatrix}$};
	\draw (u1p6) node {\footnotesize$\begin{smallmatrix} \quad 4 \quad \\ 3 \  5 \end{smallmatrix}$};
	
	\draw (u2p1) node {\footnotesize$2$};
	\draw (u2p3) node {\footnotesize$\begin{smallmatrix} 0 \  2 \  4 \\ 1 \  3 \end{smallmatrix}$};
	\draw (u2p5) node {\footnotesize$\begin{smallmatrix} 2 \  4 \\ 1 \  3 \  5 \end{smallmatrix}$};
	\draw (u2p7) node {\footnotesize$3$};
	
	\draw (u2p2) node {\footnotesize$\begin{smallmatrix} 2 \ 4 \\ 3 \end{smallmatrix}$};
	\draw (u2p4) node {\footnotesize$\begin{smallmatrix} 0 \ 2 \ 4 \\ \quad 1 \ 3 \ 5 \end{smallmatrix}$};
	\draw (u2p6) node {\footnotesize$\begin{smallmatrix} \quad 2 \quad \\ 1 \ 3 \end{smallmatrix}$};
	
	\draw (u1p1) node {\footnotesize$4$};
	\draw (u1p3) node {\footnotesize$\begin{smallmatrix} 2 \  4 \\ \quad 3 \  5 \end{smallmatrix}$};
	\draw (u1p5) node {\footnotesize$\begin{smallmatrix} 0 \  2 \\ \quad 1 \  3 \end{smallmatrix}$};
	\draw (u1p7) node {\footnotesize$1$};
	
	\draw (p2) node {\footnotesize$\begin{smallmatrix} 4 \\  5 \end{smallmatrix}$};
	\draw (p4) node {\footnotesize$\begin{smallmatrix} 2 \\  3 \end{smallmatrix}$};
	\draw (p6) node {\footnotesize$\begin{smallmatrix} 0 \\  1 \end{smallmatrix}$};
	
	\draw[<-, shorten <= \y,shorten >= \y,blue] (p1) -- (u1p2);
	\draw[<-, shorten <= \y,shorten >= \y,gray] (u1p1) -- (u2p2);
	\draw[<-, shorten <= \y,shorten >= \y,gray] (u2p1) -- (u2p2);
	\draw[<-, shorten <= \y,shorten >= \y,gray] (u2p1) -- (u1p2);
	\draw[<-, shorten <= \y,shorten >= \y,gray] (u1p1) -- (p2);
	
	\draw[<-, shorten <= \y*1.25,shorten >= \y,gray] (p2) -- (u1p3);
	\draw[<-, shorten <= \y*1.25,shorten >= \y*1.5,blue] (u1p2) -- (u2p3);
	\draw[<-, shorten <= \y*1.5,shorten >= \y*1.75,gray] (u2p2) -- (u2p3);
	\draw[<-, shorten <= \y*1.5,shorten >= \y*2,gray] (u2p2) -- (u1p3);
	\draw[<-, shorten <= \y,shorten >= \y,gray] (u1p2) -- (p3);
	
	\draw[<-, shorten <= \y,shorten >= \y*1.75,gray] (p3) -- (u1p4);
	\draw[<-, shorten <= \y*1.25,shorten >= \y*1.75,gray] (u1p3) -- (u2p4);
	\draw[<-, shorten <= \y*2,shorten >= \y*2.5,blue] (u2p3) -- (u2p4);
	\draw[<-, shorten <= \y*1.75,shorten >= \y*1.25,gray] (u2p3) -- (u1p4);
	\draw[<-, shorten <= \y*1.25,shorten >= \y,gray] (u1p3) -- (p4);
	
	\draw[<-, shorten <= \y,shorten >= \y*1.75,gray] (p4) -- (u1p5);
	\draw[<-, shorten <= \y*1.5,shorten >= \y*2.25,gray] (u1p4) -- (u2p5);
	\draw[<-, shorten <= \y*1.75,shorten >= \y*2,gray] (u2p4) -- (u2p5);
	\draw[<-, shorten <= \y*2,shorten >= \y*1.75,blue] (u2p4) -- (u1p5);
	\draw[<-, shorten <= \y*1.75,shorten >= \y,gray] (u1p4) -- (p5);
	
	\draw[<-, shorten <= \y,shorten >= \y*1.75,gray] (p5) -- (u1p6);
	\draw[<-, shorten <= \y*1.25,shorten >= \y*1.75,gray] (u1p5) -- (u2p6);
	\draw[<-, shorten <= \y*2,shorten >= \y*1.25,gray] (u2p5) -- (u2p6);
	\draw[<-, shorten <= \y*1.75,shorten >= \y*1.25,gray] (u2p5) -- (u1p6);
	\draw[<-, shorten <= \y*1.25,shorten >= \y*1.25,blue] (u1p5) -- (p6);
	
	\draw[<-, shorten <= \y,shorten >= \y,gray] (p6) -- (u1p7);
	\draw[<-, shorten <= \y*1.5,shorten >= \y,gray] (u1p6) -- (u2p7);
	\draw[<-, shorten <= \y*1.5,shorten >= \y,gray] (u2p6) -- (u2p7);
	\draw[<-, shorten <= \y*1.5,shorten >= \y,gray] (u2p6) -- (u1p7);
	\draw[<-, shorten <= \y*1.25,shorten >= \y,gray] (u1p6) -- (p7);
\end{tikzpicture}
	\caption{Above: The Auslander-Reiten quiver of the module category $\mod* KQ^{\coxA_6}$, with the ray of source $P(0)$ highlighted blue. Below: The projection of $\mod*KQ^{\coxA_6}$ onto $\real^2$ via the function $\dimproj_F$. The vertices of the innermost heptagons are the roots of $\coxI_2(7)$, with the positive roots being those that are labelled. The vertices of the middle pair of heptagons are the $2\cos\frac{\pi}{7}$ multiples of the root vectors, and the outermost vertices are the $(4\cos^2\frac{\pi}{7} -1)$ multiples.} \label{fig:I7AR}
\end{figure}

\subsection{Projections of derived categories}
Theorem~\ref{thm:Folding} naturally extends from the module category to the bounded derived category. Recall that the bounded derived category $\der_\gend=\bder(KQ^\gend)$ of the module category $\calM_\gend=\mod*KQ^\gend$ is a triangulated category with shift functor $\sus$, where as before, $\gend \in \{\coxA_{2n}, \coxD_6,\coxE_8\}$. The indecomposable objects of $\der_\gend$ are isomorphic to chain complexes with an indecomposable $KQ^\gend$-module in some degree $i$ and 0 in degree other than $i$. By a slight abuse of notation, for each indecomposable module $M \in \calM_\gend$ we denote by $M \in \der_\gend$ the indecomposable object corresponding to $M$ in degree 0. The object $\sus^i M$ is then the indecomposable object corresponding to $M$ in degree $i$. The morphisms of $\der_\gend$ are given by
\begin{equation*}
	\Hom_{\der_\gend}(\sus^i M, \sus^j N) =
	\begin{cases}
		\Ext^{j-i}_{\calM_\gend}(M,N)	&	\text{if } i \leq j, \\
		0									&	\text{otherwise}
	\end{cases}
\end{equation*}
and a composition of morphisms is given by the Yoneda product
\begin{equation*}
	\Ext^{j-i}_{\calM_\gend}(M,N) \times \Ext^{k-j}_{\calM_\gend}(N,L) \rightarrow \Ext^{k-i}_{\calM_\gend}(M,L).
\end{equation*}
The category $\der_\gend$ also has Aulsander-Reiten triangles and an Auslander-Reiten quiver with Auslander-Reiten translate $\dgart$.

The following statements are easy consequences of the known structure of $\der_\gend$, where $\gend \in \{\coxA_{2n},\coxD_6,\coxE_8\}$.

\begin{lem} \label{lem:ProjDerTranslates}
	Let $F\colon Q^\gend \rightarrow Q^{\gend'}$ be a weighted folding of quivers with $\gend \in \{\coxA_{2n},\coxD_6,\coxE_8\}$ and $\gend' \in \{\coxI_2(2n+1),\coxH_3,\coxH_4\}$. Then for any $i \in Q^\gend_0$, we have
	\begin{equation*}
		\dgart \sus^k P(i) \cong \sus^{k-1} I(j)
	\end{equation*}
	such that $\vw(j) = \vw(i)$.
\end{lem}
\begin{proof}
	For foldings $F\colon Q^{\coxA_{2n}} \rightarrow Q^{\coxI_2(2n+1)}$ with vertices labelled as in Subsection~\ref{sec:IFoldings}, it is known that
	\begin{equation*}
		\dgart \sus^k P(j) \cong \sus^{k-1}I(2n-1-j).
	\end{equation*}
	We know from Lemma~\ref{lem:ChebProps}(d) that $\vw(j)=\vw(2n-1-j)$. For foldings onto quivers of type $\coxH_3$ and $\coxH_4$, the result is straightforward, as
	\begin{equation*}
		\dgart \sus^k P(i) \cong \sus^{k-1} I(i)
	\end{equation*}
	for any $i \in Q^\gend_0$.
\end{proof}

One can extend the projection map $\dimproj_F$ to the derived category with the map
\begin{equation*}
	\derdim_F\colon\Ob(\der_\gend) \rightarrow \wh\chebr\ps{2n+1}
\end{equation*}
defined for each object $M= \oplus_{i \in \integer} \sus^i M_i \in \der_\gend$ with $M_i$ in degree 0 by
\begin{equation*}
	\derdim_F(M) = \sum_{2i \in \integer} \dimproj_F(M_i) - \sum_{2i+1 \in \integer} \dimproj_F(M_i).
\end{equation*}

\begin{prop} \label{prop:DerivedProj}
	Let $F\colon Q^\gend \rightarrow Q^{\gend'}$ be a weighted folding of quivers with $\gend \in \{\coxA_{2n},\coxD_6,\coxE_8\}$ and $\gend' \in \{\coxI_2(2n+1),\coxH_3,\coxH_4\}$.  Let $\sus^k M, \sus^k N \in \der_\gend$ be indecomposable objects for some $k \in \integer$ such that $\sus^k M \cong \dgart^m I(i)$ and $\sus^k N \cong \dgart^m I(j)$ for some $m \in \integer$ and some $i,j \in Q^\gend_0$ such that $F(i)=F(j)$. Suppose that $\vw(i) =1$. Then
	\begin{enumerate}[label=(\alph*)]
		\item $\derdim_F(\sus^{k} M)$ is a positive root of $\gend'$ if $k$ is even and is a negative root of $\gend'$ if $k$ is odd;
		\item $\vw(i) \derdim_F(\sus^{k} N) = \vw(j) \derdim_F(\sus^k M)$.
	\end{enumerate}
\end{prop}
\begin{proof}
	This follows directly from Theorem~\ref{thm:Folding}, Lemma~\ref{lem:ProjDerTranslates} and the definition of $\derdim_F$.
\end{proof}

\section{Semiring actions on categories associated to unfolded quivers}
\label{section-action}

Given a folding $F\colon Q^\gend \rightarrow Q^{\gend'}$ with $\gend \in \{\coxA_{2n},\coxD_6,\coxE_8\}$ and $\gend' \in \{\coxI_2(2n+1),\coxH_3,\coxH_4\}$, the projection map $\dimproj_F$ hints at a nice semiring action on $\mod*KQ^\gend$ and its bounded derived category, which we explore in this section. 

\subsection{$R_+$-coefficient categories}

First we will define how we want a semiring $R_+$ to act on the categories.

\begin{defn} \label{def:RPCoeffCat}
	Let $R_+$ be a partially ordered commutative semiring and $\calM$ be a Krull-Schmidt, $K$-linear, abelian or triangulated category. We say $\calM$ has the structure of an \emph{$R_+$-coefficient category} (or equivalently, has a weak semiring action of $R_+$) if it is equipped with a family of endofunctors $\{\wh{r} : r \in R_+\}$ such that the following axioms hold for any $r,s \in R_+$:
	\begin{enumerate}[label=(M\arabic*)]
		\item $\wh{r+s} \simeq \wh{r} \oplus \wh{s}$.
		\item $\wh{rs} \simeq \wh{r}\wh{s}$.
		\item $\wh{1}_{R_+} \simeq 1_\calM$.
		\item $\wh{0}_{R_+} \simeq 0_\calM$.
		\item $\wh{r}$ is exact (if $\calM$ is abelian) or triangulated (if $\calM$ is triangulated).
		\item $\wh{r}$ is faithful and $K$-linear on all $\Ext_\calM$-spaces. That is, for any $i \in \integer$ and any objects $X,Y \in \calM$, the element $r \in R_+$ induces an injective linear map
		\begin{equation*}
			\varepsilon_{r,X,Y}\ps{i}\colon \Ext_\calM^i(X,Y) \rightarrow \Ext_\calM^i(\wh{r}X,\wh{r}Y).
		\end{equation*}
	\end{enumerate}
\end{defn}

Given any morphism $f \in \Hom_\calM(X,Y)$ in an $R_+$-coefficient category, the (weak) semiring action of $R_+$ on $\calM$ is well-defined up to natural isomorphism. Therefore for any morphism $f \in \Hom_\calM(X,Y)$, we often abuse notation and write $\wh{r} f \in \Hom_\calM(\wh{r}X,\wh{r}Y)$ as $r f \in \Hom_\calM(rX,rY)$ for each $r \in R_+$. Thus, the objects and morphisms will appear to have coefficients in $R_+$, hence the name.

An immediate consequence to the axioms above is the following.
\begin{rem} \label{rem:RCoeffConseq}
	Let $\calM$ be an $R_+$-coefficient category and let $R$ be the corresponding ring given by the Grothendieck group construction on the additive commutative monoid structure of $R_+$. Then the Grothendieck group $G_0(\calM)$ of $\calM$ has the structure of an $R$-module via the action $\pm r [M] = \pm[rM]$ for each $r \in R_+$ and $[M] \in G_0(\calM)$.
\end{rem}

 The main aim of this section is to prove the following result.

\begin{thm} \label{thm:Action}
	Let $F\colon Q^\gend \rightarrow Q^{\gend'}$ be a folding of quivers with $\gend' \in \{\coxH_3,\coxH_4, \coxI_2(2n+1)\}$. Then $\calM_F=\mod*KQ^\gend$ and $\der_F=\bder(KQ^\gend)$ are $\chebsr\ps{2n+1}$-coefficient categories, where $n=2$ if $\gend' \in \{\coxH_3,\coxH_4\}$.
\end{thm}

To prove the above, we will define how $\chebsr\ps{2n+1}$ acts on the category $\calM_F=\mod*KQ^\gend$. Firstly, we will describe the principle of how we want $\chebsr\ps{2n+1}$ to act on objects. Let $\alpha$ be a positive root of $\gend'$. Then by Theorem~\ref{thm:Folding}, there exists a set
\begin{equation*}
	\mathbf{M}_{\alpha} = \{M\ps{0}_{\alpha}, \ldots, M\ps{n-1}_{\alpha}\}
\end{equation*}
of pairwise non-isomorphic indecomposable objects of $\calM_F$ such that
\begin{equation*}
	\dimproj_F(M\ps{k}_{\alpha}) = U_k\left(\cos\frac{\pi}{2n+1}\right) \alpha.
\end{equation*}
Moreover, for any positive root $\beta\neq \alpha$, the elements of $\mathbf{M}_\alpha \cap \mathbf{M}_\beta = \emptyset$. Given $0 \leq l \leq k \leq n-1$, we will define the action of $\croot_k \in \chebsr\ps{2n+1}$ on $\calM_F$ such that 
\begin{equation} \tag{$\ast$} \label{eq:ActOnObj}
	\croot_k M\ps{l}_\alpha 
	\cong \croot_l M\ps{k}_\alpha 
	\cong (\croot_k \croot_l) M\ps{0}_\alpha 
	\cong \left(\sum_{j=0}^l \croot_{k-l+2j} \right) M\ps{0}_\alpha 
	\cong \bigoplus_{j=0}^l M\ps{k-l+2j}_\alpha.
\end{equation}
That is, the action of $\croot_k$ on a module $M_\alpha\ps{l}$ respects the multiplication rule for Chebyshev polynomials of the second kind.

Describing the action on morphisms of $\calM_F$ requires us to be more explicit. So next, let us establish the following notation. Let $\{\eta_i : i \in Q^\gend_0\}$ be a complete set of primitive orthogonal idempotents of $KQ^\gend$ and let $M$ be a $KQ^\gend$-module. Let $a\colon i\rightarrow j \in Q^\gend_1$ and let $M_{a}\colon M_i \rightarrow M_j$ be the restriction of the linear action of $a$ on $M$ to the vector subspace $M_i=M\eta_i$. The data given by all vector subspaces $M_i$ and all linear maps $M_a$ entirely determine the structure of $M$ --- this is equivalently the module $M$ expressed as a $K$-representation of $Q^\gend$, which in this setting, is more convenient to work with. For any $N\in {\calM_F}$ and morphism $f \in \Hom_{\calM_F}(M,N)$, we then recall by Schur's lemma that $f$ can be written diagonally as $(f_i)_{i \in Q^\gend_0}$ with $f_i=f|_{M_i} \colon M_i \rightarrow N_i$.

For $\coxH_m$-type quivers, recall that $\chebsr\ps{5}\cong\goldsr$ and $Q_0 = \{i, \phi_i : 1 \leq i \leq m\}$. Throughout, the reader may notice that the relation $\gratio^2 = \gratio + 1$ comes into play. We define the object $\gratio M \in \calM_F$ to be the $KQ^\gend$-module with the following structure. We have vector subspaces
\begin{equation*}
	(\gratio M)_i = M_{\phi_i} \qquad \text{and} \qquad (\gratio M)_{\phi_i} = M_{\phi_i} \oplus M_i.
\end{equation*}
For each arrow $a \colon i \rightarrow j$ and corresponding arrow $b \colon \phi_i \rightarrow \phi_j$, we have
\begin{equation*}
	(\gratio M)_a = M_b \qquad \text{and} \qquad (\gratio M)_b = \begin{pmatrix} M_b & 0 \\ 0 & M_a \end{pmatrix}.
\end{equation*}
For the unique arrows $c_1\colon \phi_i \rightarrow j$, $c_2\colon \phi_i \rightarrow \phi_j$ and $c_3\colon i \rightarrow \phi_j$, we have
\begin{equation*}
	(\gratio M)_{c_1} = \begin{pmatrix} M_{c_2} & M_{c_3} \end{pmatrix}, \qquad
	(\gratio M)_{c_2} = \begin{pmatrix} M_{c_2} & \omega M_{c_3} \\ \omega M_{c_1} & 0 \end{pmatrix}, \qquad (\gratio M)_{c_3} = \begin{pmatrix} M_{c_2} \\ M_{c_1}\end{pmatrix},
\end{equation*}
where $\omega \in K \setminus \{0,1,2\inv\}$. The construction is independent (up to isomorphism) of the choice of $\omega$ subject to the condition above. The exclusion of the subset $\{0,1,2\inv\}\subset K$ is to ensure the isomorphism (\ref{eq:ActOnObj}) holds. For any $f \in \Hom_{\calM_F}(M,N)$, we define $\gratio f \in \Hom_{\calM_F}(\gratio M,\gratio N)$ by
\begin{equation*}
	(\gratio f)_i = f_{\phi_i} \qquad \text{and} \qquad (\gratio f)_{\phi_i} = \begin{pmatrix} f_{\phi_i} & 0 \\ 0 & f_{i} \end{pmatrix}.
\end{equation*}

For $\coxI$-type quivers, the principle of the construction is similar to the above. For this case, we need some additional notation. Recall that $Q_0^\gend= \{i : 0 \leq i \leq 2n -1\}$ and $\kappa(i)=U_i(\cos\frac{\pi}{2n+1})$. We will assume that even vertices are sources and odd vertices are sinks --- the construction is dual for the opposite quiver. Let $Q^\gend_1 = \{a_1,\ldots, a_{2n-1}\}$, where each arrow $a_i$ is the unique arrow between the vertices $i-1,i \in Q_0^\gend$. Define a surjection of sets $h\colon Q_0^\gend \rightarrow \{0,\ldots,n-1\}$ by $h(i) = k$ if $U_{i}(\cos \frac{\pi}{2n+1}) = U_{k}(\cos \frac{\pi}{2n+1})$. For each $j \in Q_0^\gend$, define an injection of sets $h_{j}\inv\colon\{0,\ldots,n-1\} \rightarrow Q_0^\gend$ by $h_{j}\inv(k) = i$ if $h(i) = k$ and $F(i)=F(j)$, and note that $h$ is a left inverse to $h_j\inv$ for all $j \in Q_0^\gend$. Next, for each $i \in Q^\gend_0$ and each $0 \leq k \leq n-1$ define an index set
\begin{equation*}
	V(i,k) = \{h_i\inv(l) \in Q_0^\gend : \croot_l \text{ is a summand of } \croot_{h(i)}\croot_k \in \chebsr\ps{2n+1}\}.
\end{equation*}

We may now proceed with the construction, for which we will later provide a specific example. Given an object $M \in \calM_F$, we define the object $\croot_k M \in \calM_F$ to be the $KQ^\gend$-module with the following structure. We have vector subspaces
\begin{equation*}
	(\croot_k M)_i = \bigoplus_{j \in V(i,k)} M_j.
\end{equation*}
We define the linear map $(\croot_k M)_{a_{i}}$ to be the restriction onto $(\croot_k M)_{i'}$ of the linear map
\begin{equation*}
	\begin{pmatrix}
		\lambda M_{a_1}	& M_{a_2}			& 0 					& \cdots	& 0 \\
		0					& \lambda M_{a_3}	& M_{a_4}			& \ddots	& \vdots \\
		0					& 0						& \lambda M_{a_5}	& \ddots	& 0 \\
		\vdots				& \vdots				& \ddots				& \ddots	& M_{a_{2n-1}} \\
		0					& 0						& \cdots		& 0			& \lambda M_{a_{2n-1}}
	\end{pmatrix} \colon
	\bigoplus_{0 \leq 2j \leq 2n-2} M_{2j} \rightarrow \bigoplus_{0 \leq 2j+1 \leq 2n-1} M_{2j+1}
\end{equation*}	
composed with the canonical surjection onto $(\croot_k M)_{i''}$, where $i'$ and $i''$ are the source and target vertices of $a_{i}$ respectively, $\lambda = 1$ if $i$ is odd, and $\lambda \in K \setminus J$ for some finite set $J$ if $i$ is even. The set $J$ is determined such that the isomorphism (\ref{eq:ActOnObj}) holds. For any $f \in \Hom_{\calM_F}(M,N)$, we define $\croot_k f \in \Hom_{\calM_F}(\croot_k M,\croot_k N)$ by
\begin{equation*}
	(\croot_k f)_i = \bigoplus_{j \in V(i,k)} f_j.
\end{equation*}

Finally, for both $\coxH$ and $\coxI$-type quivers, we define 
\begin{equation*}
	(r+s)f\colon (r+s) M \rightarrow (r+s) N = rf \oplus sf\colon rM \oplus sM \rightarrow rN \oplus sN
\end{equation*}
for any $r,s \in \chebsr\ps{2n+1}$ such that $r+s$ is irreducible, and we define $\croot_0 = 1 \in \chebsr\ps{2n+1}$ to act by the identity functor.

\begin{exam}\label{ex:I7Functors}
	Let $F$ be the folding given in Example~\ref{ex:I7}. Then $\chebsr\ps{7}$ acts on an object $M \in \calM_F$ in the following way.
	
	\begin{align*}
		(\croot_1 M)_0 &= M_4				&	(\croot_1 M)_1 &= M_3 \oplus M_5 \\
		(\croot_1 M)_2 &= M_2 \oplus M_4	&	(\croot_1 M)_3 &= M_1 \oplus M_3 \\
		(\croot_1 M)_4 &= M_0 \oplus M_2	&	(\croot_1 M)_5 &= M_1
	\end{align*}
	\begin{align*}
		(\croot_1 M)_{a_1} &=
		\begin{pmatrix}
			M_{a_4} \\ M_{a_5}
		\end{pmatrix}
		& (\croot_1 M)_{a_2} &=
		\begin{pmatrix}
			\lambda M_{a_3}	& M_{a_4} \\
			0					& \lambda M_{a_5}
		\end{pmatrix}
		\\ (\croot_1 M)_{a_3} &=
		\begin{pmatrix}
			M_{a_2}	& 0 \\
			M_{a_3}	& M_{a_4}
		\end{pmatrix}
		& (\croot_1 M)_{a_4} &=
		\begin{pmatrix}
			\lambda M_{a_1}	& M_{a_2} \\
			0					& \lambda M_{a_3}
		\end{pmatrix}
		\\ (\croot_1 M)_{a_5} &=
		\begin{pmatrix}
			M_{a_1} & M_{a_2}
		\end{pmatrix}
	\end{align*}
	\begin{align*}
		(\croot_2 M)_0 &= M_2								&	(\croot_2 M)_1 &= M_1 \oplus M_3 \\
		(\croot_2 M)_2 &= M_0 \oplus M_2 \oplus M_4	&	(\croot_2 M)_3 &= M_1 \oplus M_3 \oplus M_5 \\
		(\croot_2 M)_4 &= M_2 \oplus M_4					&	(\croot_2 M)_5 &= M_3
	\end{align*}
	\begin{align*}
		(\croot_2 M)_{a_1} &=
		\begin{pmatrix}
			M_{a_2} \\ M_{a_3}
		\end{pmatrix}
		& (\croot_2 M)_{a_2} &=
		\begin{pmatrix}
			\lambda M_{a_1}	& M_{a_2}			& 0 \\
			0					& \lambda M_{a_3}	& M_{a_4} 
		\end{pmatrix}
		\\ (\croot_2 M)_{a_3} &=
		\begin{pmatrix}
			M_{a_1}	& M_{a_2}	& 0 \\
			0			& M_{a_3}	& M_{a_4} \\
			0			& 0				& M_{a_5}
		\end{pmatrix}
		& (\croot_2 M)_{a_4} &=
		\begin{pmatrix}
			M_{a_2}			& 0 \\
			\lambda M_{a_3}	& M_{a_4} \\
			0					& \lambda M_{a_5}
		\end{pmatrix}
		\\ (\croot_2 M)_{a_5} &=
		\begin{pmatrix}
			M_{a_3} & M_{a_4}
		\end{pmatrix}
	\end{align*}
	where $\lambda \in K \setminus J$. The set $J$ can be computed to be
	\begin{equation*}
		J= \{0,1, 2\inv, 3 \cdot 2\inv, 3\inv, \zeta_1,\zeta_2, 2\inv \cdot \zeta_1,2\inv \cdot\zeta_2\},
	\end{equation*}
	where $\zeta_1, \zeta_2$ are roots of $x^2 -3x +1$. Figure~\ref{fig:I7CoeffAR} shows the Aulsander-Reiten quiver of $\calM_F$, with indecomposable objects written as images under the actions of $\croot_1,\croot_2 \in \chebsr\ps{7}$.
\end{exam}

\begin{proof}[Proof of Theorem~\ref{thm:Action}]
	First we will prove the theorem for the module category $\calM_F$. (M1)-(M4) follows by definition --- and in particular, the isomorphism (\ref{eq:ActOnObj}). (M5)-(M6) follows from the fact that each $r \in \chebsr\ps{2n+1}$ acts diagonally on morphisms.
	
	To see that the bounded derived category $\der_F$ is also a $\chebsr^{2n+1}$-coefficient category, one defines $r\sus^k X = \sus^k rX$ for any $k \in \integer$ and any object $X \in \der_F$ in degree 0. Since $KQ^\gend$ is hereditary and that $R_+$ acts linearly and faithfully on $\Ext^1_{\calM_F}$-spaces, the result follows.
\end{proof}

\begin{figure}[b]
	\centering
	\begin{tikzpicture}[scale=0.9]
\draw[blue] (1,-6) node {\footnotesize$I(1)$};
\draw[blue] (-1.8,-6) node {\footnotesize$\tau I(1)$};
\draw[blue] (-4.6,-6) node {\footnotesize$\tau^2 I(1)$};
\draw[blue] (-7.4,-6) node {\footnotesize$\tau\inv I(1)$};
\draw[blue] (-10.2,-6) node {\footnotesize$P(1)$};

\draw[violet] (-0.4,-4.5) node {\footnotesize$I(2)$};
\draw[violet] (-3.2,-4.5) node {\footnotesize$\tau I(2)$};
\draw[violet] (-6,-4.5) node {\footnotesize$\tau^2 I(2)$};
\draw[violet] (-8.8,-4.5) node {\footnotesize$\tau\inv P(2)$};
\draw[violet] (-11.6,-4.5) node {\footnotesize$P(2)$};

\draw[red] (-1.8,-3) node {\footnotesize$\gratio I(3)$};
\draw[red] (-4.6,-3) node {\footnotesize$\gratio \tau I(3)$};
\draw[red] (-7.4,-3) node {\footnotesize$\gratio \tau^2 I(3)$};
\draw[red] (-10.2,-3) node {\footnotesize$\gratio \tau\inv P(3)$};
\draw[red] (-13,-3) node {\footnotesize$\gratio P(3)$};

\draw[violet] (-0.4,-1.5) node {\footnotesize$\gratio I(2)$};
\draw[violet] (-3.2,-1.5) node {\footnotesize$\gratio \tau I(2)$};
\draw[violet] (-6,-1.5) node {\footnotesize$\gratio \tau^2 I(2)$};
\draw[violet] (-8.8,-1.5) node {\footnotesize$\gratio \tau\inv P(2)$};
\draw[violet] (-11.6,-1.5) node {\footnotesize$\gratio P(2)$};

\draw[blue] (1,0) node {\footnotesize$\gratio I(1)$};
\draw[blue] (-1.8,0) node {\footnotesize$\gratio \tau I(1)$};
\draw[blue] (-4.6,0) node {\footnotesize$\gratio \tau^2 I(1)$};
\draw[blue] (-7.4,0) node {\footnotesize$\gratio \tau\inv P(1)$};
\draw[blue] (-10.2,0) node {\footnotesize$\gratio P(1)$};

\draw[red] (-1.8,1.5) node {\footnotesize$I(3)$};
\draw[red] (-4.6,1.5) node {\footnotesize$\tau I(3)$};
\draw[red] (-7.4,1.5) node {\footnotesize$\tau^2 I(3)$};
\draw[red] (-10.2,1.5) node {\footnotesize$\tau\inv P(3)$};
\draw[red] (-13,1.5) node {\footnotesize$P(3)$};

\draw [->](-0.15,-5) -- (0.75,-5.5);
\draw [->](-1.55,-5.5) -- (-0.65,-5);
\draw [->](-2.95,-5) -- (-2.05,-5.5);
\draw [->](-4.35,-5.5) -- (-3.45,-5);
\draw [->](-5.75,-5) -- (-4.85,-5.5);
\draw [->](-7.15,-5.5) -- (-6.25,-5);
\draw [->](-8.55,-5) -- (-7.65,-5.5);
\draw [->](-9.95,-5.5) -- (-9.05,-5);
\draw [->](-11.35,-5) -- (-10.45,-5.5);

\draw [->](-1.55,-3.5) -- (-0.65,-4);
\draw [->](-2.95,-4) -- (-2.05,-3.5);
\draw [->](-4.35,-3.5) -- (-3.45,-4);
\draw [->](-5.75,-4) -- (-4.85,-3.5);
\draw [->](-7.15,-3.5) -- (-6.25,-4);
\draw [->](-8.55,-4) -- (-7.65,-3.5);
\draw [->](-9.95,-3.5) -- (-9.05,-4);
\draw [->](-11.35,-4) -- (-10.45,-3.5);
\draw [->](-12.75,-3.5) -- (-11.85,-4);

\draw [->](-1.55,-2.5) -- (-0.65,-2);
\draw [->](-2.95,-2) -- (-2.05,-2.5);
\draw [->](-4.35,-2.5) -- (-3.45,-2);
\draw [->](-5.75,-2) -- (-4.85,-2.5);
\draw [->](-7.15,-2.5) -- (-6.25,-2);
\draw [->](-8.55,-2) -- (-7.65,-2.5);
\draw [->](-9.95,-2.5) -- (-9.05,-2);
\draw [->](-11.35,-2) -- (-10.45,-2.5);
\draw [->](-12.75,-2.5) -- (-11.85,-2);

\draw [->](-0.15,-1) -- (0.75,-0.5);
\draw [->](-1.55,-0.5) -- (-0.65,-1);
\draw [->](-2.95,-1) -- (-2.05,-0.5);
\draw [->](-4.35,-0.5) -- (-3.45,-1);
\draw [->](-5.75,-1) -- (-4.85,-0.5);
\draw [->](-7.15,-0.5) -- (-6.25,-1);
\draw [->](-8.55,-1) -- (-7.65,-0.5);
\draw [->](-9.95,-0.5) -- (-9.05,-1);
\draw [->](-11.35,-1) -- (-10.45,-0.5);

\draw [->](-1.4,1) -- (-0.5,-1);
\draw [->](-3.1,-1) -- (-2.2,1);
\draw [->](-4.2,1) -- (-3.3,-1);
\draw [->](-5.9,-1) -- (-5,1);
\draw [->](-7,1) -- (-6.1,-1);
\draw [->](-8.7,-1) -- (-7.8,1);
\draw [->](-9.8,1) -- (-8.9,-1);
\draw [->](-11.5,-1) -- (-10.6,1);
\draw [->](-12.7,1) -- (-11.7,-1);

\end{tikzpicture}
	\caption{The Auslander-Reiten quiver of $\mod*KQ^{\coxD_6}$ viewed as $\goldsr$-coefficient category under the folding $F \colon Q^{\coxD_6} \rightarrow Q^{\coxH_3}$.} \label{fig:H3CoeffAR}
\end{figure}

\begin{figure}
	\centering
	\def\y{0.35cm}
\def\z{1}
\begin{tikzpicture}[scale=2]
	\foreach \x in {0,...,3} {
		\foreach \z in {0,2,4} {
			\coordinate (t\x_I\z) at ($(1.5*1-\x*1,{\z/2*0.8 + 3*0.8})$);
		}
	}
	\foreach \x in {0,...,2} {
		\foreach \z in {1,3,5} {
			\coordinate (t\x_I\z) at ($(1*1-\x*1,{(0.8*(\z - 1)/2 + 0.8*3.5)})$);
		}
	}
	
	\draw[<-, shorten <= \y*\z,shorten >= \y*\z] (t0_I0) -- (t0_I1);
	\draw[<-, shorten <= \y*\z,shorten >= \y*\z] (t0_I2) -- (t0_I1);
	\draw[<-, shorten <= \y*\z,shorten >= \y*\z] (t0_I2) -- (t0_I3);
	\draw[<-, shorten <= \y*\z,shorten >= \y*\z] (t0_I4) -- (t0_I3);
	\draw[<-, shorten <= \y*\z,shorten >= \y*\z] (t0_I4) -- (t0_I5);
	
	\draw[->, shorten <= \y*\z,shorten >= \y*\z] (t1_I0) -- (t0_I1);
	\draw[->, shorten <= \y*\z,shorten >= \y*\z] (t1_I2) -- (t0_I1);
	\draw[->, shorten <= \y*\z,shorten >= \y*\z] (t1_I2) -- (t0_I3);
	\draw[->, shorten <= \y*\z,shorten >= \y*\z] (t1_I4) -- (t0_I3);
	\draw[->, shorten <= \y*\z,shorten >= \y*\z] (t1_I4) -- (t0_I5);
	
	\draw[<-, shorten <= \y*\z,shorten >= \y*\z] (t1_I0) -- (t1_I1);
	\draw[<-, shorten <= \y*\z,shorten >= \y*\z] (t1_I2) -- (t1_I1);
	\draw[<-, shorten <= \y*\z,shorten >= \y*\z] (t1_I2) -- (t1_I3);
	\draw[<-, shorten <= \y*\z,shorten >= \y*\z] (t1_I4) -- (t1_I3);
	\draw[<-, shorten <= \y*\z,shorten >= \y*\z] (t1_I4) -- (t1_I5);
	
	\draw[->, shorten <= \y*\z,shorten >= \y*\z] (t2_I0) -- (t1_I1);
	\draw[->, shorten <= \y*\z,shorten >= \y*\z] (t2_I2) -- (t1_I1);
	\draw[->, shorten <= \y*\z,shorten >= \y*\z] (t2_I2) -- (t1_I3);
	\draw[->, shorten <= \y*\z,shorten >= \y*\z] (t2_I4) -- (t1_I3);
	\draw[->, shorten <= \y*\z,shorten >= \y*\z] (t2_I4) -- (t1_I5);
	
	\draw[<-, shorten <= \y*\z,shorten >= \y*\z] (t2_I0) -- (t2_I1);
	\draw[<-, shorten <= \y*\z,shorten >= \y*\z] (t2_I2) -- (t2_I1);
	\draw[<-, shorten <= \y*\z,shorten >= \y*\z] (t2_I2) -- (t2_I3);
	\draw[<-, shorten <= \y*\z,shorten >= \y*\z] (t2_I4) -- (t2_I3);
	\draw[<-, shorten <= \y*\z,shorten >= \y*\z] (t2_I4) -- (t2_I5);
	
	\draw[->, shorten <= \y*\z,shorten >= \y*\z] (t3_I0) -- (t2_I1);
	\draw[->, shorten <= \y*\z,shorten >= \y*\z] (t3_I2) -- (t2_I1);
	\draw[->, shorten <= \y*\z,shorten >= \y*\z] (t3_I2) -- (t2_I3);
	\draw[->, shorten <= \y*\z,shorten >= \y*\z] (t3_I4) -- (t2_I3);
	\draw[->, shorten <= \y*\z,shorten >= \y*\z] (t3_I4) -- (t2_I5);
	
	\draw[blue] (t0_I0) node {\footnotesize$M_{0}$};
	\draw[blue] (t1_I0) node {\footnotesize$M_{2}$};
	\draw[blue] (t2_I0) node {\footnotesize$M_{4}$};
	\draw[blue] (t3_I0) node {\footnotesize$M_{6}$};
	
	\draw[red] (t0_I1) node {\footnotesize$\croot_1 M_{1}$};
	\draw[red] (t1_I1) node {\footnotesize$\croot_1 M_{3}$};
	\draw[red] (t2_I1) node {\footnotesize$\croot_1 M_{5}$};
	
	\draw[blue] (t0_I2) node {\footnotesize$\croot_2 M_{0}$};
	\draw[blue] (t1_I2) node {\footnotesize$\croot_2 M_{2}$};
	\draw[blue] (t2_I2) node {\footnotesize$\croot_2 M_{4}$};
	\draw[blue] (t3_I2) node {\footnotesize$\croot_2 M_{6}$};
	
	\draw[red] (t0_I3) node {\footnotesize$\croot_2 M_{1}$};
	\draw[red] (t1_I3) node {\footnotesize$\croot_2 M_{3}$};
	\draw[red] (t2_I3) node {\footnotesize$\croot_2 M_{5}$};
	
	\draw[blue] (t0_I4) node {\footnotesize$\croot_1 M_{0}$};
	\draw[blue] (t1_I4) node {\footnotesize$\croot_1 M_{2}$};
	\draw[blue] (t2_I4) node {\footnotesize$\croot_1 M_{4}$};
	\draw[blue] (t3_I4) node {\footnotesize$\croot_1 M_{6}$};
	
	\draw[red] (t0_I5) node {\footnotesize$M_{1}$};
	\draw[red] (t1_I5) node {\footnotesize$M_{3}$};
	\draw[red] (t2_I5) node {\footnotesize$M_{5}$};
	
%
%
%
%
\end{tikzpicture}
	\caption{The Auslander-Reiten quiver of $\mod*KQ^{\coxA_6}$ viewed as $\chebsr\ps{7}$-coefficient category under the folding $F \colon Q^{\coxA_6} \rightarrow Q^{\coxI_2(7)}$. Modules are labelled such that $\dimproj_F(M_k)$ corresponds to the point $\e^{\frac{k \pi \imunit}{7}}$.} \label{fig:I7CoeffAR}
\end{figure}
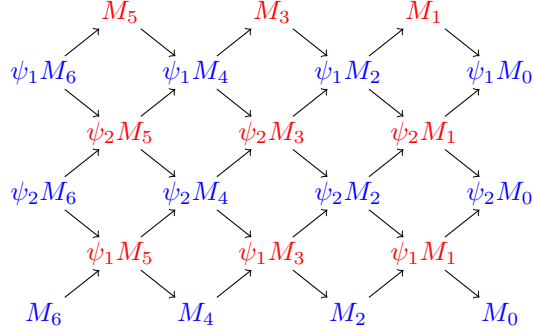

For each of the weighted foldings $F\colon Q^\gend \rightarrow Q^{\gend'}$, we respectively call the categories $\calM_F$ and $\der_F$ equipped with the appropriate semiring actions given in Theorem~\ref{thm:Action} the abelian and triangulated $R_+$-coefficient categories corresponding to $F$, where $R_+ = \chebsr\ps{2n+1}$.

\begin{rem} \label{rem:FoldingMultiplication}
	For each of the foldings $F\colon Q^\gend \rightarrow Q^{\gend'}$ with $\gend' \in \{\coxH_3,\coxH_4,\coxI_2(2n+1)\}$, the action of $R_+ = \chebsr\ps{2n+1}$ on $\calM_F$ and $\der_F$ is defined such that the following hold for any $M \in \calM_F$, any $X \in \der_F$ and any $r \in R_+$:
	\begin{enumerate}[label=(\alph*)]
		\item If $X$ is concentrated in degree $k$, then $rX$ is concentrated in degree $k$.
		\item $\dimproj_F(rM) = \srhom\ps{2n+1}(r) \dimproj_F(M)$ and $\derdim_F(rX) = \srhom\ps{2n+1}(r) \derdim_F(X)$, where $\srhom\ps{2n+1}\colon \chebsr\ps{2n+1} \rightarrow \wh\chebr\ps{2n+1}$ is the homomorphism described in Remark~\ref{rem:ChebsrOrder}.
		\item $rM$ (resp. $rX$) is a direct sum of objects that belong to the same column of the Auslander-Reiten quiver of $\calM_F$ (resp. $\der_F$). In particular, $\Hom_{\calM_F}(N,N')=0$ (resp. $\Hom_{\der_F}(Y,Y')=0$) for any non-isomorphic indecomposable direct summands $N,N' \subseteq rM$ (resp. $Y,Y' \subseteq rX$). Moreover, $\tau_{\calM_F} rM \cong r \tau_{\calM_F} M$ and  $\tau_{\der_F} rX \cong r \tau_{\der_F} X$.
		\item Suppose $rM \cong M_1 \oplus \ldots \oplus M_m$. Then the action of $r$ maps an Aulsander-Reiten sequence starting (or ending) in $M$ to a direct sum of Auslander-Reiten sequences starting (resp. ending) in each $M_i$. A similar statement holds for $r X$ and Auslander-Reiten triangles starting/ending in $X$.
	\end{enumerate}
\end{rem}

\begin{cor}
	Let $F\colon Q^\gend \rightarrow Q^{\gend'}$ be a folding of quivers with $\gend' \in \{\coxH_3,\coxH_4, \coxI_2(2n+1)\}$, and let $\calM_F$ be the corresponding abelian $\chebsr\ps{2n+1}$-coefficient category. Then the Grothendieck group $G_0(\calM_F)$ has the structure of a free $\chebr\ps{2n+1}$-module of rank $|Q_0^{\gend'}|$.
\end{cor}
\begin{proof}
	The $\chebr\ps{2n+1}$-module structure of $G_0(\calM_F)$ follows from Remark~\ref{rem:RCoeffConseq}(a). That $G_0(\calM_F)$ is free of rank $|Q_0^{\gend'}|$ follows from the fact that a basis for $G_0(\calM_F)$ is given by the iso-classes of simple modules, and that any simple module is obtained as an $R_+$-multiple of a simple module that projects onto a simple root of $\gend'$.
\end{proof}

\subsection{The reduced Auslander-Reiten quiver of an $R_+$-coefficient category}
It is possible to construct a simplified version of the Auslander-Reiten quiver for an $R_+$-coefficient category, which we call the \emph{reduced Auslander-Reiten quiver}. This takes its inspiration from the valuations on the Auslander-Reiten quiver of the module category of Artin algebras over non-algebraically-closed fields (cf. \cite[VII]{AuslanderReitenSmalo}).

\begin{defn} \label{def:Generators}
	Let $R_+$ be a partially ordered semiring and $R$ be the associated ring (Remark~\ref{rem:RCoeffConseq}) and suppose $\mathcal{M}$ is an abelian $R_+$-coefficient category. We say an object $N \in \calM$ is \emph{$R_+$-generated by $M \in \calM$} if there exists a split exact sequence
	\begin{equation*}
		0 \rightarrow sM \rightarrow rM \rightarrow L \rightarrow 0
	\end{equation*}
	such that $N \cong L$. Similarly, for a triangulated $R_+$-coefficient category $\der$, we say an object $N \in \der$ is $R_+$-generated by $M \in \der$ if there exists a split triangle
	\begin{equation*}
		sM \rightarrow rM \rightarrow L \rightarrow \sus sM
	\end{equation*}
	with monic first morphism such that $N \cong L$. In both the abelian and triangulated cases, we call the pair $(r,s)$ the \emph{$R_+$-generating pair} of $N$ by $M$, and call the element $r-s \in R$ the \emph{$R$-index} of $N$ with respect to $M$. We denote the class of all objects $R_+$-generated by $M$ by $\gen_M$.
	
	Let $\Gamma$ be a set of objects of $\calM$ and define $\calM(\Gamma)$ to be the full subcategory of $\calM$ whose objects are isomorphic to finite direct sums of the objects in the set
	\begin{equation*}
		\gen_\Gamma=\{L \in \gen_M :  M \in \Gamma \}.
	\end{equation*}
	We say $\Gamma$ is a set of \emph{$R_+$-generators} of $\calM$ if $\calM(\Gamma) \simeq \calM$.
\end{defn}

\begin{rem} \label{rem:SplitGen}
	By the properties of triangulated categories, the condition that $f$ is a monomorphism in a triangle
	\begin{equation*}
		\xymatrix@C=2ex{
		sM \ar[r]^-f& rM \ar[r]^-g& L \ar[r]^-h& \sus sM
		}
	\end{equation*}
	is strong enough to ensure that $rM \cong sM \oplus L$, that $g$ is an epimorphism and that $h$ is the zero morphism. Thus, the triangle is split. See \cite[IV.1]{GelfandManin} for an account.
\end{rem}

\begin{defn} \label{defn:RPPO}
	Given a set $\Gamma$ of objects of $\calM$, we can endow the category $\calM(\Gamma)$ with a partial ordering given by $L_1 \leq L_2$ if and only if $L_1 \cong L_2$ or $L_1$ and $L_2$ are $R_+$-generated by a common $M \in \Gamma$ with respective $R$-indices $r_1$ and $r_2$  such that $r_1 < r_2$.
\end{defn}

\begin{defn} \label{def:MinGenerators}
	We say a set $\Gamma$ of $R_+$-generators of $\calM$ is \emph{basic} if the elements of $\Gamma$ are pairwise non-isomorphic indecomposable objects of $\calM$. We say a set $\Gamma$ of basic $R_+$-generators is \emph{minimal} if for any other set of basic $R_+$-generators $\Gamma'$, we have an injective map of sets $\theta\colon\Gamma \rightarrow \Gamma'$ and inclusions $\iota\colon \Gamma \rightarrow \Gamma \cup \Gamma'$ and $\iota'\colon \Gamma' \rightarrow \Gamma \cup \Gamma'$ such that for any $M \in \Gamma$, we have $\iota(M) \leq \iota'\theta(M)$ in the category $\calM(\Gamma \cup \Gamma')\simeq\calM$. 
\end{defn}

For the $R_+$-coefficient categories $\calM_F$ and $\der_F$ that we have constructed from a weighted folding of quivers $F$, the above notions of $R_+$-generators and minimality can be deduced from the geometry given by the projection maps $\dimproj_F$ and $\derdim_F$ respectively.

\begin{lem} \label{lem:FoldingGenerators}
	Let $F\colon Q^\gend \rightarrow Q^{\gend'}$ be a folding of quivers with $\gend'\in\{\coxH_3,\coxH_4,\coxI_2(2n+1)\}$ and let $\calM_F$ and $\der_F$ be the corresponding abelian and triangulated $R_+$-coefficient categories, respectively.
	\begin{enumerate}[label=(\alph*)]
		\item  Let $M_\alpha \in \calM_F$ be an indecomposable object such that $\dimproj_F(M_\alpha)=\alpha$ is a positive root of $\gend'$. Then an indecomposable $L \in \calM_F$ is $R_+$-generated by $M_\alpha$ if and only if $\dimproj_F(L)$ is collinear to $\alpha$.
		\item Let $X_\alpha \in \der_F$ be an indecomposable object such that $\derdim_F(X_\alpha)=\alpha$ is a root of $\gend'$. Then an indecomposable $Y \in \der_F$ is $R_+$-generated by $X_\alpha$ if and only if $Y$ is concentrated in the same degree as $X_\alpha$ and $\derdim_F(Y)$ is collinear to $\alpha$.
	\end{enumerate}
\end{lem}
\begin{proof}
	(a) For clarity, we have $R_+ = \chebsr\ps{2n+1}$ in all $\coxI$-type cases and $R_+ = \chebsr\ps{5} \cong \goldsr$ in the $\coxH$-type cases. Recall that by Theorem~\ref{thm:Folding}, the construction in Theorem~\ref{thm:Action}, and by Remark~\ref{rem:FoldingMultiplication}(b), for each positive root $\alpha$ of $\gend'$, there exists a set
	\begin{equation*}
		\mathbf{M}_\alpha = \left\{M_{\alpha},\croot_1 M_{\alpha},\ldots, \croot_{n-1} M_{\alpha}\right\}
	\end{equation*}
	of pairwise non-isomorphic indecomposable modules such that
	\begin{equation*}
		\dimproj_F(\croot_i M_{\alpha})=\srhom\ps{2n+1}(\croot_i)\dimproj_F(M_{\alpha}) = U_i\left(\cos\frac{\pi}{2n+1}\right) \alpha.
	\end{equation*}
	
	It follows that $L \in \calM_F$ is $\chebsr\ps{2n+1}$-generated by $M_\alpha$ if and only if $L$ is isomorphic to a direct sum of objects in $\mathbf{M}_\alpha$. The sufficiency follows from the fact that any element of $\chebsr\ps{2n+1}$ is of the form $r=a\ps{0}\croot_0+\ldots+a\ps{n-1}\croot_{n-1}$ with each $a\ps{i} \in \nnint$ and that $\bigcup_{\alpha \in \posroots(\gend')} \mathbf{M}_\alpha$ is a complete set of non-isomorphic indecomposables of $\calM_F$, where $\posroots(\gend')$ is the set of positive roots of $\gend'$. The necessity follows from the split exact sequences
	\begin{equation*}
		0 \rightarrow s M_\alpha \rightarrow (s + r)M_\alpha \rightarrow r M_\alpha \rightarrow 0
	\end{equation*}
	for any $s \in \chebsr\ps{2n+1}$. On the other hand, it is also clear that for any $L_1,L_2 \in \mathbf{M}_\alpha$, we have that $\dimproj_F(L_1)$ is collinear to $\dimproj_F(L_2)$. Conversely, if an indecomposable $L \in \calM_F$ is such that $\dimproj_F(L)$ is collinear to $M_\alpha$ for some positive root $\alpha$ of $\gend'$, then Theorem~\ref{thm:Folding} tells us that $L \in \mathbf{M}_\alpha$, as required.
	
	(b) This follows from the proof of (a) along with Remark~\ref{rem:FoldingMultiplication}(a) and that $R_+$ acts by a triangulated functor.
\end{proof}

\begin{prop} \label{prop:GenGeometry}
	Let $\calM_F$ and $\der_F$ be as previously stated and $\leq$ be the partial ordering in Definition~\ref{defn:RPPO}.
	\begin{enumerate}[label=(\alph*)]
		\item For any indecomposable $L_1,L_2 \in \calM_F$, we have the following:
		\begin{enumerate}[label=(\roman*)]
			\item $L_1$ and $L_2$ have a common $R_+$-generator if and only if $\dimproj_F(L_1)$ is collinear to $\dimproj_F(L_2)$;
		\item $L_1 < L_2$ if and only if $\dimproj_F(L_1)$ is collinear to $\dimproj_F(L_2)$ and $||\dimproj_F(L_1)||<||\dimproj_F(L_2)||$.
		\end{enumerate}
		\item For any indecomposable $X_1,X_2 \in \der_F$, we have the following:
		\begin{enumerate}[label=(\roman*)]
			\item $X_1$ and $X_2$ have a common $R_+$-generator if and only if $X_1$ and $X_2$ are concentrated in the same degree and $\dimproj_F(X_1)$ is collinear to $\dimproj_F(X_2)$;
		\item $X_1 < X_2$ if and only if $X_1$ and $X_2$ are concentrated in the same degree, $\dimproj_F(X_1)$ is collinear to $\dimproj_F(X_2)$, and $||\dimproj_F(X_1)||<||\dimproj_F(X_2)||$.
		\end{enumerate}
	\end{enumerate} 
\end{prop}
\begin{proof}
	We will prove (b), as the proof for (a) is identical to (b) concentrated in degree 0.
	
	(b)(i) By Remark~\ref{rem:FoldingMultiplication}(a), if $X_1$ and $X_2$ are both $R_+$-generated by $Y \in \der_F$, then $X_1$ and $X_2$ are concentrated in the same degree. It further follows from Proposition~\ref{prop:DerivedProj} and Remark~\ref{rem:FoldingMultiplication}(b) and (c) that $\derdim(X_1)$ and $\derdim(X_2)$ are some $\wh\chebr{2n+1}$-multiples of a common root $\alpha$ of $\gend'$. Thus, $X_1$ and $X_2$ have a common $R_+$ generator only if $X_1$ and $X_2$ are concentrated in the same degree and $\derdim(X_1)$ and $\derdim(X_2)$ are collinear. Conversely, if $X_1$ and $X_2$ are concentrated in the same degree and $\derdim(X_1)$ and $\derdim(X_2)$ are collinear, then $X_1$ and $X_2$ have a common $R_+$-generator by Lemma~\ref{lem:FoldingGenerators}.
	
	(ii) $X_1 < X_2$ if and only if $X_1 \not\cong X_2$ and both have a common $R_+$-generator with respective (positive) $\chebr\ps{2n+1}$-indices $r_1 < r_2$. By Lemma~\ref{lem:FoldingGenerators} and Remark~\ref{rem:FoldingMultiplication}(b), this is true if and only if $X_1$ and $X_2$ are concentrated in the same degree, $\derdim(X_1)$ and $\derdim(X_2)$ are collinear, and $\wh\sigma\ps{2n+1}\sigma\ps{2n+1}(r_1) < \wh\sigma\ps{2n+1}\sigma\ps{2n+1}(r_2)$. The result then follows from the fact that $||\dimproj_F(X_1)||<||\dimproj_F(X_2)||$ if and only if $\wh\sigma\ps{2n+1}\sigma\ps{2n+1}(r_1) < \wh\sigma\ps{2n+1}\sigma\ps{2n+1}(r_2)$.
\end{proof}

\begin{thm} \label{thm:Generators}
	Let $F\colon Q^\gend \rightarrow Q^{\gend'}$ be a folding of quivers with $\gend=\{\coxA_{2n},\coxD_6,\coxE_8\}$ and $\gend'=\{\coxH_4,\coxH_3,\coxI_2(2n+1)\}$, and let $\calM_F$ and $\der_F$ be the corresponding abelian and triangulated $R_+$-coefficient categories respectively. Define sets
	\begin{align*}
		\Gamma_{\calM_F} &=\{M \in \calM_F : M \text{ indecomposable and }\dimproj_F(M) \text{ is a positive root of } \gend'\}, \\
		\Gamma_{\der_F} &=\{X \in \der_F : X \text{ indecomposable and }\derdim_F(X) \text{ is a root of } \gend'\}.
	\end{align*}
	Then $\Gamma_{\calM_F}$ and $\Gamma_{\der_F}$ are minimal sets of $R_+$-generators for $\calM_F$ and $\der_F$ respectively that are both unique up to isomorphic elements.
\end{thm}
\begin{proof}
	Let $\mathcal{A}$ denote the category $\calM_F$ or $\der_F$. That $\Gamma_{\mathcal{A}}$ is a basic set of $R_+$-generators for $\mathcal{A}$ follows from Theorems~\ref{thm:Folding}, Proposition~\ref{prop:DerivedProj}, \ref{thm:Action}, and Lemma~\ref{lem:FoldingGenerators}. That $\Gamma_{\mathcal{A}}$ is minimal follows from Proposition~\ref{prop:GenGeometry} and Remark~\ref{rem:FoldingMultiplication}. That $\Gamma_{\calM_F}$ is unique up to isomorphic elements follows from the fact that there exists a unique object $M_\alpha \in \calM_F$ up to isomorphism for each positive root of $\gend'$. That $\Gamma_{\der_F}$ is unique up to isomorphic elements follows from the fact, in each even degree, there exists a unique object $X_\alpha \in \calM_F$ up to isomorphism for each positive root of $\gend'$, and in each odd degree, there exists a unique object $X_\alpha \in \calM_F$ up to isomorphism for each negative root of $\gend'$.
\end{proof}

\begin{cor}
	For each non-projective object $M \in \Gamma_{\calM_F}$, we have $\tau_{\calM_F} M \in \Gamma_{\calM_F}$. For each object $X \in \Gamma_{\der_F}$, we have $\tau_{\der_F} X \in \Gamma_{\der_F}$.
\end{cor}
\begin{proof}
	This follows from Theorem~\ref{thm:Folding} and Proposition~\ref{prop:DerivedProj} and the fact that the objects in $\Gamma_{\calM_F}$ and $\Gamma_{\der_F}$ map to roots of $\gend'$ under $\dimproj_F$ and $\derdim_F$ respectively.
\end{proof}

\begin{rem} \label{rem:SimplerGen}
	Let $F\colon Q^\gend \rightarrow Q^{\gend'}$ be a weighted folding with $\gend' = \{\coxH_4, \coxH_3, \coxI_2(2n+1) \}$. Given the minimal set of $R_+$-generators $\Gamma$ for $\calM_F$, we can actually use a simpler notion of $R_+$-generation. In particular, we have
	\begin{equation*}
		\gen_M = \{L \in \calM_F: L \cong rM, r \in R_+\}
	\end{equation*}
	for any $M \in \Gamma$. Note however that this statement is not necessarily true for $N \not\in \Gamma$.
\end{rem}

\begin{defn} \label{def:ReducedAR}
	Let $\calM$ be an $R_+$-coefficient category with an Auslander-Reiten quiver and suppose that $\calM$ has a minimal set of $R_+$-generators $\Gamma$ that is unique up to isomorphic elements and closed under $\tau_\calM$. The \emph{reduced Auslander-Reiten quiver} $\mathcal{A}(\calM)$ of $\calM$ is an $R$-valued translation quiver such that the following hold:
	\begin{enumerate}[label=(A\arabic*)]
		\item The vertices of $\mathcal{A}(\calM)$ are $R_+$-generators $\Gamma \subset \calM$.
		\item The translation of $\mathcal{A}(\calM)$ is the Auslander-Reiten translation $\tau_\calM$.
		\item There exists a valued morphism $\xymatrix{M \ar[r]^-{(r_1,r_2)} & N}$ in $\mathcal{A}(\calM)$ if and only if (i)-(iii) hold.
		\begin{enumerate}[label=(\roman*)]
			\item There exist irreducible morphisms $X_1 \rightarrow N$ and $M \rightarrow X_2$ such that $X_1$ is $R_+$-generated by $M$ with $R$-index $r_1$ and $X_2$ is $R_+$-generated by $N$ with $R$-index $r_2$.
			\item If there exists an Auslander-Reiten sequence/triangle
			\begin{equation*}
				\tau_\calM N \rightarrow X' \rightarrow N \rightarrow
			\end{equation*}
			then $X' \cong X'_0 \oplus X_1$ for some (possibly zero) object $X'_0$ such that no direct summand of $X'_0$ is in $\gen_M$.
			\item If there exists an Auslander-Reiten sequence/triangle
			\begin{equation*}
				M \rightarrow X'' \rightarrow \tau_\calM\inv M \rightarrow
			\end{equation*}
			then $X'' \cong X''_0 \oplus X_2$ for some (possibly zero) object $X''_0$ such that no direct summand of $X''_0$ is in $\gen_N$.
		\end{enumerate}
	\end{enumerate}
\end{defn}

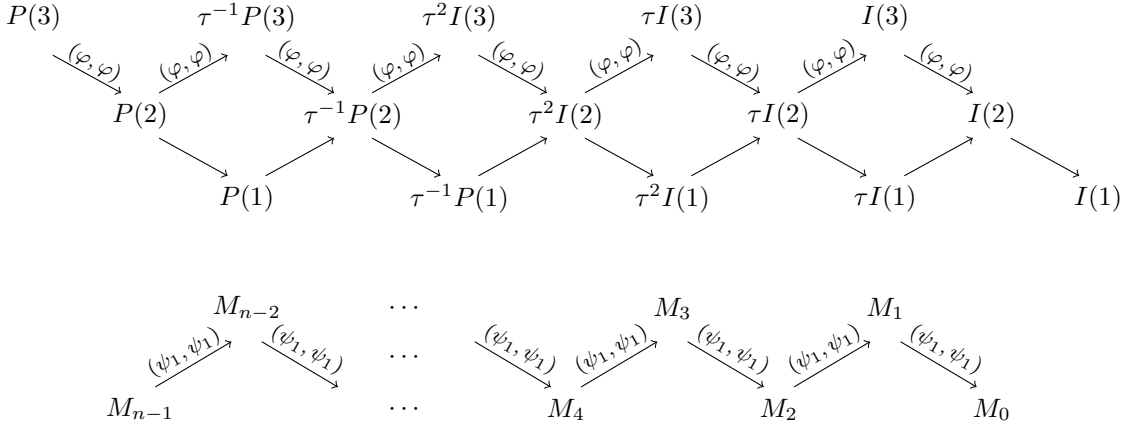
\begin{figure}
	\centering
	\begin{tikzpicture}
\draw[black] (1,-5.8) node {\footnotesize$I(1)$};
\draw[black] (-1.8,-5.8) node {\footnotesize$\tau I(1)$};
\draw[black] (-4.6,-5.8) node {\footnotesize$\tau^2 I(1)$};
\draw[black] (-7.4,-5.8) node {\footnotesize$\tau\inv P(1)$};
\draw[black] (-10.2,-5.8) node {\footnotesize$P(1)$};

\draw[black] (-0.4,-4.7) node {\footnotesize$I(2)$};
\draw[black] (-3.2,-4.7) node {\footnotesize$\tau I(2)$};
\draw[black] (-6,-4.7) node {\footnotesize$\tau^2 I(2)$};
\draw[black] (-8.8,-4.7) node {\footnotesize$\tau\inv P(2)$};
\draw[black] (-11.6,-4.7) node {\footnotesize$P(2)$};

\draw[black] (-1.8,-3.4) node {\footnotesize$I(3)$};
\draw[black] (-4.6,-3.4) node {\footnotesize$\tau I(3)$};
\draw[black] (-7.4,-3.4) node {\footnotesize$\tau^2 I(3)$};
\draw[black] (-10.2,-3.4) node {\footnotesize$\tau\inv P(3)$};
\draw[black] (-13,-3.4) node {\footnotesize$P(3)$};

\draw [->](-0.15,-5) -- (0.75,-5.5);
\draw [->](-1.55,-5.5) -- (-0.65,-5);
\draw [->](-2.95,-5) -- (-2.05,-5.5);
\draw [->](-4.35,-5.5) -- (-3.45,-5);
\draw [->](-5.75,-5) -- (-4.85,-5.5);
\draw [->](-7.15,-5.5) -- (-6.25,-5);
\draw [->](-8.55,-5) -- (-7.65,-5.5);
\draw [->](-9.95,-5.5) -- (-9.05,-5);
\draw [->](-11.35,-5) -- (-10.45,-5.5);

\draw [->](-1.55,-3.9) -- (-0.65,-4.4);
\draw [->](-2.95,-4.4) -- (-2.05,-3.9);
\draw [->](-4.35,-3.9) -- (-3.45,-4.4);
\draw [->](-5.75,-4.4) -- (-4.85,-3.9);
\draw [->](-7.15,-3.9) -- (-6.25,-4.4);
\draw [->](-8.55,-4.4) -- (-7.65,-3.9);
\draw [->](-9.95,-3.9) -- (-9.05,-4.4);
\draw [->](-11.35,-4.4) -- (-10.45,-3.9);
\draw [->](-12.75,-3.9) -- (-11.85,-4.4);

\node[rotate=-29] at (-1,-4) {\tiny$(\gratio,\gratio)$};
\node[rotate=-29] at (-3.8,-4) {\tiny$(\gratio,\gratio)$};
\node[rotate=-29] at (-6.63,-3.98) {\tiny$(\gratio,\gratio)$};
\node[rotate=-29] at (-9.44,-3.97) {\tiny$(\gratio,\gratio)$};
\node[rotate=-29] at (-12.2,-4) {\tiny$(\gratio,\gratio)$};
\node[rotate=29] at (-2.54,-3.95) {\tiny$(\gratio,\gratio)$};
\node[rotate=29] at (-5.35,-3.94) {\tiny$(\gratio,\gratio)$};
\node[rotate=29] at (-8.2,-4) {\tiny$(\gratio,\gratio)$};
\node[rotate=29] at (-11,-4) {\tiny$(\gratio,\gratio)$};

\draw[black] (-0.4,-8.6) node {\footnotesize$M_0$};
\draw[black] (-3.2,-8.6) node {\footnotesize$M_2$};
\draw[black] (-6,-8.6) node {\footnotesize$M_4$};
\draw[black] (-8.1,-8.6) node {\footnotesize$\cdots$};
\draw[black] (-11.6,-8.6) node {\footnotesize$M_{n-1}$};

\draw [->](-1.6,-7.7) -- (-0.6,-8.3);
\draw [->](-3,-8.3) -- (-2,-7.7);
\draw [->](-4.4,-7.7) -- (-3.4,-8.3);
\draw [->](-5.8,-8.3) -- (-4.8,-7.7);
\draw [->](-7.2,-7.7) -- (-6.2,-8.3);
\draw[black] (-8.1,-7.26) node {\footnotesize$\cdots$};
\draw [->](-10,-7.7) -- (-9,-8.3);
\draw [->](-11.4,-8.3) -- (-10.4,-7.7);

\node[rotate=-29] at (-1,-7.8) {\tiny$(\croot_1,\croot_1)$};
\node[rotate=-29] at (-3.8,-7.8) {\tiny$(\croot_1,\croot_1)$};
\node[rotate=-29] at (-6.6,-7.8) {\tiny$(\croot_1,\croot_1)$};
\node[rotate=-29] at (-9.4,-7.8) {\tiny$(\croot_1,\croot_1)$};
\node[rotate=29] at (-2.6,-7.8) {\tiny$(\croot_1,\croot_1)$};
\node[rotate=29] at (-5.4,-7.8) {\tiny$(\croot_1,\croot_1)$};
\node[rotate=29] at (-11.02,-7.84) {\tiny$(\croot_1,\croot_1)$};

\draw[black] (-1.8,-7.26) node {\footnotesize$M_1$};
\draw[black] (-4.6,-7.26) node {\footnotesize$M_3$};
\draw[black] (-8.1,-7.9) node {\footnotesize$\cdots$};
\draw[black] (-10.2,-7.26) node {\footnotesize$M_{n-2}$};
\end{tikzpicture}
	\caption{Reduced Auslander-Reiten quivers of type $\coxH_3$ (top) and $\coxI_2(2n+1)$ (bottom). The top is a reduction of the Auslander-Reiten quiver of $\mod*KQ^{\coxD_6}$ with respect to the action of $\goldsr$. The bottom is a reduction of $\mod*KQ^{\coxA_{2n}}$ with respect to $\chebsr\ps{2n+1}$. Unlabelled arrows have valuation (1,1).} \label{fig:ReducedAR}
\end{figure}

\section{Cluster categories and mutations arising from weighted foldings}
\label{section-mutations}
Throughout this section, let $F\colon Q^\gend \rightarrow Q^{\gend'}$ be a weighted folding of quivers with $\gend' = \{\coxH_4,\coxH_3,\coxI_2(2n+1)\}$. By the results of the last section, the folding $F$ gives rise to an associated $R_+$-coefficient category $\calM_F$ by Theorem~\ref{thm:Action}, and the corresponding bounded derived category $\der_F=\bder(\calM_F)$ inherits the semiring action on $\calM_F$. Here $R_+=\chebsr\ps{2n+1}$, where $n=2$ in the $\coxH$-type cases.

The category $\calM_F$ is the module category of a finite-dimensional hereditary algebra of finite-representation type. Thus it makes sense to construct its \emph{cluster category} $\clus_F$ (see for example \cite{ClusterTiltingSurvey}). Recall that the cluster category of $\calM_F$ is the quotient category $\clus_F = \der_F / (\dfart\inv\sus)$. This category inherits the Auslander-Reiten triangles and translation from $\der_F$ with the identification $\sus M \cong \dfart M$. Moreover, $\clus_F$ naturally inherits the semiring action on $\der_F$. Specifically, one has a canonical quotient functor $G \colon \der_F \rightarrow \clus_F$ that is triangulated (due to \cite{KellerOrbitCats}) such that $M \cong N \Rightarrow G(M) \cong G(N)$ and $\cfart G(M) \cong G(\dfart M)$ for any $M,N \in \der_F$. The $R_+$-action on $\clus_F$ is then defined by $rG = Gr$, which is well-defined because the $R_+$-action on $\der_F$ preserves the degree of indecomposable objects and commutes with both $\sus$ and $\dfart$. Given any minimal set of $R_+$-generators $\Gamma_{\der_F}$ of $\der_F$ that is closed under $\sus$ and $\dfart$, one can similarly define a set of minimal $R_+$-generators for $\clus_F$ as the set of iso-classes
\begin{equation*}
	\Gamma_{\clus_F} = \{[M] \in \clus_F : M \in G(\Gamma_{\der_F})\},
\end{equation*}
and we can define the corresponding reduced Auslander-Reiten quiver of $\clus_F$. We also have an analogue of Remark~\ref{rem:SimplerGen} for sets of minimal $R_+$-generators.

Further recall from \cite{BMRRT} the following definitions, which we will later generalise to the setting of $R_+$-coefficient categories.
\begin{defn}[\cite{BMRRT}]
	Let $\clus_{A}$ be the cluster category of a finite-dimensional hereditary algebra $A$. An object $T \in \clus_{A}$ is called a \emph{(cluster)-tilting object} if
	\begin{enumerate}[label=(\roman*)]
		\item $T$ is \emph{rigid}: $\Hom_{\clus_{A}}(T, \tau_{\clus_{A}} T)=0$; and
		\item $T$ is maximal: if there exists $X \in \clus_{A}$ such that $\Hom_{\clus_{A}}(T \oplus X,\tau_{\clus_{A}}(T \oplus X))=0$, then $X \in \add T$, where $\add T$ is the full subcategory of $\clus_{A}$ containing direct summands of finite direct sums of copies of $T$.
	\end{enumerate}
	We say $T$ is \emph{basic} if the indecomposable direct summands of $T$ are pairwise non-isomorphic. We call the algebras Morita equivalent to $\End_{\clus_{A}}(T)^{\op}$ \emph{cluster-tilted algebras}.
\end{defn}

\begin{prop}[\cite{BMRRT}] \label{prop:TiltingSummands}
	Let $T \in \clus_{A}$ be a basic tilting object. Then $T$ has $|Q_0|$ indecomposable direct summands.
\end{prop}

\begin{prop}[\cite{BMRTilting}]
	Let $T \in \clus_{A}$ be a tilting object. Then the functor $\Hom_{\clus_A}(T,{-})$ induces an equivalence of categories $\clus_A / \add \tau_{\clus_A} T \rightarrow \mod*\End_{\clus_{A}}(T)^{\op}$.
\end{prop}

\begin{defn}[\cite{BMRRT}]
	An object $T \in \clus_{A}$ is called an \emph{almost complete tilting object} if $\Hom_{\clus_{A}}(T,\tau_{\clus_{A}} T) = 0$ and there exists an indecomposable object $X \in \clus_{A}$ such that $T \oplus X$ is a cluster-tilting object. We call such an object $X$ a \emph{complement} of $T$.
\end{defn}

\begin{prop}[\cite{BMRRT,BMRMutation}] \label{prop:TiltingComplements}
	Let $T \in \clus_{\catH}$ be an almost complete tilting object. Then $T$ has exactly two complements $X$ and $X'$. The corresponding cluster-tilted algebras $\End_{\clus_{A}}(T \oplus X')^{\op}$ and $\End_{\clus_{A}}(T \oplus X'')^{\op}$ are respectively isomorphic to algebras $KQ'/I'$ and $KQ''/I''$ such that $Q'$ differs from $Q''$ by a mutation at a single vertex.
\end{prop}

For the next result from classical cluster-tilting theory, we require the notion of \emph{approximations} in additive categories, which is similar to the notion of almost split morphisms in Auslander-Reiten theory. Let $\mathcal{X}$ be an additive category and let $\mathcal{Y} \subseteq \mathcal{X}$ be an additive subcategory. We say a morphism $f \in \Hom_{\mathcal{X}}(Y,X)$ with $Y \in \mathcal{Y}$ is called a \emph{right $\mathcal{Y}$-approximation} if for any morphism $f' \in \Hom_{\mathcal{X}}(Y', X)$ with $Y' \in \mathcal{Y}$, there exists a morphism $g \in \Hom_{\mathcal{X}}(Y',Y)$ such that the following diagram commutes:
\begin{equation*}
	\xymatrix{Y' \ar[d]^-g \ar[dr]^-{f'} \\ Y \ar[r]^-f & X}
\end{equation*}
We say that a morphism $f \in \Hom_{\mathcal{X}}(X,X')$ is \emph{right minimal} if for any $g \in \End_{\mathcal{X}}(X)$ such that $fg=f$, the morphism $g$ is an automorphism.

There is also a notion of \emph{left $\mathcal{Y}$-approximations} and the notion of \emph{left minimal} morphisms, which are both dual to their right counterparts. We say that a morphism is a \emph{right} (resp. \emph{left}) \emph{minimal $\mathcal{Y}$-approximation} if it is both right (left) minimal and a right (left) $\mathcal{Y}$-approximation.

\begin{prop}[\cite{BMRRT}] \label{prop:approximations}
	Let $T \in \clus_A$ be an almost complete tilting object with complements $X_1, X_2 \in \clus_A$. Then there exist triangles
	\begin{align*}
		\xymatrix@1{X_1 \ar[r]^-{g} & E} &\xymatrix@1{\ar[r]^-{f} & X_2 \ar[r] &  \tau_{\clus_A} X_1} \\
		\xymatrix@1{X_2 \ar[r]^-{g'} & E'} &\xymatrix@1{\ar[r]^-{f'} & X_1 \ar[r] &  \tau_{\clus_A} X_2}
	\end{align*}
	in $\clus_A$ such that $f$ and $f'$ are right minimal $\add T$-approximations, and $g$ and $g'$ are left minimal $\add T$-approximations.
\end{prop}

We shall now generalise these concepts and results to our setting of cluster categories arising from foldings. For the theory that follows, we will define the following sets for each object $M \in \clus_F$:
\begin{align*}
	\gen_M &= \{L \in \clus_{F} : L \cong rM, r \in R_+\} \\
	\mathcal{I}_M &= \{L \in \mathcal{G}_M : L \text{ indecomposable}\}
\end{align*}

The latter of these sets satisfies some small but important properties which we will use later. Namely, we have the following lemmas:

\begin{lem}\label{lem:IndSetSize}
	Let $F\colon Q^\gend \rightarrow Q^{\gend'}$ be a folding of quivers, and suppose $\Gamma$ is a minimal set of $R_+$-generators for $\clus_F$. Then for any $M \in \Gamma$ we have $|Q^\gend_0| = |\mathcal{I}_{M}| |Q^{\gend'}_0|$.
\end{lem}
\begin{proof}
	For foldings onto $H$-type quivers, this follows from the fact that $\mathcal{I}_M = \{M, \gratio M\}$ for any $M \in \Gamma$ and that $|Q^\gend_0|=2 |Q^{\gend'}_0|$. For foldings $F\colon Q^{\coxA_{2n}} \rightarrow Q^{\coxI_2(2n+1)}$, this follows from the shape of the Auslander-Reiten quiver of a bipartite $\coxA_{2n}$ quiver and Remark~\ref{rem:FoldingMultiplication}(c).
\end{proof}

\begin{lem} \label{lem:DisjointGeneration}
	Suppose $\Gamma$ is a minimal set of $R_+$-generators of $\clus_F$. Then $M_1,M_2 \in \Gamma$ are distinct if and only if $\mathcal{I}_{M_1} \cap \mathcal{I}_{M_2} = \emptyset$.
\end{lem}
\begin{proof}
	If $M_1 \cong M_2$, then we clearly have that $\mathcal{I}_{M_1} \cap \mathcal{I}_{M_2} \neq \emptyset$ by Lemma~\ref{lem:IndSetSize}. Conversely, if $M_1 \not\cong M_2$ then note that $\mathcal{I}_{M_1} \cap \mathcal{I}_{M_2} = \emptyset$ by Proposition~\ref{prop:DerivedProj}, Remark~\ref{rem:FoldingMultiplication}(c), and the fact that the action of $R_+$ commutes with both $\sus$ and $\dfart$.
\end{proof}

We shall now generalise cluster-tilting to our setting.

\begin{defn}
	Let $F \colon Q^\gend \rightarrow Q^{\gend'}$ be a weighted folding, and let $\clus_F$ be the cluster category of the associated $R_+$-coefficient category. Let $\Gamma$ be a minimal set of $R_+$-generators for $\clus_F$. We say that an object $T \in \clus_F$ is \emph{basic (cluster)-$R_+$-tilting} if the following hold:
	\begin{enumerate}[label=(T\arabic*)]
		\item $T \cong \bigoplus_{i=1}^m T_i$ with each $T_i \in \Gamma$ and $T_i \not\cong T_j$ for any $i \neq j$.
		\item $T$ is $R_+$-rigid: $\Hom_{\clus_F}(T',\cfart T'') = 0$ for any $T',T'' \in \mathcal{G}_T$.
		\item $T$ is maximal: if there exists $X \in \Gamma$ such that $T \oplus X$ is $R_+$-rigid, then $X$ is isomorphic to a direct summand of $T$.
	\end{enumerate}
\end{defn}

Before we proceed with our main results for this section, we must prove some technical lemmas which further quantify the above conditions.

\begin{lem} \label{lem:RPRigidEquiv}
	Let $\Gamma$ be a minimal set of $R_+$-generators of $\clus_F$, and let $T= \bigoplus_{i=1}^{|T|} T_i \in \clus_F$ with each $T_i \in \Gamma$. Then $T$ is $R_+$-rigid if and only if for any $i,j \in \{1,\ldots,|T|\}$ we have $\Hom_{\clus_F}(Z_i,\cfart Z_j)=0$ for any $Z_i \in \mathcal{I}_{T_i}$ and $Z_j \in \mathcal{I}_{T_j}$.
\end{lem}
\begin{proof}
	By definition, we have $T' \in \gen_T$ if and only if we have $T' \cong rT \cong \bigoplus_{i=1}^{|T|} r T_i$ for some $r \in R_+$. For each such $T' \cong rT \in \gen_T$, write $T'_i = rT_i$. Now since $T_i \in \Gamma$, every indecomposable direct summand $Z_i$ of $T'_i$ is such that $Z_i \in \mathcal{I}_{T_i}$. Conversely, for any $Z_i \cong r' T_i \in \mathcal{I}_{T_i}$, there exists $T'' \cong r' T \in \gen_T$ such that $Z_i$ is a direct summand of $T''$. Thus, we have $\Hom_{\clus_F}(T', \cfart T'')=0$ for any $T', T'' \in \gen_T$ if and only if for any $i,j \in \{1,\ldots,|T|\}$, we have $\Hom_{\clus_F}(Z_i,\cfart Z_j)=0$ for any $Z_i \in \mathcal{I}_{T_i}$ and $Z_j \in \mathcal{I}_{T_j}$, as required.
\end{proof}

\begin{lem}\label{lem:ClusExact}
	Let $\Gamma$ be a minimal set of $R_+$-generators for $\clus_F$ and let $M_1,M_2 \in \Gamma$. Suppose there exists $i$ and $j$ such that $\Hom_{\clus_F}(\croot_i M_1, \croot_j M_2)\neq 0$. Then
	\begin{enumerate}[label=(\alph*)]
		\item $\Hom_{\clus_F}(\croot_1\croot_i M_1, \croot_{j-1} M_2)\neq 0$ and $\Hom_{\clus_F}(\croot_1\croot_i M_1, \croot_{j+1} M_2)\neq 0$;
		\item $\Hom_{\clus_F}(\croot_{i-1} M_1, \croot_1\croot_j M_2)\neq 0$ and $\Hom_{\clus_F}(\croot_{i+1} M_1, \croot_1\croot_j M_2)\neq 0$.
	\end{enumerate}
\end{lem}
\begin{proof}
	We have $\Hom_{\der_F}(\sus^k\croot_i\wt{M}_1,\sus^l\croot_j\wt{M}_2) \neq 0$, where $\sus^k\wt{M}_1,\sus^l\wt{M}_2 \in \der_F$ are representatives of $M_1, M_2 \in \clus_F$, and by an abuse of notation, $\wt{M}_1, \wt{M}_2 \in \calM_F$. In particular, either $l=k$ or $l=k+1$.
	
	Consider the case where $l=k$. Then we have a pair of exact sequences
	\begin{align*}
		&\xymatrix@1{
			0
				\ar[r] &
			\Ker \wt{g}
				\ar[r]&
			\croot_i \wt{M}_1
				\ar[r]^-{\wt{g}} &
			\im \wt{f}
				\ar[r]&
			0
		} \\
		&\xymatrix@1{
			0
				\ar[r] &
			\im \wt{f}
				\ar[r]^-{\wt{h}}&
			\croot_j \wt{M}_2
				\ar[r] &
			\Coker \wt{h}
				\ar[r]&
			0
		},
	\end{align*}
	where $\wt{f} \in \Hom_{\calM_F}(\croot_i \wt{M}_1,\croot_j \wt{M}_2)$ is non-zero. Since $\croot_i \wt{M}_1$ and $\croot_j \wt{M}_2$ are non-isomorphic indecomposable modules, $\Coker \wt{h}$ is either zero (if $\wt{h}$ is surjective) or it is isomorphic to a direct sum of indecomposable modules that reside in columns of the Aulsander-Reiten quiver distinct from the columns containing $\wt{M}_1$ and $\wt{M}_2$. Importantly, we note from this that $\croot_{j'}\wt{M}_2$ is not a direct summand of $\Coker \wt{h}$ for any $j'$. Multiplying by $\croot_1$ yields
	\begin{align*}
		&\xymatrix@1{
			0
				\ar[r] &
			\croot_1 \Ker \wt{g}
				\ar[r]&
			\croot_1\croot_i \wt{M}_1
				\ar[r]^-{\croot_1\wt{g}} &
			\croot_1\im \wt{f}
				\ar[r]&
			0
		} \\
		&\xymatrix@1{
			0
				\ar[r] &
			\croot_1\im \wt{f}
				\ar[r]^-{\croot_1\wt{h}}&
			\croot_{j-1} \wt{M}_2 \oplus \croot_{j+1} \wt{M}_2
				\ar[r] &
			\croot_1\Coker \wt{h}
				\ar[r]&
			0
		}.
	\end{align*}
	Now we necessarily have that $\im \croot_1 \wt{h} \cap \croot_{j-1} \wt{M}_2 \neq \emptyset$ and $\im \croot_1 \wt{h} \cap \croot_{j+1} \wt{M}_2\neq \emptyset$: otherwise either $\croot_{j-1} \wt{M}_2$ or $\croot_{j+1} \wt{M}_2$ is a direct summand of $\croot_1 \Coker \wt{h}$, which by Remark~\ref{rem:FoldingMultiplication}(c) would contradict the previous observation that $\croot_j \wt{M}_2$ is not a direct summand of $\Coker \wt{h}$. Thus, there must exist non-zero morphisms $\wt{f}'_{j-1} \colon \croot_1\croot_{i} \wt{M}_1 \rightarrow \croot_{j-1}\wt{M}_2$ and $\wt{f}'_{j+1} \colon \croot_1\croot_{i} \wt{M}_1 \rightarrow \croot_{j+1}\wt{M}_2$. This implies that $\Hom_{\clus_F}(\croot_1\croot_{i}M_1, \croot_{j-1} M_2) \neq 0$ and $\Hom_{\clus_F}(\croot_1\croot_{i}M_1, \croot_{j+1} M_2) \neq 0$, as required. That we also have $\Hom_{\clus_F}(\croot_{i-1} M_1, \croot_1\croot_j M_2)\neq 0$ and $\Hom_{\clus_F}(\croot_{i+1} M_1, \croot_1\croot_j M_2)\neq 0$ follows by dual reasoning to the above.
	
	Now consider the case where $l=k+1$. Then $\Ext^1_{\calM_F}(\croot_i \wt{M}_1, \croot_j \wt{M}_2) \neq 0$ and consequently we have a morphism of non-split exact sequences
	\begin{equation*}
		\xymatrix{
			0
				\ar[r]
			& \croot_j \wt{P}_1
				\ar[r]
				\ar[d]^-{\wt{g}}
			& \croot_j \wt{P}_0
				\ar[r]
				\ar[d]
			& \croot_j \wt{M}_2
				\ar[r]
				\ar@{=}[d]
			& 0
			\\
			0
				\ar[r]
			& \croot_i \wt{M}_1
				\ar[r]
			& \wt{N}
				\ar[r]
			& \croot_j \wt{M}_2
				\ar[r]
			& 0,
		}
	\end{equation*}
	where the top row is a projective resolution of $\croot_j \wt{M}_2$, and $\wt{N}$ is determined by $\wt{g}$. Multiplying by $\croot_1$ yields
	\begin{equation*}
		\xymatrix{
			0
				\ar[r]
			& \croot_{j-1} \wt{P}_1 \oplus \croot_{j+1} \wt{P}_1
				\ar[r]
				\ar[d]^-{\croot_1 \wt{g}}
			& \croot_{j-1} \wt{P}_0 \oplus \croot_{j+1} \wt{P}_0
				\ar[r]
				\ar[d]
			& \croot_{j-1} \wt{M}_2 \oplus \croot_{j+1} \wt{M}_2
				\ar[r]
				\ar@{=}[d]
			& 0
			\\
			0
				\ar[r]
			& \croot_1 \croot_i \wt{M}_1
				\ar[r]
			& \croot_1\wt{N}
				\ar[r]
			& \croot_{j-1} \wt{M}_2 \oplus \croot_{j+1} \wt{M}_2
				\ar[r]
			& 0,
		}
	\end{equation*}
	where by the previous argument, we have $\croot_1 \wt{g} = \left(\begin{smallmatrix} \wt{g}_- & \wt{g}_+ \end{smallmatrix}\right)$ with $\wt{g}_-,\wt{g}_+ \neq 0$. Restricting to either direct summand yields a non-split exact sequence. The dual argument also applies, which proves (a) and (b) for this final case.
\end{proof}

\begin{lem} \label{lem:ClusHomCols}
	Let $\Gamma$ be a minimal set of $R_+$-generators for $\clus_F$, and let $M_1,M_2 \in \Gamma$.
	\begin{enumerate}[label=(\alph*)]
		\item Suppose there exists $L_2 \in \mathcal{I}_{M_2}$ such that $\Hom_{\clus_F}(L_1, L_2)=0$ for any $L_1 \in \mathcal{I}_{M_1}$. Then $\Hom_{\clus_F}(L'_1, L'_2)=0$ for any $L'_1 \in \mathcal{I}_{M_1}$ and $L'_2 \in \mathcal{I}_{M_2}$.
		\item Suppose there exists $L_1 \in \mathcal{I}_{M_1}$ such that $\Hom_{\clus_F}(L_1, L_2)=0$ for any $L_2 \in \mathcal{I}_{M_2}$. Then $\Hom_{\clus_F}(L'_1, L'_2)=0$ for any $L'_1 \in \mathcal{I}_{M_1}$ and $L'_2 \in \mathcal{I}_{M_2}$.
	\end{enumerate}
\end{lem}
\begin{proof}
	(a) By Remark~\ref{rem:SimplerGen}, every element of $\mathcal{I}_{M_k}$ may be written as $\croot_j M_k$. We will prove the contrapositive of the lemma statement in this notation. Namely, that if there exist $i$ and $j$ such that $\Hom_{\clus_F}(\croot_i M_1, \croot_j M_2) \neq 0$, then for any $j'$, there exists $i'$ such that $\Hom_{\clus_F}(\croot_{i'} M_1, \croot_{j'} M_2) \neq 0$. The result is trivial in the cases where either $M_1 \cong M_2$ or where $j=0$, so we will assume neither is the case (a priori). For all other cases, the result is actually a straightforward consequence of the inductive application of Lemma~\ref{lem:ClusExact}(a).
	
	(b) The proof is similar to (a), with the use of Lemma~\ref{lem:ClusExact}(b) instead.
\end{proof}

The condition of being $R_+$-tilting is closely related to the classical notion of tilting. In fact, the next theorem shows that basic $R_+$-tilting objects of $\clus_F$ correspond to a subset of basic tilting objects of $\clus_F$.

\begin{thm} \label{thm:TiltingEquiv}
	Let $T= \bigoplus_{i=1}^{|T|} T_i \in \clus_F$, and let
	\begin{equation*}
		\wh{T} = \bigoplus_{i=1}^{|T|}\bigoplus_{Z \in \mathcal{I}_{T_i}} Z \in \clus_F.
	\end{equation*}
	Then $T$ is a basic $R_+$-tilting object if and only if $\wh{T}$ is a basic tilting object.
\end{thm}
\begin{proof}
	If $T$ basic then $\wh{T}$ basic by Lemma~\ref{lem:DisjointGeneration} and property (T1). The converse follows from the fact that $T_i \in \mathcal{I}_{T_i}$.

	By Lemma~\ref{lem:RPRigidEquiv}, $T$ is $R_+$-rigid if and only if for any $i,j \in \{1,\ldots, |T|\}$ we have $\Hom_{\clus_F}(Z_i, \cfart Z_j) =0$ for any $Z_i \in \mathcal{I}_{T_i}$ and any $Z_j \in \mathcal{I}_{T_j}$. But this is true if and only if $\Hom_{\clus_F}(\widehat{T},\cfart \wh{T})=0$, and thus $T$ is $R_+$-rigid if and only if $\wh{T}$ is rigid.
	
	We can now prove the proposition statement. The remaining condition to check is the equivalence of the maximality conditions, which is easier to prove by the contrapositive statement. So suppose $T$ is  basic and $R_+$-rigid, but not maximal. Then there must exist some $X \in \Gamma$ such that $T \oplus X$ is basic and $R_+$-rigid. This implies that $\wh{T} \oplus \wh{X}$ is basic and rigid, where $\wh{X}\cong\bigoplus_{Z \in \mathcal{I}_{X}} Z$. But then clearly $\wh{T}$ is not maximal, because $\wh{X} \not \in \add \wh{T}$ by the fact that $\wh{T} \oplus \wh{X}$ is basic.
	
	Conversely, suppose $\wh{T}$ is basic and rigid, but not maximal. Then by Proposition~\ref{prop:TiltingSummands}, we must have $|\wh{T}| \neq |Q_0^\gend|$. In fact, we actually have $|\wh{T}| \leq |Q_0^\gend| - |\mathcal{I}_X|$ for some $X \in \Gamma$. Let $Y \in \clus_F$ be such that $\wh{T} \oplus Y$ is basic and rigid, and let $X \in \Gamma$ be such that $Y \in \mathcal{I}_X$. Then by Lemma~\ref{lem:ClusHomCols}, we know that $\wh{T} \oplus \wh{X}$ is basic and rigid, where $\wh{X} \cong \bigoplus_{Z \in \mathcal{I}_{X}} Z$. But this is true if and only if $T \oplus X$ is basic and $R_+$-rigid. So $T$ cannot be maximal, as required.
\end{proof}

Henceforth, we will refer to the basic tilting object $\wh{T} \in \clus_F$ defined in the above theorem as the tilting object that corresponds to the basic $R_+$-tilting object $T$.

\begin{thm}
  \label{thm:Length}
 	Let $T \in \clus_F$ be a basic $R_+$-tilting object. Then $T$ has $|Q_0^{\gend'}|$ indecomposable direct summands.
\end{thm}
\begin{proof}
	This is a consequence of Lemma~\ref{lem:IndSetSize}, Theorem~\ref{thm:TiltingEquiv} and Proposition~\ref{prop:TiltingSummands}.
\end{proof}

\begin{defn}
	Let $F \colon Q^\gend \rightarrow Q^{\gend'}$ be a weighted folding, and let $\clus_F$ be the cluster category of the associated $R_+$-coefficient category. Let $\Gamma$ be a minimal set of $R_+$-generators for $\clus_F$. An object $T \in \clus_F$ is called an \emph{almost complete basic $R_+$-tilting object} if $T$ is $R_+$-rigid and there exists an object $X \in \Gamma$ such that $T \oplus X$ is a basic $R_+$-tilting object. We call such an object $X$ a \emph{complement} of $T$.
\end{defn}

\begin{thm}
\label{thm:Complements}
  Let $T$ be an almost complete basic $R_+$-tilting object of $\clus_F$. Then $T$ has exactly two complements.
\end{thm}
\begin{proof}
	For foldings $F\colon Q^{\coxA_{2n}} \rightarrow Q^{\coxI_2(2n+1)}$, the theorem is a straightforward consequence of well-known results on the Auslander-Reiten theory of $\coxA$-type quivers (cf. \cite[Lemma 2.3, 2.5, Theorem 2.13]{CCS} for the setting of cluster categories and \cite[pp. 166-172]{ButlerRingel}, \cite[Sec. 2, Theorem]{CBMaps} for module categories of type $\coxA$). In particular, for any indecomposable object $M \in \clus_F$, let $N_1,N_2 \in \clus_F$ be the indecomposable objects in the top-most and bottom-most rows of the Auslander-Reiten quiver that are given by the rays of irreducible morphisms of source $M$ (of which there are at most two). Note that if $M$ is already located on the top-most (resp. bottom-most) row of the Auslander-Reiten quiver, then we set $N_1 = M$ (resp. $N_2 = M$). Next, let $L \in \clus_F$ be the indecomposable object given by the intersection of the (unique) rays of source $N_1$ and $N_2$. Note that if $N_1= M$ (resp. $N_2 = M$), then $L=N_2$ (resp. $L= N_1$). Then $\Hom_{\clus_F}(M,M') \neq 0$ if and only if $M'$ is located within the rectangular lattice of the Auslander-Reiten quiver spanned by the corners $M$, $N_1$, $N_2$ and $L$ (inclusive of the boundary). See Figure~\ref{fig:TypeAHoms} for an illustration.
	
	\begin{figure}[b]
		\centering
		\def\z{0.5}
		\begin{tikzpicture}
			\foreach \x in {0,2,4,6,8} {
				\foreach \y in {0,2,4} {
					\coordinate (v\x\y) at ($\x*(\z, 0)+\y*(0,\z)$);
					\draw (v\x\y) ellipse (0.05 and 0.05);
				}
			}
			\foreach \x in {1,3,5,7} {
				\foreach \y in {1,3,5} {
					\coordinate (v\x\y) at ($\x*(\z, 0)+\y*(0,\z)$);
					\draw (v\x\y) ellipse (0.05 and 0.05);
				}
			}
			
			\foreach \x in {0,8} {
				\foreach \y in {0,2,4} {
					\draw[] (v\x\y) ellipse (0.05 and 0.05);
				}
			}
		 	\foreach \x in {1} {
		 		\foreach \y in {1,3,5} {
		 			\draw[] (v\x\y) ellipse (0.05 and 0.05);
		 		}
		 	}
			
			\draw[cyan] (v00) -- (v55);
			\draw[cyan, fill, fill opacity = 0.2] (v02) -- (v35) -- (v53) -- (v20) -- (v02);
			\draw[cyan, fill, fill opacity = 0.2] (v04) -- (v15) -- (v51) -- (v40) -- (v04);
			\draw[anchor = west] (v02) node {\footnotesize $M$};
			\draw[anchor = north] (v35) node {\footnotesize $N_1$};
			\draw[anchor = south] (v20) node {\footnotesize $N_2$};
			\draw[anchor = east] (v53) node {\footnotesize $L$};
		\end{tikzpicture}
		\caption{The Auslander-Reiten quiver of $\clus_F$ for a folding $F\colon Q^{\coxA_6} \rightarrow Q^{\coxI_2(7)}$, with irreducible morphisms omitted. The leftmost column represents an object $T \in \Gamma$ which is almost complete $R_+$-tilting. $\Hom_{\clus_F}(M,M') \neq 0$ for an indecomposable $M'$ if and only if $M'$ resides in the shaded region spanned by the corners $M$, $N_1$, $N_2$ and $L$ (inclusive of the boundary). $\Hom_{\clus_F}(\wh{T},M') \neq 0$ for an indecomposable $M'$ if and only if $M'$ resides within any shaded region (or line).} \label{fig:TypeAHoms}
	\end{figure}
	
	Now in the case of the folding $F\colon Q^{\coxA_{2n}} \rightarrow Q^{\coxI_2(2n+1)}$, we have $T \in \Gamma$ by Theorem~\ref{thm:Length}. By Theorem~\ref{thm:TiltingEquiv}, $T$ corresponds to a basic and rigid object $\wh{T}$, which is precisely the direct sum of indecomposable objects in a single column of the Auslander-Reiten quiver of $\clus_F$. It therefore follows from the exposition above that $\Hom_{\clus_F}(\wh{T},M') =0$ for any indecomposable object $M'$ that resides in a column of the Auslander-Reiten quiver between $2n$ to $2n+2$ places to the right of the column containing $T$ (inclusive). In fact, the converse is also true: $\Hom_{\clus_F}(\wh{T},M') \neq 0$ if $L$ resides in a column at most $2n-1$ places to the right of $T$, as illustrated in Figure~\ref{fig:TypeAHoms}. Thus, $\Hom_{\clus_F}(\wh{T},\cfart M') =0$ if and only if $M'$ is an indecomposable residing in a neighbouring column to $T$, or the same column as $T$. Consequently, $T$ has exactly two complements $X$ and $X'$, which correspond to the neighbouring columns of $T$ in the Aulsander-Reiten quiver. This proves the result for the folding $F\colon Q^{\coxA_{2n}} \rightarrow Q^{\coxI_2(2n+1)}$. Thus, we will assume for the remainder of the proof that $F$ is instead a folding onto a $\coxH$-type quiver.

	First, we show that $T$ has at least two complements. By definition, $T$ has at least one complement $X_1$. Let $\wh{T} \oplus \wh{X}_1 \in \clus_F$ be the corresponding tilting object, where $\wh{X}_1=X_1 \oplus \gratio X_1$. Consider the object $\wh{T} \oplus \wh{X}_1 / X_1 \in \clus_F$. This object is almost tilting, and by Proposition~\ref{prop:TiltingComplements}, has exactly two complements, one of which is $X_1$. Let $Y$ be the other complement and let $X_2 \in \Gamma$ be such that $Y \in \mathcal{I}_{X_2}=\{X_2, \gratio X_2\}$. Then $\wh{T} \oplus Y$ is basic and rigid, and $X_2$ is necessarily not isomorphic to any direct summand of $T \oplus X_1$. By Lemma~\ref{lem:ClusHomCols}, this implies that $\wh{T} \oplus \wh{X}_2$ is basic and rigid, where $\wh{X}_2 = X_2 \oplus \gratio X_2$. By Lemma~\ref{lem:IndSetSize} and Proposition~\ref{prop:TiltingSummands}, $\wh{T} \oplus \wh{X}_2$ is also maximal and hence tilting. Thus $T \oplus X_2$ is $R_+$-tilting (Theorem~\ref{thm:TiltingEquiv}), and hence $X_2$ is another complement of $T$. So $T$ has at least two complements.
	
	We now show that $T$ has at most two complements. Let $X_1$ and $X_2$ be as previously stated, and let $Y' \not\cong Y$ be the other object in $\mathcal{I}_{X_2}$. To summarise, we already know that
	\begin{equation*}
		\wh{T} \oplus X_1 \oplus \gratio X_1, \qquad \wh{T} \oplus Y \oplus Y', \qquad
		\wh{T} \oplus Y \oplus \gratio X_1
	\end{equation*}
	are basic tilting objects. The first step of the proof is to show that $\wh{T} \oplus X_1 \oplus Y'$ is also a tilting object, and thus that $X_1$ is a complement to $\wh{T} \oplus Y'$. In this regard, we know that $\Hom_{\clus_F}(X_1, \cfart Y) \neq 0$ or $\Hom_{\clus_F}(Y, \cfart X_1) \neq 0$, otherwise $\wh{T} \oplus X_1 \oplus \gratio X_1 \oplus Y$ would be tilting. In the former case of $\Hom_{\clus_F}(X_1, \cfart Y) \neq 0$, suppose for a contradiction that $\Hom_{\clus_F}(X_1, \cfart Y') \neq 0$. Then we have $\Hom_{\clus_F}(\gratio X_1, \cfart \gratio Y') \neq 0$. Lemma~\ref{lem:ClusExact}(a) then implies that we have $\Hom_{\clus_F}(\gratio X_1, \cfart Y) \neq 0$, which contradicts the fact that $\wh{T} \oplus Y \oplus \gratio X_1$ is tilting. So we must have $\Hom_{\clus_F}(X_1, \cfart Y')=0$. The dual argument applies to the latter case of $\Hom_{\clus_F}(Y, \cfart X_1) \neq 0$. Thus, $\wh{T} \oplus X_1 \oplus Y'$ must also be tilting.
	
	Suppose for a contradiction that $X_3$ is a complement of $T$ distinct from $X_1$ and $X_2$. Let $Z \not\cong X_3$ be a complement of the almost complete tilting object $\wh{T} \oplus \gratio X_3$. Then $Z \not\in \mathcal{I}_{X_1} \cup \mathcal{I}_{X_2} \cup \mathcal{I}_{X_3}$, since this would otherwise contradict the fact that almost complete tilting objects in $\clus_F$ have exactly two complements. Thus, the existence of a third distinct complement of $T$ implies the existence of a fourth distinct complement $X_4$ with $Z \in \mathcal{I}_{X_4}$. By similar reasoning to the above, there must also exist distinct basic tilting objects
	\begin{align*}
		&\wh{T} \oplus X_3 \oplus \gratio X_3, & &\wh{T} \oplus Z \oplus Z', \\
		&\wh{T} \oplus Z \oplus \gratio X_3, 	& &\wh{T} \oplus X_3 \oplus Z'
	\end{align*}
	with $Z,Z' \in \mathcal{I}_{X_4}$.
	
	By Proposition~\ref{prop:approximations}, there exist triangles
	\begin{align*}
		\xymatrix@1{X_1 \ar[r] & E} &\xymatrix@1{\ar[r]^-f & Y \ar[r] & \cfart X_1} \\
		\xymatrix@1{\gratio X_1 \ar[r]^-{f'} & E'} &\xymatrix@1{\ar[r] & Y' \ar[r] & \cfart \gratio X_1} \\
		\xymatrix@1{X_3 \ar[r] & E''} &\xymatrix@1{\ar[r]^-{f''} & Z \ar[r] & \cfart X_3}.
	\end{align*}
	In particular, $f$ is both a minimal right $\add (\wh{T} \oplus \gratio X_1)$-approximation and a minimal right $\add (\wh{T} \oplus Y')$-approximation. So $f$ must be a minimal right $\add \wh{T}$-approximation, and hence $E \in \add \wh{T}$. By a similar argument, we also have $E',E'' \in \add \wh{T}$. Applying the functor $\Hom_{\clus_F}({-},X_3)$ to the first and second triangles, and applying the functors $\Hom_{\clus_F}({-},X_1)$ and $\Hom_A({-},\gratio X_1)$ to the third triangle yields
	\begin{align*}
		\Hom_{\clus_F}(X_1,\cfart X_3) = \Hom_{\clus_F}(\gratio X_1,\cfart X_3) &= 0 \\
		\Hom_{\clus_F}(X_3,\cfart X_1) = \Hom_{\clus_F}(X_3,\cfart \gratio X_1) &= 0.
	\end{align*}
	But then this implies that $M = \wh{T} \oplus \wh{X}_1 \oplus X_3$ is basic tilting, which cannot possibly be true, since $|M| > |Q^\gend_0|$. Hence our supposition of the existence of a third distinct complement must be false.
\end{proof}

A number of useful corollaries result from the proof of the above Theorem, which we will highlight here.

\begin{cor} \label{cor:ClosedComplements}
	Let $T \in \clus_F$ be a basic almost complete $R_+$-tilting object and let $X_1,X_2 \in \clus_F$ be its distinct complements. For any $Y \in \mathcal{I}_{X_1}$, there exists $Z \in \mathcal{I}_{X_2}$ such that $Z$ is a complement of the basic almost complete tilting object $\wh{T} \oplus \wh{X}_1 / Y$.
\end{cor}
\begin{proof}
	We have already shown this for $\coxH$-type foldings. For $\coxI$-type foldings, this is a straightforward consequence of the Auslander-Reiten theory of $\coxA$-type quivers, which can easily be deduced from Figure~\ref{fig:TypeAHoms}. Suppose that $X_1$ represents the neighbouring column of the Aulsander-Reiten quiver to the right of $T$ (and thus $X_2$ is to the left of $T$). Since the object $T \oplus X_1$ is basic $R_+$-tilting, the object $\wh{T} \oplus \wh{X}_1$ is basic tilting. Consequently, it follows from property (T3) that $\Hom_{\clus_F}(\wh{T} \oplus \wh{X}_1, \cfart M')= 0$ for an indecomposable $M'$ if and only if $M'$ is a summand of $\wh{T} \oplus \wh{X}_1$, or equivalently, if and only if $M'$ resides in the same column as $T$ or $X_1$.
	
	Now let $Y \in \mathcal{I}_{X_1}$ and consider the rectangular lattice of the Auslander-Reiten quiver that determines all indecomposables $M'$ such that $\Hom_{\clus_F}(Y,M')\neq 0$. Necessarily, the rightmost corner of the lattice corresponds to an indecomposable object $\cfart Z \in \mathcal{I}_{\cfart X_2}$. Moreover, it can easily be seen that for any $Y' \in \mathcal{I}_{X_1}$ with $Y' \not\cong Y$, the rightmost corner of the rectangular lattice corresponding to $Y'$ is an idecomposable $\cfart Z' \in \mathcal{I}_{\cfart X_2}$ with $Z' \not\cong Z$.
	
	Taking the quotient $\wh{T} \oplus \wh{X}_1 / Y$ corresponds to removing the rectangular lattice in the Auslander-Reiten quiver originating from $Y$. This frees up precisely one object $\cfart Z$ in the column represented by $\cfart X_2$. Thus, we have $\Hom_{\clus_F}(\wh{T} \oplus \wh{X}_1 / Y,\cfart Z)=0$. In fact, $Z \cong \cfart Y$, so we also have $\Hom_{\clus_F}(Z, \cfart(\wh{T} \oplus \wh{X}_1 / Y))=0$, as required. The argument for where $X_1$ represents the column to the left of $T$ is similar.
\end{proof}

\begin{cor}\label{thm:RPApproximations}
	Let $T \in \clus_F$ be a basic almost complete $R_+$-tilting object, and let $X_1,X_2 \in \clus_F$ be its distinct complements. Let $\wh{T} \in \clus_F$ be the basic rigid object corresponding to $T$. Then for any $Y \in \mathcal{I}_{X_1}$, there exists $Z \in \mathcal{I}_{X_2}$ such that there exist triangles
	\begin{align*}
		\xymatrix@1{Y \ar[r]^-{g} & E} &\xymatrix@1{\ar[r]^-{f} & Z \ar[r] &  \tau_{\clus_F} Y} \\
		\xymatrix@1{Z \ar[r]^-{g'} & E'} &\xymatrix@1{\ar[r]^-{f'} & Y \ar[r] &  \tau_{\clus_F} Z}
	\end{align*}
	in $\clus_F$, where $f$ and $f'$ are right minimal $\add \wh{T}$-approximations, and $g$ and $g'$ are left minimal $\add \wh{T}$-approximations.
\end{cor}
\begin{proof}
	This is a consequence of Corollary~\ref{cor:ClosedComplements}, Proposition~\ref{prop:approximations} and the reasoning used in the proof of the previous Theorem for $\coxH$-type foldings, which can readily be applied to $\coxI$-type foldings too.
\end{proof}

\begin{defn}
	Let $T \in \clus_F$ be a basic $R_+$-tilting object. We define $\End_{\clus_F}(\wh{T})^{\op}$ to be the corresponding \emph{cluster-$R_+$-tilted} algebra, where $\wh{T} \in \clus_F$ is the basic tilting object corresponding to $T$.
\end{defn}

\begin{thm} \label{thm:TiltedRPAction}
	Let $T \in \clus_F$ be a basic $R_+$-tilting object, and let $A_T$ be the corresponding cluster-$R_+$-tilted algebra. Then $\mod*A_T$ has the structure of an $R_+$-coefficient category.
\end{thm}
\begin{proof}
	From \cite{BMRTilting}, it is known that the functor $\Hom_{\clus_F}(\wh{T},{-})$ induces equivalences 
	\begin{equation*}
		H\colon\clus_F/ \add \cfart \wh{T} \rightarrow \mod*A_T \qquad \text{and} \qquad H'\colon\add \wh{T} \rightarrow \proj A_T,
	\end{equation*}
	where $\proj A_T$ is the full subcategory of finitely generated projective $A_T$-modules. This endows $\mod*A_T$ with an $R_+$-action by defining $r H = H r$ for all $r \in R_+$. The additivity of this functor ensures (M1) and (M2) of Definition~\ref{def:RPCoeffCat} are satisfied. (M3)-(M5) are satisfied directly from the definition of the $R_+$-action on $\mod*A_T$. (M6) follows from the fact that $H$ and $H'$ are fully faithful.
\end{proof}

\begin{thm} \label{thm:TiltedFolding}
	Let $T$ be a basic almost complete $\chebsr\ps{2n+1}$-tilting object of $\clus_F$ with complements $X_1$ and $X_2$. Let $A_1 \simeq KQ^{\gend_1} / (\rho_1)$ and $A_2 \simeq KQ^{\gend_2} / (\rho_2)$ be the cluster-$\chebsr\ps{2n+1}$-tilted algebras corresponding to the basic $\chebsr\ps{2n+1}$-tilting objects $T \oplus X_1$ and $T \oplus X_2$ respectively, where $\rho_1$ and $\rho_2$ are relations on the respective path algebras. Then
	\begin{enumerate}[label=(\alph*)]
		\item There exist $\wh\chebr\ps{2n+1}$-quivers $Q^{\gend'_1}$ and $Q^{\gend'_2}$ and respective weighted foldings $F_1 \colon Q^{\gend_1} \rightarrow Q^{\gend'_1}$ and $F_2 \colon Q^{\gend_2} \rightarrow Q^{\gend'_2}$.
		\item $Q^{\gend'_1}$ and $Q^{\gend'_2}$ differ by mutation at a single vertex.
	\end{enumerate}
\end{thm}
\begin{proof}
	In this proof, we use the fact that the objects of $\clus_F$ may be represented by the objects of $\calM_F$ and the shift of the indecomposable projectives of $\calM_F$. First consider the basic tilting object $\wh{T} \in \clus_F$ and an associated object $T \in \clus_F$ represented as
	\begin{equation*}
		\wh{T} \cong \bigoplus_{i \in Q_0^\gend} P(i) \qquad \text{and} \qquad T \cong \bigoplus_{i \in Q_0^{\gend} : \vw(i)=1} P(i).
	\end{equation*}
	In particular, $P(i) \in \Gamma$ for any $i$ such that $\vw(i)=1$, and there exist precisely $|Q_0^{\gend'}|$ such indecomposable projectives, so $T$ is in fact the $\chebsr\ps{2n+1}$-tilting object corresponding to $\wh{T}$. Under the equivalences
	\begin{equation*}
		H\colon\clus_F/ \add \cfart \wh{T} \rightarrow \mod*A_T \qquad \text{and} \qquad H'\colon\add \wh{T} \rightarrow \proj A_T,
	\end{equation*}
	one has $A_T \simeq KQ^\gend$, where $A_T$ is the cluster-$\chebsr\ps{2n+1}$-tilted algebra associated to $T$. Now consider the almost complete tilting object $\wh{T} / P(k)$ for some $k$ such that $\vw(k)=1$. Let $M_k$ be the complement distinct from $P(k)$, and let $X \in \Gamma$ be such that $M_k \in \mathcal{I}_X$. By \cite{BMRMutation}, Proposition~\ref{prop:TiltingComplements} and the equivalence $H'$, the quiver $Q'$ of cluster-tilted algebra $A_{\wh{T}'}$ differs from $Q^\gend$ by a mutation at vertex $k$, where $\wh{T}' \cong \wh{T} / P(k) \oplus M_k$.
	
	Next, consider an almost complete tilting object $\wh{T}' / P(j)$ for some $j \neq k$ such that $F(j)=F(k)$. Let $M_j$ be its complement distinct from $P(j)$ and write $\wh{T}'' \cong \wh{T}' / P(j) \oplus M_j$. Now $F(j)=F(k)$ if and only if $P(j) \cong r P(k)$ for some $r \in \chebsr\ps{2n+1}$. Thus, by Corollary~\ref{cor:ClosedComplements}, $M_j \in \mathcal{I}_X$. Similar to before, the quiver $Q''$ of cluster-tilted algebra $A_{\wh{T}''}$ differs from $Q'$ by a mutation at vertex $j$. But since $F$ is a folding and $F(k)=F(j)$, we have $\mu_j\mu_k(Q^\gend) =\mu_k\mu_j(Q^\gend)$.
	
	Proceeding through the above steps iteratively for all other vertices $j \in Q^\gend_0$ such that $F(j)=F(k)$, we obtain a tilting object
	\begin{equation*}
		\wh{T}_k = \left(\bigoplus_{i \in Q_0^\gend : F(i) \neq F(k)} P(i)\right) \oplus \left(\bigoplus_{j \in Q_0^\gend : F(j) = F(k)} M_j\right).
	\end{equation*}
	Now $P(k) \in \Gamma$ and $M_j \in \mathcal{I}_X$ for any $j$ such that $F(j)=F(k)$. So in particular, $\wh{T}_k$ corresponds to the $\chebsr\ps{2n+1}$-tilting object $T_k \cong T / P(k) \oplus X$. Moreover, the quiver $Q^k$ of the cluster-tilted algebra $A_{\wh{T}_k}$ differs from $Q^\gend$ by a well-defined composite mutation $\prod_{j \in Q_0^\gend : F(j) = F(k)} \mu_j$. But $F$ is a folding, which means that such a composite mutation of $Q^\gend$ must agree with a mutation of $Q^{\gend'}$ at the vertex $F(k)$. Thus, there must exist a folding $F'\colon Q^k \rightarrow \mu_{F(k)}(Q^{\gend'})$, which proves (a).
	
\end{proof}
\begin{cor}
	Let $F\colon Q^{\coxA_{2n}} \rightarrow Q^{\coxI_2(n+1)}$ be a weighted folding of quivers, and let $\clus_F$ be the corresponding cluster category. Let $T\in\clus_F$ be an $R_+$-tilting object, and let $A$ be the corresponding cluster-$R_+$-tilted algebra, which is Morita equivalent to $KQ/(\rho)$ for certain quiver $Q$.  Then either $Q = Q^{\coxA_{2n}}$ or $Q= (Q^{\coxA_{2n}})^{\op}$.
\end{cor}
\begin{proof}
	This is a consequence of the fact that $Q^{\coxI_2(n+1)}$ is a rank 2 quiver, and thus the only possible quiver mutations are sink/source mutations.
\end{proof}

\begin{exam}[Mutations of $\coxH_3$]
Consider the folding $F\colon Q^{\coxD_6} \rightarrow Q^{\coxH_3}$ given in Example~\ref{ex:D6H3Folding}. The Auslander-Reiten quiver of $\calM_F$ is given in Figure~\ref{fig:D6ARQuiver}, and the category has the structure of a $\goldsr$-coefficient category. Moreover, $\calM_F$ has a minimal set of $\goldsr$-generators that is unique up to isomorphic elements and closed under the Auslander-Reiten translation. Thus, $\calM_F$ has a reduced Auslander-Reiten quiver, which is given in the top of Figure~\ref{fig:ReducedAR}.

All of this carries over to the cluster category $\clus_F$. The Auslander-Reiten quiver of $\clus_F$ is similar to Figure~\ref{fig:D6ARQuiver}, but with additional objects
\begin{align*}
	\cfart\inv I(1) \cong &\sus P(1) \cong \cfart P(1),	&	\cfart\inv I(\ivx_1) \cong &\sus P(\ivx_1) \cong \cfart P(\ivx_1), \\
	\cfart\inv I(2) \cong &\sus P(2) \cong \cfart P(2),	&	\cfart\inv I(\ivx_2) \cong &\sus P(\ivx_2) \cong \cfart P(\ivx_2), \\
	\cfart\inv I(3) \cong &\sus P(3) \cong \cfart P(3), &	\cfart\inv I(\ivx_3) \cong &\sus P(\ivx_3) \cong \cfart P(\ivx_3).
\end{align*}
Likewise, the reduced Auslander-Reiten quiver of $\clus_F$ is similar to Figure~\ref{fig:ReducedAR}, but with additional objects
\begin{align*}
	\cfart\inv M_{(1,0,0)} \cong &\sus M_{(1,1,\gratio)} \cong \cfart M_{(1,1,\gratio)},	\\	\cfart\inv M_{(1,1,0)} \cong &\sus M_{(0,1,\gratio)} \cong \cfart M_{(0,1,\gratio)},	\\	\cfart\inv M_{(\gratio,\gratio,1)} \cong &\sus M_{(0,0,1)} \cong \cfart M_{(0,0,1)}.
\end{align*}
Figure~\ref{fig:H3Mutations} provides examples of mutations of both $\goldsr$-tilting objects of $\clus_F$ and of the corresponding quivers in the folding $F$. In fact, the cluster-$\goldsr$-tilted algebra $A_T \simeq KQ/(\rho)$ of the $\goldsr$-tilting object $T$ given in each rectangle is such that $Q$ is the quiver on the left-hand side of the folding. The category $\mod* A_T$ has the structure of a $\goldsr$-coefficient category by Theorem~\ref{thm:TiltedRPAction} and also has a reduced Aulsander-Reiten quiver. The indecomposable projective objects of $\mod* A_T$ are given by the direct summands of $T$ (black vertices in Figure~\ref{fig:H3Mutations}) and their $\gratio$-multiples. The reduced Auslander-Reiten quiver of $\mod* A_T$ can also be obtained from $\clus_F$ by removing all vertices and arrows associated to $\cfart T$. This is analogous to what is presented in \cite{BMRRT} for the Auslander-Reiten quiver of cluster-tilted algebras.

\begin{figure}[h]
	\centering
	\input{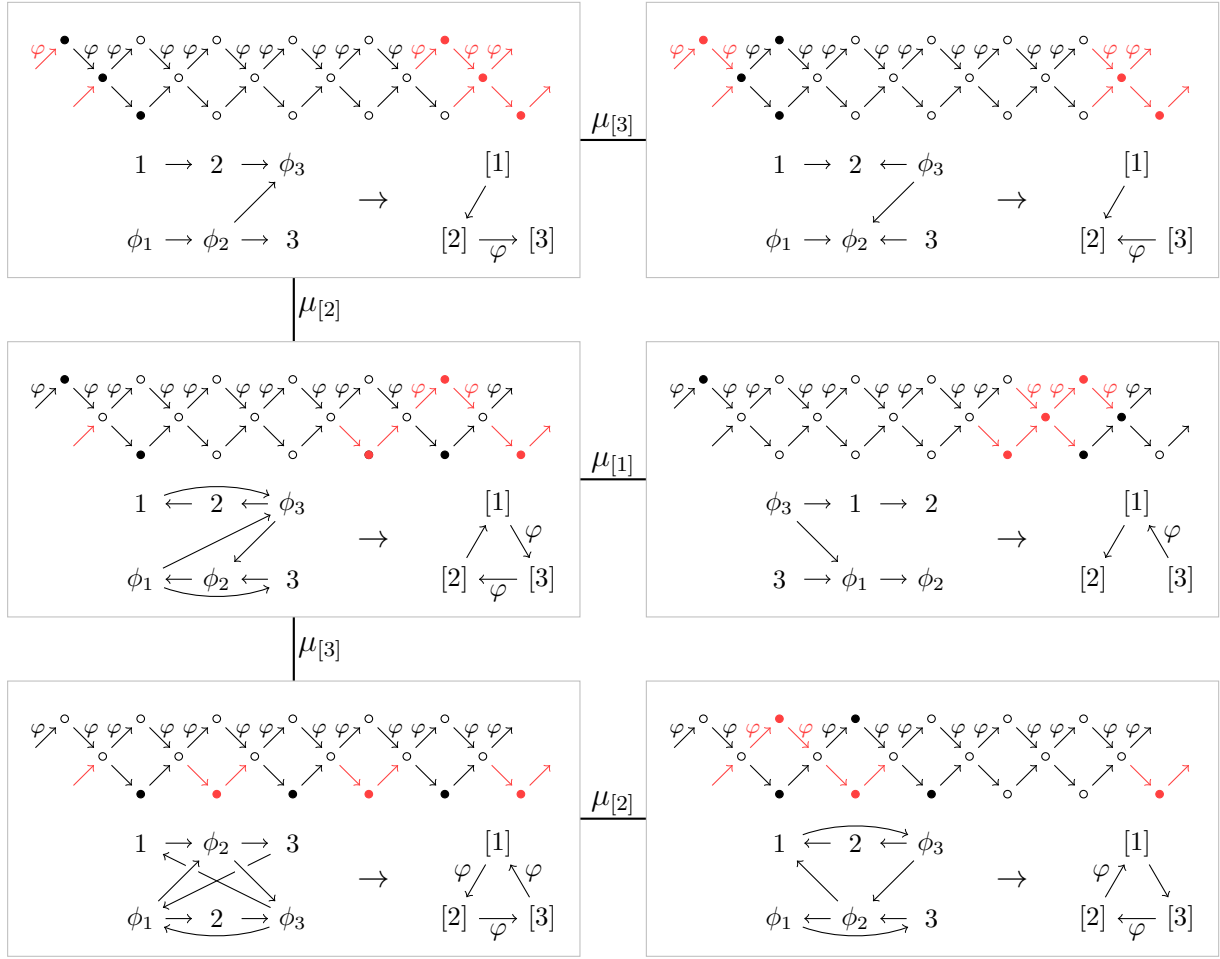}
	\caption{Examples of mutations of $\coxH_3$. Each rectangle contains a weighted folding $F \colon Q \rightarrow Q'$ of type $\coxH_3$ (bottom) and the reduced Auslander-Reiten quiver of  $\clus_F$ and $\calM_F$ (top). For the Auslander-Reiten quiver: black vertices correspond to an $R_+$-tilting object $T \in \clus_F$; red vertices correspond to $\cfart T$; arrows marked with $\gratio$ have valuation $(\gratio,\gratio)$; and the leftmost arrows with (seemingly) no source vertex are identified with the rightmost arrows with no target vertex. The reduced Auslander-Reiten quiver of $\calM_F$ is obtained by deleting all red vertices and arrows. For the folding: vertices labelled $\ivx_i$ have weight $\gratio$, and all other vertices and the unlabelled arrows have weight 1. All rectangles are related by mutation.} \label{fig:H3Mutations}
\end{figure}

\end{exam}

\section{$\gv$-vectors and $\cv$-vectors}
\label{section-cg}
There is a classical notion of $\gv$-vectors and $\cv$-vectors for integer exchange matrices introduced in~\cite{FominZelevinskyIV}, which play an important role in the theory of cluster algebras. In this section we adapt these notions to exchange matrices of types $H_4$, $H_3$ and $I_2(2n+1)$, and show that natural properties of $\cv$- and $\gv$-vectors continue to hold.

\subsection{Extended exchange matrices and tropical $y$-seeds}
Let $J$ be an index set and denote by $\mathbb{T}_J$ the $|J|$-regular tree with edges labelled by elements of $J$, such that every vertex is incident to edges with all labels distinct. Choose a vertex $t_0$ of $\mathbb{T}_J$.
Fix an $|J| \times |J|$ exchange matrix $B$ over $R$, and define a $2|J| \times |J|$ \emph{extended exchange matrix} over $R$
\begin{equation*}
	\wt{B}_{t_0} =
	\begin{pmatrix}
		B_{t_0} \\ C_{t_0}
	\end{pmatrix},
\end{equation*}
where $B_{t_0}=B$ and $C_{t_0}$ is an identity matrix. Define the {\em initial tropical $y$-seed} as $\{B_{t_0}; \ccv_{j,t_0} : j \in J\}$, where each $\ccv_{j,t_0}$ is the $j$-th column of the matrix $C_{t_0}$. The vectors $ \ccv_{j,t_0}$ are called {\em initial $\cv$-vectors}. We can now define a {\em tropical $y$-seed pattern} by assigning to every vertex $t$ of  $\mathbb{T}_J$ a tropical $y$-seed $\{B_{t}; \ccv_{j,t_0} : j \in J\}$, where $\{\ccv_{j,t_0}:j \in J\}$ are the column vectors of the matrix $C_t$, and for an edge $t \stackrel{k}{\text{-----}} t'$ the matrices
\begin{equation*}
	\wt{B}_{t}=	\begin{pmatrix}
		B_{t} \\ C_{t}
	\end{pmatrix}\qquad\text{and}\qquad\wt{B}_{t'}=	\begin{pmatrix}
		B_{t'} \\ C_{t'}
	\end{pmatrix}
\end{equation*}	
are related by a mutation at $k \in J$. The matrices $C_t$ are called \emph{$C$-matrices}, and the vectors $ \ccv_{j,t}$ are {\em $\cv$-vectors}. Abusing notation, we will refer to tropical $y$-seeds as {\em seeds} (as we have no other types of seeds in the paper). 

It was conjectured in~\cite{FominZelevinskyIV} and proved in~\cite{DWZ2,GHKK} that for $R=\integer$ the $\cv$-vectors are {\em sign-coherent}: for every $\cv$-vector all its components are either non-positive or non-negative. If $B$ defines an acyclic quiver, the sign-coherence is also implied by the following result~\cite{DWZ2,SpeyerThomas}: all $\cv$-vectors are roots of a certain root system constructed by $B$. Although the sign-coherence does not hold in general for $R$-quivers, a straightforward computation shows the following:
\begin{lem} \label{lem:Roots}
	Let $B$ be of type $H_4$, $H_3$ or $I_2(n)$. Then all $\cv$-vectors are roots of the corresponding root system.
\end{lem}

In particular, Lemma~\ref{lem:Roots} implies that all $\cv$-vectors for quivers of types  $H_4$, $H_3$ or $I_2(n)$ are sign-coherent. We will show later that Lemma~\ref{lem:Roots} can also be proved without any computations by applying results on the associated module category (see Corollary~\ref{thm:SignCoherent}). It was shown in~\cite{NZ} that, assuming sign-coherence of $\cv$-vectors, integer $\cv$- and $\gv$-vectors satisfy {\em tropical duality}: given $C$- and $G$-matrices $C_t$ and $G_t$, one has $G_t=(C_t^T)^{-1}$. This motivates the following definition.

\begin{defn} \label{def:G}
	Let $\wt{B}_t=(B_t,C_t)$ be an extended exchange matrix with columns indexed by a set $J$, where $B_{t_0}$ is an exchange matrix of type $H_4$, $H_3$ or $I_2(n)$. Define the corresponding {\em $G$-matrix} by
	\begin{equation*}
		G_t=(C_t^T)^{-1}.
	\end{equation*}
	We define {\em $\gv$-vectors} $\{\ggv_{j,t}:j \in J\}$ to be the column vectors of $G_t$.
\end{defn}

\begin{rem} \label{rem:MutG}
	Note that the definition~\ref{def:G} gives us a way to {\em mutate $G$-matrices}: given a $G$-matrix, we compute the corresponding $C$-matrix, mutate the seed, and then take the inverse transpose of the $C$-matrix again. The results of~\cite{NZ} actually show that mutating $G$-matrices in this way is the same as mutation given by the explicit formula in~\cite[(6.12)--(6.13)]{FominZelevinskyIV}.
\end{rem}
        

 
\subsection{Compatibility of $\cv$- and $\gv$-vectors with (un)folding}
Recall that given a weighted folding of quivers $F\colon Q^\gend \rightarrow Q^{\gend'}$ with $\gend' \in \{\coxH_{4},\coxH_{3},\coxI_2(2n+1)\}$, we have projection maps $\dimproj_F=d_F \dimvect$ and $\derdim_F$ that respectively map the objects in the categories $\calM_F$ and $\der_F$ to vectors in the root lattice of $\gend'$ (see Definition~\ref{def:ProjMap} for a recollection of the map $d_F$.) We will define a matrix generalisation of the map $d_F$ by
\begin{align*}
	\mathbf{d}_F\colon \integer^{|Q_0^\gend| \times |Q_0^\gend|} &\rightarrow \wh{\chebr}_{2n+1}^{|Q_0^{\gend'}| \times |Q_0^{\gend'}|} \\
	(\mathbf{v}_j)_{j \in Q_0^\gend} &\mapsto (d_F(\mathbf{v}_j))_{j \in Q_0^\gend : \vw(j)=1},
\end{align*}
where each $\mathbf{v}_j$ is a column vector.

Consider seed patterns for the exchange matrices $B$ and $B'$ corresponding to $Q^\gend$ and $Q^{\gend'}$. If $t$ is a vertex of $\mathbb{T}_{Q_0^{\gend'}}$ with assigned seed $\{B'_{t}; \ccv'_{[j],t}: [j] \in Q_0^{\gend'}\}$, we denote by $\wh{t}$ the vertex of $\mathbb{T}_{Q_0^{\gend}}$ with assigned seed $\{B_{\wh t}; \ccv_{j,\wh t}: j \in Q_0^{\gend}\}$ obtained as a result of the unfolding procedure. That is, if $B'_t$ is obtained from $B'_{t_0}$ by a mutation sequence $\mu_{[k_1]}\dots\mu_{[k_m]}$, then $B_{\wh t}$ is obtained from $B_{\wh{t}_0}$ by a mutation sequence $\wh\mu_{[k_1]}\ldots\wh\mu_{[k_m]}$, where $\wh\mu_{[k_i]}$ denotes composite mutation at each vertex $l \in Q_0^\gend$ such that $F(l)=[k_i]$. The main aim of this section is to establish the following compatibility result.

\begin{thm} \label{thm:Cube}
	Let $F\colon Q^\gend \rightarrow Q^{\gend'}$ be a weighted folding of quivers, with $\gend' = \{\coxH_{4},\coxH_{3}$, $\coxI_2(2n+1)\}$. Let $t_1$ and $t_2$ be vertices of $\mathbb{T}_{Q_0^{\gend'}}$ connected by an edge labelled $[k]$. Then the following diagram commutes:
		\begin{center}
			\begin{tikzpicture}[scale=1.5]
	\draw (0,0,0) node {$C_{\widehat{t}_1}$};
	\draw (2,0,0) node {$G_{\widehat{t}_1}$};
	\draw (0,-2,0) node {$C_{\widehat{t}_2}$};
	\draw (2,-2,0) node {$G_{\widehat{t}_2}$};
	
	\draw (0,0,-2) node {$C'_{t_1}$};
	\draw (2,0,-2) node {$G'_{t_1}$};
	\draw (0,-2,-2) node {$C'_{t_2}$};
	\draw (2,-2,-2) node {$G'_{t_2}$};
	
	\draw[|->, shorten <= 2.3ex, shorten >= 2.3ex] (0,0,-2) -- (2,0,-2);
	\draw[|->, shorten <= 2.3ex, shorten >= 2.3ex] (0,0,-2) -- (0,-2,-2);
	\draw[|->, shorten <= 2.3ex, shorten >= 2.3ex] (2,0,-2) -- (2,-2,-2);
	\draw[|->, shorten <= 2.3ex, shorten >= 2.3ex] (0,-2,-2) -- (2,-2,-2);
	
	\draw [white,fill=white] (-0.1,-0.72,-2) rectangle (0.1,-0.83,-2);
	\draw [white,fill=white](1.95,-1.15,0) rectangle (2.05,-1.3,0);
	
	\draw[|->, shorten <= 2ex, shorten >= 2ex] (0,0,0) -- (2,0,0);
	\draw[|->, shorten <= 2ex, shorten >= 2ex] (0,0,0) -- (0,-2,0);
	\draw[|->, shorten <= 2ex, shorten >= 2ex] (2,0,0) -- (2,-2,0);
	\draw[|->, shorten <= 2ex, shorten >= 2ex] (0,-2,0) -- (2,-2,0);
	
	\draw[|->, shorten <= 2ex, shorten >= 2ex] (0,0,0) -- (0,0,-2);
	\draw[|->, shorten <= 2ex, shorten >= 2ex] (2,0,0) -- (2,0,-2);
	\draw[|->, shorten <= 2ex, shorten >= 2ex] (0,-2,0) -- (0,-2,-2);
	\draw[|->, shorten <= 2ex, shorten >= 2ex] (2,-2,0) -- (2,-2,-2);
	
	\draw (1.3,0.15,0) node {\footnotesize$({-}^T)^{-1}$};
	\draw (1,0.15,-2) node {\footnotesize$({-}^T)^{-1}$};
	\draw (1,-1.85,0) node {\footnotesize$({-}^T)^{-1}$};
	\draw (0.75,-1.85,-2) node {\footnotesize$({-}^T)^{-1}$};
	
	\draw (0.22,-1.2,-2) node {\footnotesize$\mu_{[k]}$};
	\draw (2.22,-1.1,-2) node {\footnotesize$\mu_{[k]}$};
	\draw (0.22,-0.9,0) node {\footnotesize$\widehat{\mu}_{[k]}$};
	\draw (2.22,-0.8,0) node {\footnotesize$\widehat{\mu}_{[k]}$};
	
	\draw (-0.14,0.14,-1) node {\footnotesize$\mathbf{d}_F$};
	\draw (1.86,0.14,-1) node {\footnotesize$\mathbf{d}_F$};
	\draw (-0.14,-1.86,-1.1) node {\footnotesize$\mathbf{d}_F$};
	\draw (1.86,-1.86,-1.1) node {\footnotesize$\mathbf{d}_F$};
\end{tikzpicture}
		\end{center}
\end{thm}
In particular, Theorem~\ref{thm:Cube} gives an equivalent definition of $\cv$- and $\gv$-vectors for types $\{\coxH_{4},\coxH_{3},\coxI_2(2n+1)\}$: these can be defined by the projection of $\cv$- and $\gv$-vectors in the (weighted) unfolding via the map $d_F$. 


\begin{rem}
	The front and back faces of the cube follow from sign-coherence and \cite{NZ}, so we need only prove that the other four faces of the cube are commutative.
\end{rem}
	
The regular representation $\rho\colon \chebr\ps{2n+1} \rightarrow \integer^{n \times n}$ (Definition~\ref{def:RegRep}) and the surjective ring homomorphism $\sigma\ps{2n+1}\colon \chebr\ps{2n+1} \rightarrow \wh\chebr\ps{2n+1}$ (Remark~\ref{rem:ChebsrOrder}) will play a crucial role in the proof. Throughout, we use the ordering $\leq$ of $Q^\gend_0$ from  Proposition~\ref{prop:RegRepExchange}. This, in turn, determines a `nice' ordering of the rows and columns of the associated exchange matrix that allows us to prove a number of results on $C$-matrices, $G$-matrices.

\begin{lem}\label{lem:dFMatMap}
	Let $X=(x_{ij})_{i,j \in Q^{\gend}_0}$ be a matrix with entries in $\integer$, let $Y=(y_{[i][j]})_{[i],[j] \in Q^{\gend'}_0}$ be a matrix with entries in $\chebr\ps{2n+1}$, and let $Z=(z_{[i][j]})_{[i],[j] \in Q^{\gend'}_0}$ be a matrix with entries in $\wh\chebr\ps{2n+1}$. Consider $X$ as a block matrix $(X_{[i][j]})_{[i],[j] \in Q^{\gend'}_0}$, where $x_{kl}$ is an entry of $X_{[i][j]}$ if $F(k)=[i]$ and $F(l)=[j]$. Suppose that $X$, $Y$ and $Z$ are such that $X_{[i][j]}=\rho(y_{[i][j]})$ and $\sigma\ps{2n+1}(y_{[i][j]}) = z_{[i][j]}$. Then $\mathbf{d}_F(X)=Z$.
\end{lem}
\begin{proof}
	For each $[j] \in Q^{\gend'}_0$, consider the block column vector $\mathbf{X}_{[j]} = (X_{[i][j]})_{[i] \in Q^{\gend'}_0}$. By assumption, we have $\mathbf{X}_{[j]} = (\rho(y_{[i][j]}))_{[i] \in Q^{\gend'}_0}$ for some $y_{[i][j]} \in \chebr\ps{2n+1}$. Since $\rho$ is the regular representation of $\chebr\ps{2n+1}$, it is clear that each column of $\mathbf{X}_{[j]}$ is uniquely determined by $\croot_0$-th column of each matrix $\rho(y_{[i][j]})$. In particular, this implies that every column of $\mathbf{X}_{[j]}$ is uniquely determined by the column indexed by $l \in Q^\gend_0$ such that $F(l)=[j]$ and $\vw(l)=1$. Now given any $k \in Q^\gend_0$ such that $F(k)=[i]$, the ordering $\leq$ of $Q^\gend_0$ is such that the $k$-th row of $X$ coincides with the $\croot_{k'}$-th row of the matrix $\rho(y_{[i][j]})$ for some $0 \leq k' \leq n-1$. In particular, our setup is such that $\vw(k)=\sigma\ps{2n+1}(\croot_{k'})$. It is then a consequence of the definitions that we have $\mathbf{d}_F(X)=Z$.
\end{proof}

Given $\wh{t} \in \tree_{Q^\gend_0}$, recall that the matrix $(B_{\wh{t}})_{i,j \in Q_0^\gend}$ has block structure $(B_{\wh{t}})_{[i],[j] \in Q_0^{\gend'}}$, where the $([i],[j])$-th block consists of entries indexed by $(k,l)$ such that $F(k)=[i]$ and $F(l)=[j]$. This block structure can be extended to the matrices $C_{\wh{t}}$ in the natural way. This allows us to prove a result similar to Proposition~\ref{prop:RegRepExchange} for $C$-matrices.

\begin{prop}\label{prop:CBlockMutation}
	Let $t_1$ be a vertex of $\tree_{Q^{\gend'}_0}$, and consider the corresponding vertex $\wh{t}_1 \in \tree_{Q^{\gend}_0}$ related to $t_1$ by unfolding.
	\begin{enumerate}[label=(\alph*)]
		\item For any $[i],[j] \in Q_0^{\gend'}$, we have $C_{[i][j],\wh{t}_1}=\rho(r_{[i][j],\wh{t}_1})$ for some $r_{[i][j],\wh{t}_1} \in \Lambda \subset \chebr\ps{2n+1}$, where $\rho\colon \chebr\ps{2n+1} \rightarrow \integer^{n \times n}$ is the regular representation of $\chebr\ps{2n+1}$,
		\begin{equation*}
			\Lambda=\left\{\sum_{l=1}^{n-1}a_l\croot_l:a_0,\ldots,a_{n-1} \in \integer, \sgn(a_0)=\ldots=\sgn(a_{n-1})\right\},
		\end{equation*}
		and $\sgn(a)$ denotes the sign of $a \in \integer$. In particular, the entries of $C_{[i][j],\wh{t}_1}$ are sign-coherent.
		\item Let $\xymatrix@1{t_1 \ar@{-}[r]^-{[k]} & t_2}$ be an edge of $\tree_{Q^{\gend'}_0}$. Then the matrix $C_{\wh{t}_2} = \wh\mu_{[k]}(C_{\wh{t}_1})$ is such that
		\begin{equation*}
			C_{[i][j],\wh{t}_2} = 
			\begin{cases}
				- C_{[i][k],\wh{t}_1}														& \text{if } [j]=[k] \\
				C_{[i][j],\wh{t}_1} + \sgn(C_{[i][k],\wh{t}_1}) [C_{[i][k],\wh{t}_1} B_{[k][j],\wh{t}_1}]_+		& \text{otherwise,}
			\end{cases}
		\end{equation*}
		where $\sgn(C_{[i][k],\wh{t}_1})$ is the sign of the $([i],[k])$-th block of $C_{\wh{t}_1}$, and for any matrix $A=(a_{ij})$, we define $[A]_+$ as the matrix whose entries are given by $[a_{ij}]_+=\max(0,a_{ij})$.
	\end{enumerate}
\end{prop}
\begin{proof}
  We show first that (a) implies (b), and then prove (a).

  \noindent
	(a) $\Rightarrow$ (b): Suppose that $C_{\wh{t}_1}=(c_{ij,\wh{t}_1})_{i,j \in Q^\gend_0}$ satisfies (a). For any edge $\xymatrix@1{\wh{t}_1 \ar@{-}[r]^-{k} & t'}$ in $\tree_{Q^\gend_0}$, recall that by the classical mutation formula, the matrix $C_{t'} = \mu_k(C_{\wh{t}_1})$ is given by
	\begin{equation*}
		c_{ij, t'} = 
		\begin{cases}
			- c_{ik,\wh{t}_1}											& \text{if } j=k \\
			c_{ij,\wh{t}_1} + \sgn(c_{ik,\wh{t}_1}) [c_{ik,\wh{t}_1} b_{kj,\wh{t}_1}]_+		& \text{otherwise.}
		\end{cases}
	\end{equation*}
	Note that for all $i$, we have $c_{ij,t'} = c_{ij,\wh{t}_1}$ for any $j$ such that $j\neq k$ and $F(j)=F(k)$. This is because the entry $b_{kj,\wh{t}_1}$ belongs to the block $B_{[k][k],\wh{t}_1}$ which is a zero matrix due to the definition of unfolding. Thus, it immediately follows that the matrix $C_{\wh{t}_2}$ is such that $c_{ij, \wh{t}_2}=- c_{ij,\wh{t}_1}$ for any $j$ such that $F(j)=[k]$. Furthermore, one can see that by iteratively taking mutations we have
	\begin{equation*}
		c_{ij,\wh{t}_2} = c_{ij,\wh{t}_1}+ \sum_{F(l)=[k]} \sgn(c_{il,\wh{t}_1}) [c_{il,\wh{t}_1} b_{lj,\wh{t}_1}]_+
	\end{equation*}
	for any $j$ such that $F(j) \neq [k]$. Since $C_{[i][k],\wh{t}_1}$ is sign-coherent by (a), this reduces to
	\begin{align*}
		c_{ij,\wh{t}_2} &= c_{ij,\wh{t}_1}+ \sgn(C_{[i][k],\wh{t}_1})\sum_{F(l)=[k]} [c_{il,\wh{t}_1} b_{lj,\wh{t}_1}]_+ \\
			&=c_{ij,\wh{t}_1} + \sgn(C_{[i][k],\wh{t}_1}) ([C_{[i][k],\wh{t}_1}B_{[k][j],\wh{t}_1}]_+)_{ij},
	\end{align*}
	which shows that the block mutation formula in (b) follows from (a).
	
	(a): Firstly, one notes that it is a trivial consequence of Lemma~\ref{lem:ChebRegRep}(a) that if a block $C_{[i][j],\wh{t}_1}=\rho(r_{[i][j],\wh{t}_1})$ for some $r_{[i][j],\wh{t}_1} \in \Lambda$, then $C_{[i][j],\wh{t}_1}$ is sign-coherent. Thus, we will focus on proving that the blocks of all $C$-matrices are indeed expressible by such representations.
	
	It is immediate the lemma is true for $C_{\wh{t}_0}$, as we have $C_{[i][j],\wh{t}_0} \in \{\rho(0),\rho(1)\}$ for any $[i],[j] \in Q^{\gend'}_0$. So suppose for an induction argument that a matrix $C_{\wh{t}_1}$ satisfies (a). For each $[i],[j] \in Q^{\gend'}_0$, let $s_{[i][j],\wh{t}_1} \in \chebr\ps{2n+1}$ be such that $B_{[i][j],\wh{t}_1}=\rho(s_{[i][j],\wh{t}_1})$, which we may do due to Proposition~\ref{prop:RegRepExchange}. By the block mutation formula, the matrix $C_{\wh{t}_2} = \wh\mu_{[k]}(C_{\wh{t}_1})$ is such that
	\begin{equation*}
		C_{[i][k],\wh{t}_2} = - C_{[i][k],\wh{t}_1} = -\rho(r_{[i][k],\wh{t}_1}) = \rho(-r_{[i][k],\wh{t}_1})
	\end{equation*}
	for any $[i] \in Q^{\gend'}_0$. For $[j] \neq [k]$, we have
	\begin{align*}
		C_{[i][j],\wh{t}_2} &= \rho(r_{[i][j],\wh{t}_1}) + \sgn(\rho(r_{[i][k],\wh{t}_1}))[\rho(r_{[i][k],\wh{t}_1})\rho(s_{[k][j],\wh{t}_1})]_+ \\
		&=\rho(r_{[i][j],\wh{t}_1}) + \sgn(r_{[i][k],\wh{t}_1})[\rho(r_{[i][k],\wh{t}_1}s_{[k][j],\wh{t}_1})]_+.
	\end{align*}
	By Lemma~\ref{lem:RegRepExchange}, we may reduce this further to
	\begin{align*}
		C_{[i][j],\wh{t}_2} &= \rho(r_{[i][j],\wh{t}_1}) + \rho(\sgn(r_{[i][k],\wh{t}_1})[r_{[i][k],\wh{t}_1}s_{[k][j],\wh{t}_1}]_+) \\
		&= \rho(r_{[i][j],\wh{t}_1} + \sgn(r_{[i][k],\wh{t}_1})[r_{[i][k],\wh{t}_1}s_{[k][j],\wh{t}_1}]_+).
	\end{align*}
	So each block of $C_{\wh{t}_2}$ is expressible as the matrix regular representation of some value in $\chebr\ps{2n+1}$. Now we must show that each block is specifically the matrix representation of a value in the subset $\Lambda$. Indeed, if there exists a block $C_{[i][j],\wh{t}_2} = \rho(r_{[i][j],\wh{t}_2})$ with $r_{[i][j],\wh{t}_2} \not \in \Lambda$, the column of $\rho(r_{[i][j],\wh{t}_2})$ representing multiplication by $\croot_0=1$ would not be sign-coherent. In particular, this would imply that $\cv$-vectors for $Q^\gend$ are not sign-coherent, which is false. This proves the matrix $C_{\wh{t}_2}$ satisfies (a). Thus every $C$-matrix obtained by composite mutations associated to vertices in $Q^{\gend'}_0$ satisfies (a) by induction.
\end{proof}

A number of results follow as a consequence of the above proposition. The first important consequence is the commutativity of the left face of the cube in the compatibility theorem.

\begin{cor}
	For any vertex $t \in \tree_{Q^{\gend'}_0}$, we have $\mathbf{d}_F(C_{\wh{t}})=C'_t$. In particular, the left square of Theorem~\ref{thm:Cube} commutes, and $\cv$-vectors are sign-coherent.
\end{cor}
\begin{proof}
	It is clear from Lemma~\ref{lem:dFMatMap} that $\mathbf{d}_F(C_{\wh{t}_0})=C'_{t_0}$, as these are both identity matrices. That $\mathbf{d}_F(C_{\wh{t}})=C'_t$ for each $t \in \tree_{Q^{\gend'}_0}$ follows by induction on mutation. So suppose there exists a vertex $t_1 \in \tree_{Q^{\gend'}_0}$ such that $\mathbf{d}_F(C_{\wh{t}_1})=C'_{t_1}=(c'_{[i][j],t_1})$. In particular, this implies by Lemma~\ref{lem:dFMatMap} and Proposition~\ref{prop:CBlockMutation}(a) that for each $[i],[j] \in Q^{\gend'}_0$ we have $C_{[i][j]}=\rho(r_{[i][j],\wh{t}_1})$ with $\sigma\ps{2n+1}(r_{[i][j],\wh{t}_1})=c'_{[i][j],t_1}$. Write $B'_{t_1}= (b'_{[i][j],t_1})$ and let $s_{[i][j],\wh{t}_1} \in \{0,\pm 1,\pm \croot_1\} \subset \chebr\ps{2n+1}$ be such that  $\sigma\ps{2n+1}(s_{[i][j],\wh{t}_1})=b'_{[i][j],t_1}$, which we may do due to Proposition~\ref{prop:RegRepExchange}.
	
	Now we shall compare mutation formulas. Let $\xymatrix@1{t_1 \ar@{-}[r]^-{[k]} & t_2} \in \tree_{Q^{\gend'}_0}$. Then
	\begin{align*}
		C_{[i][j],\wh{t}_2} &= 
		\begin{cases}
			\rho(-r_{[i][k],\wh{t}_1})													& \text{if } [j]=[k] \\
			\rho(r_{[i][j],\wh{t}_1} + \sgn(r_{[i][k],\wh{t}_1})[r_{[i][k],\wh{t}_1}s_{[k][j],\wh{t}_1}]_+)	& \text{otherwise;}
		\end{cases} \\
		c'_{[i][j],t_2}& = 
		\begin{cases}
			- c'_{[i][k],t_1}												& \text{if } [j]=[k] \\
			c'_{[i][j],t_1} + \sgn(c'_{[i][k],t_1}) [c'_{[i][k],t_1} b'_{[k][j],t_1}]_+	& \text{otherwise;}
		\end{cases} \\
		&=
		\begin{cases}
			\sigma\ps{2n+1}(-r_{[i][k],\wh{t}_1})												& \text{if } [j]=[k] \\
			\sigma\ps{2n+1}(r_{[i][j],\wh{t}_1} + \sgn(r_{[i][k],\wh{t}_1}) [r_{[i][k],\wh{t}_1} s_{[k][j],\wh{t}_1}]_+)	& \text{otherwise.}
		\end{cases}
	\end{align*}
	Thus, it is clear from Lemma~\ref{lem:dFMatMap} that we have $\mathbf{d}_F(C_{\wh{t}_2})=C'_{t_2}$. So $\mathbf{d}_F(C_{\wh{t}})=C'_t$ for each $t \in \tree_{Q^{\gend'}_0}$ by induction. It automatically follows from this that $\mathbf{d}_F\wh\mu_{[k]}(C_{\wh{t}}) =\mu_{[k]}\mathbf{d}_F(C_{\wh{t}})$ for each vertex $t \in \tree_{Q^{\gend'}_0}$, as required.
\end{proof}

\begin{cor} \label{cor:CBlockProps}
	For any vertex $t \in \tree_{Q^{\gend'}_0}$, the blocks of $C_{\wh{t}}$ commute. That is $C_{[i][j],\wh{t}}C_{[k][l],\wh{t}}=C_{[k][l],\wh{t}}C_{[i][j],\wh{t}}$ for any $[i],[j],[k],[l] \in Q^{\gend'}_0$.
\end{cor}
\begin{proof}
	This follows from the fact that the blocks of $C_{\wh{t}}$ arise from the regular representation of $\chebr\ps{2n+1}$, which is a commutative ring.
\end{proof}

The next consequence of Proposition~\ref{prop:CBlockMutation} can be verified explicitly by computation. Nevertheless, we provide a proof using the module category of an unfolded quiver.

\begin{cor} \label{thm:SignCoherent}
	The $\cv$-vectors of $Q^{\gend'}$ are roots of $\gend'$.
\end{cor}
\begin{proof}
	By the results of \cite{DWZ2,GHKK}, the $\cv$-vectors of $Q^\gend$ are roots of $\gend$. By Gabriel's Theorem, the $\cv$-vectors of $Q^\gend$ are therefore the $\pm 1$ multiples of dimension vectors of some indecomposable $KQ^\gend$-modules. Now recall that $KQ^\gend$ has the structure of a $\chebsr\ps{2n+1}$-coefficient category $\calM_F$. So by Theorems~\ref{thm:Folding} and \ref{thm:Action}, for any indecomposable $KQ^\gend$-module $M$, there exists $0 \leq k \leq n-1$ and a positive root $\alpha$ of $\gend'$ such that $M \cong \croot_k M_\alpha$. In particular, $\dimproj_F(M) = \srhom\ps{2n+1}(\croot_k) \alpha$.
	
	Let $t$ be a vertex of $\tree_{Q^{\gend'}_0}$. Suppose for a contradiction that there exists a column $\ccv_{j,\wh{t}}$ of $C_{\wh{t}}$ indexed by $j \in Q^\gend_0$ with $\vw(j)=1$ such that $\ccv'_{[j],t}= d_F(\ccv_{j,\wh{t}})$ is not a root of $\gend'$. Let $M$ be the indecomposable $KQ^\gend$-module with dimension vector $\pm\ccv_{j,\wh{t}}$. Then $M \cong \croot_k M_\alpha$ for some $k>0$ and positive root $\alpha$ of $\gend'$. Now by Proposition~\ref{prop:CBlockMutation}(a), for any $i \in Q^\gend_0$, the $[i][j]$-th block of $C_{\wh{t}}$ is the matrix regular representation of some $r_{[i][j],\wh{t}} \in \chebr\ps{2n+1}$. Consequently, for any other column $\ccv_{j',\wh{t}}$ of $C$ indexed by $j' \in Q^\gend_0$ such that $\vw(j') \neq 1$ and $F(j')=[j]$, we have $d_F(\ccv_{j',\wh{t}}) = \sigma\ps{2n+1}(\croot_l) d_F(\ccv_{j,\wh{t}}) = \pm\sigma\ps{2n+1}(\croot_k\croot_l) \alpha$ with $l \neq 0$. But by the product rule of $\chebr\ps{2n+1}$ (Definition~\ref{def:ChebSemiring}) and the axioms of a $\chebsr\ps{2n+1}$-coefficient category (Definition~\ref{def:RPCoeffCat}), this is the $F$-projected dimension vector of the decomposable object $\croot_k\croot_l M_\alpha \in \calM_F$. Moreover, there exists no indecomposable $KQ^\gend$-module $M'$ with $\dimproj_F(M') = \sigma\ps{2n+1}(\croot_k\croot_l) \alpha$ by Theorem~\ref{thm:Folding} and the fact that $\croot_0,\ldots,\croot_{n-1}$ are linearly independent (Remark~\ref{rem:ChebBasis}). But then $\ccv_{j',\wh{t}}$ is not the $\pm1$ multiple of a dimension vector of an indecomposable $KQ^\gend$-module --- a contradiction. Thus, every column $\ccv_{j,\wh{t}}$ of $C_{\wh{t}}$ with $\vw(j)=1$ is such that $d_F(\ccv_{j,\wh{t}})$ is a root of $\gend'$, and hence every column of $C'_t =\mathbf{d}_F(C_{\wh{t}})$ is a root of $\gend'$, as required. 
\end{proof}

The commutativity of the top and bottom squares also follows from Proposition~\ref{prop:CBlockMutation}, but first we must prove a technical result on the determinants of $C$-matrices over the ring $\chebr\ps{2n+1}$, which for integer $C$-matrices is already well-known. Henceforth, given any matrix $A$, we denote its determinant by $|A|$.

\begin{cor} \label{cor:CDeterminant}
	Let $t$ be a vertex of $\tree_{Q^{\gend'}_0}$. Let $X_{\wh{t}}=(r_{[i][j],\wh{t}})$ be the matrix with entries in $\chebr\ps{2n+1}$ such that $\rho(r_{[i][j],\wh{t}})=C_{[i][j],\wh{t}}$ and $\sigma\ps{2n+1}(r_{[i][j],\wh{t}})=c'_{[i][j],t}$. Then 
	\begin{enumerate}[label=(\alph*)]
		\item $|C'_{t}|= \pm 1$,
		\item $|X_{\wh{t}}| = \pm 1$ and $\sigma\ps{2n+1}(|X_{\wh{t}}|)=|C'_t|$.
	\end{enumerate}
\end{cor}
\begin{proof}
	(a) Fix indices $[j]$ and $[k]$. Since $\cv$-vectors of $Q^{\gend'}$ are sign-coherent, one sees that the mutation formula for $\mu_{[k]}(C'_t)$ is given by adding some $\wh{\chebr}_{2n+1}$-multiple of column $\ccv'_{[k],t}$ of $C'_t$ to some other columns, and then multiplying $\ccv'_{[k],t}$ by $-1$. By the rules governing determinants, this gives us $|C'_t|=-|\mu_{[k]}(C'_t)|$. That this quantity is $\pm 1$ follows from the fact that $C'_{t_0}$ is the identity matrix.
	
	(b) By the fact that $\sigma\ps{2n+1}$ is a homomorphism mapping positive elements to positive elements, $X_{\wh{t}}$ is column sign-coherent. Given an edge $\xymatrix@1{t \ar@{-}[r]^-{[k]} & t'} \in \tree_{Q^{\gend'}_0}$, applying the mutation formula to $X_{\wh{t}}$ similarly yields $|X_{\wh{t}}|=-|X_{\wh{t}'}|$, and $X_{\wh{t}'}=(r_{[i][j],\wh{t}'})$ is the matrix such that $\rho(r_{[i][j],\wh{t}'})=C_{[i][j],\wh{t}'}$ and $\sigma\ps{2n+1}(r_{[i][j],\wh{t}'})=c'_{[i][j],t'}$. Again, that this quantity is $\pm 1$ is a result of the initial matrix being the identity matrix. That $\sigma\ps{2n+1}(|X_{\wh{t}}|)=|C'_t|$ follows from the fact that $\sigma\ps{2n+1}(|X_{\wh{t}_0}|)=|C'_{t_0}|=1$.
\end{proof}

\begin{prop}\label{prop:TBSquares}
	For any vertex $t \in \tree_{Q^{\gend'}_0}$, we have $\mathbf{d}_F(G_{\wh{t}}) = G_t$. In particular, the top, bottom and right squares of Theorem~\ref{thm:Cube} commute.
\end{prop}
\begin{proof}
	The proof will not require mutation, and so for the purpose of readability, write $C=(C_{[i][j]})=C_{\wh{t}}$ and $C'=\mathbf{d}_F(C)=(c'_{[i][j]})=C'_t$. Consider the corresponding matrices $G=G_{\wh{t}}=(C^T)\inv$ and $G'=G'_{t}=({C'}^T)\inv$. We can describe the structure of these $G$-matrices explicitly in terms of the $C$-matrices. The matrix $C'$ is at most $4 \times 4$ with $|C'|= \pm 1$, and so it is easily verified by computation that the matrix $G'=(g'_{[i][j]})$ is such that $g'_{[i][j]} \in \integer[V']$ with $V' = \{c'_{[k][l]}:[k],[l] \in Q^{\gend'}_0\}$. 
	
	Now define $V = \{C_{[k][l]}:[k],[l] \in Q^{\gend'}_0\}$ and define a new matrix $\wt{G}$ with blocks indexed by $Q^{\gend'}_0$ such that $\wt{G}_{[i][j]} \in \integer[V]$ is the polynomial given by replacing each term $c'_{[k][l]}$ in $g'_{[i][j]}$ by the block matrix $C_{[k][l]}$. This is well-defined because the blocks of $C$ commute. It is easy to see that $\mathbf{d}_F(\wt{G})$ is the matrix given by substituting each block $C_{[k][l]}= \rho(r_{[k][l]})$ with entry $c'_{[k][l]}=\sigma\ps{2n+1}(r_{[k][l]})$. Hence, we have defined $\wt{G}$ in such a way that we have $\mathbf{d}_F(\wt{G})=G'$.
	
	We claim that we actually have $G = \wt{G}$. That is, we shall verify that $\wt{G} = (C^T)\inv$. To check this, we note that each block $C_{[i][j]}$ is symmetric by Lemma~\ref{lem:ChebRegRep}(b). Thus, the matrix $C^T$ has block structure $C^T_{[i][j]} = C_{[j][i]}$.
	
	For $\coxH$-type foldings, we have $\chebr\ps{5} \cong \wh\chebr\ps{5} \cong \goldint$, so we can abuse notation and dispense with the homomorphism $\sigma\ps{2n+1}$ and write $C_{[i][j]}=\rho(c'_{[i][j]})$ for each $[i],[j] \in Q^{\gend'}_0$. Thus,
	\begin{align*}
		(C^T \wt{G})_{[i][j]} &= \sum_{[k] \in Q^{\gend'}_0} C_{[k][i]}\wt{G}_{[k][j]} \\
		&= \sum_{[k] \in Q^{\gend'}_0} \rho(c'_{[k][i]}g'_{[k][j]}).
	\end{align*}
	But since $G'=({C'}^T)\inv$, we have $c'_{[k][i]}g'_{[k][j]} = 1$ whenever $[i]=[j]$ and $c'_{[k][i]}g'_{[k][j]} = 0$ otherwise. It therefore follows that $(C^T \wt{G})_{[i][i]}=\id_n$ and $(C^T \wt{G})_{[i][j]}=0$ for any $[i] \neq [j]$. So $\wt{G} = G$, as required.
	
	On the other hand, if $\gend' = \coxI_2(2n+1)$, then $C'$ is a $2 \times 2$ matrix and so we specifically have
	\begin{equation*}
		G' = |C'|
		\begin{pmatrix}
			c'_{[1][1]} 	& -c'_{[1][0]} \\
			-c'_{[0][1]} 	& c'_{[0][0]}
		\end{pmatrix},
	\end{equation*}
	where $|C'|=|C'|\inv=\pm 1$ by Corollary~\ref{cor:CDeterminant}(a). Let $X=(r_{[i][j]})$ be the matrix with entries in $\chebr\ps{2n+1}$ such that $\rho(r_{[i][j]})=C_{[i][j]}$ and $\sigma\ps{2n+1}(r_{[i][j]})=c'_{[i][j]}$. Then we have
	\begin{equation*}
		\wt{G} = 
		\begin{pmatrix}
			\rho(|X|r_{[1][1]}) 	& \rho(-|X|r_{[1][0]}) \\
			\rho(-|X|r_{[0][1]}) 	& \rho(|X|r_{[0][0]})
		\end{pmatrix},
	\end{equation*}
	where $|X|=|X|\inv=|C'|=\pm 1$ by Corollary~\ref{cor:CDeterminant}(b). So we have
	\begin{align*}
		C^T \wt{G} &= 
		\begin{pmatrix}
			\rho(r_{[0][0]}) 	& \rho(r_{[1][0]}) \\
			\rho(r_{[0][1]}) 	& \rho(r_{[1][1]})
		\end{pmatrix}
		\begin{pmatrix}
			\rho(|X|r_{[1][1]}) 	& \rho(-|X|r_{[1][0]}) \\
			\rho(-|X|r_{[0][1]}) 	& \rho(|X|r_{[0][0]})
		\end{pmatrix}
		\\ &=
		\begin{pmatrix}
			\rho(|X| (r_{[0][0]}r_{[1][1]} - r_{[0][1]}r_{[1][0]}))  	& \rho(0) \\
			\rho(0) 	& \rho(|X| (r_{[0][0]}r_{[1][1]} - r_{[0][1]}r_{[1][0]}))
		\end{pmatrix}
		\\ &=
		\begin{pmatrix}
			\rho(|X|^2)  	& \rho(0) \\
			\rho(0) 	& \rho(|X|^2)
		\end{pmatrix}
		=
		\begin{pmatrix}
			\rho(1)  	& \rho(0) \\
			\rho(0) 	& \rho(1)
		\end{pmatrix}
		=\id_{2n},
	\end{align*}
	as required. So we indeed have $G=\wt{G}$ in all cases. Hence, $\mathbf{d}_F(G_{\wh{t}}) = G_t$ for each vertex $t \in \tree_{Q^{\gend'}_0}$. We have also shown that the top and bottom squares of Theorem~\ref{thm:Cube} commute. That the right square also commutes follows from the fact that the map $({-}^T)\inv$ is involutive and that that all other squares of the cube commute.
\end{proof}

\subsection{Categorification of the $G$-matrices of $Q^{\gend'}$}
\label{section-cat-g}
One has a categorification of $\gv$-vectors of $Q^\gend$ due to \cite{DehyKeller} and \cite{Plamondon}. We will show how folding and the $\gv$-vectors of $Q^{\gend'}$ fits into this categorification.

Consider the full subcategory $\mathcal{P} \subset \clus_F$ consisting of objects that belong to classes $P(i) \in \der_F$, where $P(i)$ is the complex of the projective object corresponding to $i \in Q^\gend_0$ concentrated in degree 0. For any object $M \in \clus_F$, we have a triangle
\begin{equation*}
	P_1 \rightarrow P_0 \rightarrow M \rightarrow \sus P_1
\end{equation*}
where $P_0,P_1 \in \mathcal{P}$. In particular, we may write
\begin{equation*}
	P_0 = \bigoplus_{i \in Q^\gend_0} a_i P(i) \qquad \text{and} \qquad P_1 = \bigoplus_{i \in Q^\gend_0} b_i P(i),
\end{equation*}
where each $a_i ,b_i\in \nnint$. Alternatively, we may write
\begin{equation*}
	P_0 = \bigoplus_{[i] \in Q^{\gend'_0}} r_{[i]} P(j_{[i]}) \qquad \text{and} \qquad P_1 = \bigoplus_{[i] \in Q^{\gend'_0}} s_{[i]} P(j_{[i]}),
\end{equation*}
where each $r_{[i]},s_{[i]} \in \chebsr\ps{2n+1}$ and each $j_{[i]} \in Q^\gend_0$ is such that $F(j_{[i]})=[i]$ and $\vw(j_{[i]})=1$. In fact, this is equivalent to writing
\begin{equation*}
	P_0 = \bigoplus_{M \in \Gamma_\mathcal{P}} r_M M \qquad \text{and} \qquad P_1 = \bigoplus_{M \in \Gamma_\mathcal{P}} r_M M,
\end{equation*}
where the set $\Gamma_\mathcal{P}$ consists of the iso-classes of objects of $\mathcal{P}$ that belong to the minimal set of $\chebsr\ps{2n+1}$-generators of $\clus_F$, and where each $r_M,s_M \in \chebsr\ps{2n+1}$.

\begin{defn}
	For any object $M \in \clus_F$, we define two $\gv$-vectors associated to $M$ by
	\begin{equation*}
		\ggv_\gend^M = (a_i -b_i)_{i \in Q^\gend_0}
		\qquad \text{and} \qquad 
		\ggv_{\gend'}^M = (r_{[i]} -s_{[i]})_{[i] \in Q^{\gend'}_0}.
	\end{equation*}
\end{defn}

It is in fact easy to see that $d_F(\ggv_\gend^M)=\ggv_{\gend'}^M$ and that $\ggv_{\gend'}^{rM} = r \ggv_{\gend'}$. Now turning our attention to $G$-matrices, this is where $\chebsr\ps{2n+1}$-tilting objects come in.

\begin{defn}
	Let $T = \bigoplus_{[i] \in Q^{\gend'}_0} T_{[i]} \in \clus_F$ be a $\chebsr\ps{2n+1}$-tilting object and let $\wh{T}  = \bigoplus_{[i] \in Q^{\gend'}_0} \bigoplus_{Z \in \mathcal{I}_{T_{[i]}}} Z \in \clus_F$ be the corresponding tilting object. We define a matrix associated to $\wh{T}$ by
	\begin{equation*}
		G_{\widehat{T}}= (\ggv^{Z}_\gend)_{Z \in \mathcal{I}_{T_{[i]}}, [i] \in Q^{\gend'}_0}
	\end{equation*}
	and define a matrix associated to $T$ to by
	\begin{equation*}
		G'_{T} = (\ggv^{T_{[i]}}_{\gend'})_{[i] \in Q^{\gend'}_0}
	\end{equation*}
\end{defn}

Classically, we know that $G_{\widehat{T}}$ is indeed a $G$-matrix of $Q^\gend$. However by construction, we also have that the columns of $G_{\widehat{T}}$ consist of $\gv$-vectors $\ggv_\gend^{\croot_j T_{[i]}}$ for each $[i] \in Q^{\gend'}_0$ and $0 \leq j \leq n-1$. Since the blocks of $G$-matrices obtained by composite mutations are matrix regular representations of $\chebr\ps{2n+1}$, this is precisely a $G$-matrix $G_{\widehat{t}}$ for some vertex $t \in \tree_{Q^{\gend'}_0}$. Moreover, we can see that $\mathbf{d}_F(G_{\widehat{T}})=G'_{T}$, which is a $G$-matrix of $Q^{\gend'}$ by the compatibility theorem (Theorem~\ref{thm:Cube}). In terms of mutation, we know that the $G$-matrix of a tilting object in $\clus_F$ corresponds to a classical mutation of the $G$-matrix. Thus, by Theorem~\ref{thm:Cube}, we can see that the mutation of a $\chebsr\ps{2n+1}$-tilting object $T$ corresponds to the mutation of its $G$-matrix. Thus, we have the following.

\begin{cor}
	Let $F\colon Q^\gend \rightarrow Q^{\gend'}$ be a weighted folding of type $\coxH_4$, $\coxH_3$ or $\coxI_2(2n+1)$ and let $\clus_F$ be the associated cluster category. Let $T_1$ and $T_2$ be basic $\chebsr\ps{2n+1}$-tilting objects that differ by a single indecomposable direct summand. Then $G'_{T_1}$ and $G'_{T_2}$ are $G$-matrices of $Q^{\gend'}$ that differ by mutation at a single vertex in $Q^{\gend'}$.
\end{cor}

	\bibliography{Bibliography}{}
	\bibliographystyle{habbrv}
\end{document}